\documentclass[aos, preprint]{imsart}

\RequirePackage[utf8]{inputenc}
\RequirePackage{amsthm,amsmath,amsfonts,amssymb,mathtools}
\RequirePackage[authoryear]{natbib}
\RequirePackage{enumitem}
\RequirePackage{algorithm}
\RequirePackage{algpseudocode}
\RequirePackage{booktabs,xcolor}
\RequirePackage[normalem]{ulem}
\RequirePackage[colorlinks]{hyperref}
\usepackage[noabbrev,capitalize]{cleveref}
\usepackage{longtable}
\usepackage{graphicx}
\usepackage{tikz}
\usetikzlibrary{arrows.meta}

\hypersetup{
  colorlinks=true,
  citecolor=blue,
  linkcolor=blue,
  urlcolor=blue,
  filecolor=blue,
}

\startlocaldefs
\numberwithin{equation}{section}
\newif\if@nosecdot@once
\let\ims@setaftersec@dot\setaftersec@dot
\def\setaftersec@dot#1{%
  \if@nosecdot@once
    \def\aftersec@dot{}\global\@nosecdot@oncefalse
  \else
    \ims@setaftersec@dot{#1}%
  \fi}
\AtBeginDocument{%
  \let\ims@paragraph\paragraph
  \def\paragraph{\global\@nosecdot@oncetrue\ims@paragraph}}

\definecolor{darkred}{RGB}{139,0,0}
\definecolor{darkgreen}{RGB}{0,100,0}
\colorlet{RED}{red}

\DeclareRobustCommand{\strikeddarkred}[1]{%
  \ifmmode
    \text{\color{darkred}\sout{\ensuremath{\displaystyle #1}}}%
  \else
    {\color{darkred}\sout{#1}}%
  \fi
}

\theoremstyle{plain}
\newtheorem{lemma}{Lemma}
\newtheorem{proposition}{Proposition}
\newtheorem{corollary}{Corollary}
\newtheorem{theorem}{Theorem}
\theoremstyle{remark}
\newtheorem{definition}{Definition}
\newtheorem{assumption}{Assumption}
\newtheorem{remark}{Remark}

\crefname{assumption}{Assumption}{Assumptions}
\Crefname{assumption}{Assumption}{Assumptions}
\crefname{algorithm}{Algorithm}{Algorithms}
\Crefname{algorithm}{Algorithm}{Algorithms}

\g@addto@macro\appendix{%
  \counterwithin{theorem}{section}%
  \counterwithin{lemma}{section}%
  \counterwithin{proposition}{section}%
  \counterwithin{corollary}{section}%
  \counterwithin{definition}{section}%
  \counterwithin{assumption}{section}%
  \counterwithin{remark}{section}%
  \counterwithin{figure}{section}%
  \counterwithin{table}{section}%
  \counterwithin{algorithm}{section}%
}

\newcommand{\cX}{\mathcal{X}}
\newcommand{\bR}{\mathbb{R}}
\newcommand{\bOne}{\mathbf{1}}
\newcommand{\Ham}{d_{\mathrm{H}}}
\newcommand{\tok}{\mathrm{tok}}
\newcommand{\Id}{\mathbf{I}}
\newcommand{\defn}{\coloneqq}
\newcommand{\pa}{\operatorname{pa}}

\newcommand{\cV}{\mathcal{V}}
\newcommand{\sD}{{\mathsf{D}}}

\newcommand{\x}{\mathbf{x}}
\newcommand{\y}{\mathbf{y}}
\newcommand{\z}{\mathbf{z}}

\newcommand{\Fd}{\binom{[\sD]}{d}}
\newcommand{\sm}{{\mathsf{m}}}
\newcommand{\sH}{{\mathsf{H}}}
\newcommand{\MRF}{\mathrm{MRF}}

\newcommand{\N}{\mathbb{N}}
\newcommand{\sL}{{\mathsf{L}}}
\newcommand{\sW}{{\mathsf{W}}}
\newcommand{\sS}{{\mathsf{S}}}
\newcommand{\sB}{{\mathsf{B}}}
\newcommand{\sN}{{\mathsf{N}}}
\newcommand{\Sws}{{\mathsf{S}_{\mathrm{ws}}}}
\newcommand{\btheta}{\boldsymbol{\theta}}
\newcommand{\bTheta}{\boldsymbol{\Theta}}
\newcommand{\bv}{\mathbf{v}}
\newcommand{\W}{\mathbf{W}}
\newcommand{\bb}{\mathbf{b}}

\newcommand{\Lw}{L_{\omega}}

\newcommand{\bE}{\mathbb{E}}
\newcommand{\cL}{\mathcal{L}}
\newcommand{\bC}{\mathbf{C}}
\DeclareMathOperator*{\argmin}{arg\,min}
\newcommand{\sA}{{\mathsf{A}}}
\newcommand{\R}{\mathbb{R}}
\newcommand{\e}{\mathbf{e}}
\newcommand{\onehot}{\mathbf{o}}
\newcommand{\se}{{\mathrm{se}}}
\newcommand{\sR}{{\mathsf{R}}}
\newcommand{\w}{\mathbf{w}}

\newcommand{\com}{c}
\endlocaldefs

\begin{document}

\begin{frontmatter}
\title{Sample complexity bounds for categorical Markov random fields via Discrete Diffusions}
\runtitle{Structured sampling via discrete diffusions}

\begin{aug}
\author[A]{\fnms{Shivam}~\snm{Kumar}\ead[label=e1]{shivam.kumar@chicagobooth.edu}}
\and 
\author[A]{\fnms{Nabarun}~\snm{Deb}\ead[label=e2]{nabarun.deb@chicagobooth.edu}}

\address[A]{Booth School of Business, University of Chicago\printead[presep={,\ }]{e1,e2}}

\end{aug}

\begin{abstract}
Many applications in statistics, economics, and physics require sampling from high-dimensional categorical distributions with local dependence structures. Examples include finite memory language models, Ising and Potts systems in statistical physics and protein folding, etc. In modern machine learning, discrete diffusions have emerged as a flexible approach for sampling such data, with strong empirical performance. Motivated by this, we develop learning methods with end-to-end sample complexity bounds for discrete diffusion with uniform noising under local dependence, which we model through low order Markov random fields (MRFs). Our main technical insight is a new \emph{pinning decomposition} of the discrete score. It shows that unlike in continuous diffusions, the score decomposes into components where the dependence on time separates multiplicatively from the dependence on the target. Building on this decomposition, we propose a \emph{weight-sharing neural score learner} and combine it with $\tau$-leaping to obtain an end-to-end sampling procedure. Rather than treating score-learning error as a black-box input, as is common in existing sampling analyses, we study the score learning error from finite data and derive optimal sampling guarantees with explicit dependence on the vocabulary size, the interaction order of the MRF, and the sample size. Moreover, our strategy trains a single score network across uniform noise levels while leaving the sampling discretization to be chosen at inference-time. This allows the same trained model to trade accuracy for computational cost as inference-time budgets vary. Numerical experiments on Potts, Ising, and tree-structured models show that weight-sharing score networks outperform fully connected ones for sampling long sequences.
\end{abstract}


\end{frontmatter}

%\tableofcontents

\section{Introduction}

Diffusion models were introduced by \citet{sohldickstein2015} and became the dominant approach to generative modeling of continuous data after the work of \citet{ho2020ddpm} and \citet{song2021sde}. Their extension to discrete data began with multinomial diffusion \citep{hoogeboom2021multinomial} and with the structured transition matrices of D3PM \citep{austin2021d3pm}, both formulated in discrete time. \citet{campbell2022} later moved the framework to continuous time, where the forward and reverse processes are Markov chains, and \citet{lou2024discrete} showed that a model trained with the score entropy loss can match the perplexity of an autoregressive language model of the same size. Since then, discrete diffusion based language models have been trained with billions of parameters \citep{nie2025llada,ye2025dream}, and commercial systems built on them report large speedups over autoregressive decoding \citep{khanna2025mercury}. The same tools are also used to generate proteins and graphs \citep{campbell2024multiflow,gat2024dfm}, and we refer to \citet{yu2025survey} for a recent survey.

Motivated by the empirical success of discrete diffusions, in this paper we study their statistical complexity when sampling from a discrete law with a large support. Throughout this paper, we assume access to observations
\[
\x_0^{(1)},\ldots,\x_0^{(n)}
\overset{\mathrm{i.i.d.}}{\sim}q^\star,
\]
where $q^\star$ is an \emph{unknown discrete distribution} taking values in $[S]^{\sD}$. In other words, each of the $\sD$ coordinates ranges over a vocabulary of size $S$ which is potentially large. Our goal is to sample from $q^\star$ using discrete diffusions with a uniform noising kernel and provide end-to-end statistical convergence guarantees for our proposed approach. At a high level, a discrete diffusion model corrupts a sample from $q^\star$ by running a continuous-time Markov chain forward until its law is close to a simple reference distribution, and then generates new samples by reversing this process. If $q_t$ denotes the noised marginal at forward time $t$, the reverse dynamics are known except for the ratios
\begin{align}\label{eq:fsc}
s_t^\star(\y,\x)=\frac{q_t(\y)}{q_t(\x)}
\end{align}
between states that differ in one coordinate. This is the concrete score of \citet{meng2022concrete} and \citet{sun2023scorebased}, and it plays the role of $\nabla\log q_t$ in continuous diffusion. In practice, the score is learned with a neural network using score-entropy objectives \citep{lou2024discrete}, and the reverse chain is simulated approximately by $\tau$-leaping, a Poisson jump scheme originating in stochastic chemical kinetics \citep{gillespie2001,anderson2007} and adapted to discrete diffusion by \citet{campbell2022}.

Most of the discrete-diffusion theory instead begins after score learning, by assuming an accuracy guarantee for the score and studying how this error propagates through the sampler \citep{ren2024stochastic,liang2025samplers,dmitriev2026sharp}. Finite-sample guarantees for learning the discrete score itself are much more recent. \citet{srikanth2025sample} study neural score estimation on a fixed sampling schedule, \citet{cho2026minimax} show that the unstructured rate $S^{\sD}/n$ is minimax optimal under bounded neighboring ratios, and \citet{wakasugi2025statesize} obtain the same state-space-dependent rate for ReLU score networks, together with state-size-independent guarantees under an additional spectral smoothness assumption. By contrast, for continuous diffusion models there is already a developed route from score approximation, through finite-sample estimation, to distributional recovery \citep{block2020,suzuki023diffusion,chen2023score}. These results suggest that escaping the factor $S^{\sD}$ in the discrete problem requires a structural assumption on the target distribution.

A natural source of such structure is local dependence which is often exhibited in many popular real world models. Examples include finite-memory language models \citep{shannon1948}, Ising and Potts systems \citep{ising1925,potts1952}, and pairwise Potts models for protein sequences \citep{weigt2009,morcos2011}, etc. all of which are specified through interactions among small subsets of coordinates. We formalize this through a Markov random field of order $d$ where the probability mass function factors into potentials involving at most $d$ coordinates \citep{hammersley1971,besag1974,wainwright2008graphical}. Such a model is described by at most
\[
\binom{\sD}{d}S^d
\]
potential values, rather than $S^{\sD}$ probabilities. This makes it plausible that when $d\ll D$, the statistical complexity of discrete diffusion should be governed by the interaction order $d$ rather than by the ambient dimension $D$. There is, however, an important obstruction. The reverse process depends not on $q^\star$ itself, but on the noised marginals $q_t$, and noising generally destroys the original local factorization. Thus low-order structure of the data distribution does not by itself imply a local or low-dimensional representation of the noised score in \eqref{eq:fsc}. The relevant question is instead what trace of the structure in $q^\star$ survives after noising.

Our main answer is a new \emph{pinning decomposition}. Under uniform noising, the score can be written as a ratio of functions $g_t$, and each $g_t$ admits a representation of the schematic form
\begin{align}\label{eq:pin}
g_t(\x)
=
\sum_{\sm\subseteq[\sD]}
\underbrace{w_t(\sm)}_{\text{time}}
\underbrace{A_\sm(\x_\sm)}_{\text{target}},
\end{align}
where the weights $w_t(\sm)$ are known functions of the noise level and $A_\sm$ is a partial marginal of $q^\star$. The decomposition reveals an important structure in the forward chain. It shows that for learning the numerator and the denominator of the score, it is possible to separate the dependence on time from the dependence on the unknown target $q^\star$. Therefore the same target-specific quantities can be reused across all noise levels. To the best of our knowledge, the score in continuous diffusion models does not allow such exact representation featuring separation of time and target.

We also prove that for an order-$d$ Markov random field, the partial marginals $A_\sm(x_{\sm})$ in the pinning decomposition are sums of products of the original low-order potential tables. They can be evaluated exactly by eliminating the unpinned coordinates one at a time. A naive network implementation can nevertheless be extremely large because the same potential entries occur repeatedly across many intermediate sums. The key is that these entries need not be stored repeatedly. Sharing each table entry wherever it reappears naturally motivates the use of a weight-sharing neural architecture whose effective number of target-dependent parameters is comparable to the number of parameters describing the Markov random field itself. This is the mechanism by which our proposed score learner avoids the ambient state-space complexity $S^{\sD}$.

A second gap in the literature concerns the interface between training and sampling. In practice, one score network is trained across a continuum of noise levels, while the number and location of reverse-time steps (for sampling) can be chosen later at inference time \citep{lou2024discrete,sahoo2024mdlm,nie2025llada}. Much of the existing theory instead assumes score accuracy at the points of a prescribed sampling grid \citep{liang2025samplers,liang2025absorbing,cho2026minimax}, and existing finite-sample results that supply such guarantees likewise fit or analyze the score on a fixed schedule \citep{srikanth2025sample,cho2026minimax}. We take a different approach in this paper and control the risk of a single time-dependent score network integrated over the whole training interval. The sampling grid can then be chosen after training, allowing the same fitted score to be used at different computational budgets.

Together, these ideas lead to an end-to-end statistical rate of order
\[
\widetilde O_{\mathbb P}\!\left(\frac{S^d}{n}\right)
\]
for fixed interaction order, up to logarithmic factors and polynomial dependence on the remaining ambient parameters. Thus the interaction order $d$, rather than the sequence length $\sD$, determines the exponent of the vocabulary size. We now state our technical contributions more precisely.

\subsection{Contributions}\label{sec:contrib}

In this paper, we develop a discrete diffusion method with end-to-end statistical guarantees when $q^\star$ is a Markov random field of order $d$ with a bounded flip ratio (\Cref{ass:flip}), which is the discrete counterpart of a Lipschitz score. Our main insight is a pinning decomposition of the score (see \cref{prop:pinning}), which identifies the structure that the score retains after noising and explains why parameter sharing is the natural way to learn it. Our contributions are as follows.

\begin{enumerate}[leftmargin=1em,itemsep=3pt]
\item \emph{A pinning decomposition that separates time and target.}  Under uniform noising, the score is the ratio $s_t^\star(\x\oplus_i c,\x)=g_t(\x\oplus_i c)/g_t(\x)$, where $g_t(\x)=\bE_{\x_0\sim q^\star}\bigl[\alpha_t^{\Ham(\x,\x_0)}\bigr]$ and $\alpha_t\in[0,1)$ is an explicit function of $t$. We show that
\begin{align}\label{eq:gtapp}
  g_t(\x)
  =
  \sum_{\sm\subseteq[\sD]}\mu_t^{|\sm|}\alpha_t^{\sD-|\sm|}A_\sm(\x_\sm)
  =
  \sum_{k=0}^{\sD}
  \underbrace{\binom{\sD}{k}\mu_t^{k}\alpha_t^{\sD-k}}_{\text{time}}
  \;
  \underbrace{\bar A_k(\x)}_{\text{target}},
\end{align}
where $\mu_t=1-\alpha_t$, $A_\sm$ is the marginal of $q^\star$ on the coordinates in $\sm$, and $\bar A_k$ is the average of the $A_\sm$ over the sets $\sm$ of size $k$ (see \Cref{prop:pinning}). In words, each coordinate of the noised state is pinned to its clean value with probability $\mu_t$, and $g_t$ averages the clean marginals over the pinned sets. As a result, the decomposition separates the two sources of variation in the score exactly. Time enters only through the Binomial$(\sD,\mu_t)$ weights, which are known and do not depend on $q^\star$, while the data distribution enters only through its partial marginals, which do not depend on $t$. The same target quantities therefore serve every noise level, and a single network can learn the score over the whole interval $(0,T]$. To the best of our knowledge, the score of a continuous diffusion admits no exact separation of this kind. Moreover, the partial marginals of a Markov random field are sums of products of its small tables, so the decomposition also shows where the structure of $q^\star$ enters the score.

\item \emph{A combinatorial architecture, and why parameter sharing works.}  In continuous diffusions and flow-based models, the score is a gradient, and approximation results for score networks typically exploit the smoothness of the target and the Euclidean geometry of its support \citep{suzuki023diffusion,chen2023score}. Categorical data have neither derivatives nor a Euclidean geometry, so the architecture has to be built combinatorially from the structure of $q^\star$, and the pinning decomposition shows how to do this. Partial marginals on at most $d$ coordinates are tables with at most $S^d$ entries and can be stored directly. For larger pinned sets, $A_\sm$ is obtained by summing $q^\star$ over the free coordinates one at a time, and each such sum becomes a layer of the network. At first sight these layers look expensive. For instance, take $\sD=6$, $d=2$, and $\sm=\{1,2,3\}$ in \eqref{eq:gtapp}. To evaluate $A_{\sm}$, a naive sum over $(x^4,x^5,x^6)$ would require storing at least $S^3$ weights/parameters. However by exploiting the properties of Markov random fields, we show that these $S^3$ weights can be represented as the product of entries from two tables, each with $S^2$ parameters, i.e., $2S^2$ stored numbers can be reused to fill in all $S^3$ positions. We extend this observation to general $D,d$ via backward induction. This parameter reuse observation immediately motivates the use of a weight-sharing ReLU network where every table entry is stored once. We show that this network approximates the score uniformly over $t\in(0,T]$ and over all states, with an effective number of parameters of order $\binom{\sD}{d}S^d$ up to logarithmic factors (\Cref{thm:score,thm:bias}). As a methodological remark, this structure throws light on the practical success of convolutional and recurrent layers \citep{lecun1998gradient,elman1990finding,hochreiter1997lstm}, which reuse a small set of parameters at many positions or time steps. This is the same pattern that the pinning decomposition (\Cref{prop:pinning}) exposes in the score of locally dependent categorical data.

\item \emph{Estimation over the whole time interval.}  We prove an oracle inequality for empirical minimizers over weight-sharing classes under general conditions on the training loss (\Cref{thm:score_est}). These conditions require that the expected excess loss equals the integrated score risk, together with boundedness, Lipschitz continuity, and a second-moment bound in terms of that risk, and we verify them for the denoising and implicit score entropy losses of \citet{lou2024discrete} (\Cref{prop:dse_ise_estimation}). The complexity term depends on the effective number of parameters, and the second-moment bound yields a fast rate through a peeling argument. Because the training times are drawn at random from $(0,T]$, the controlled quantity is the time integral of the score entropy risk, and for either loss the learned score satisfies $\cL(\widehat s)=\widetilde O_{\mathbb P}(S^d/n)$ (\Cref{cor:score_est_rate}).

\item \emph{Sampling on any grid.}  Existing analyses of $\tau$-leaping assume that the score is accurate at the points of a prescribed sampling grid \citep{ren2024stochastic,dmitriev2026sharp}. Since our estimation guarantee controls the time integral of the risk, the grid can instead be chosen after training. For a time-dependent variant of $\tau$-leaping \citep{gillespie2001,campbell2022}, we bound the Kullback--Leibler divergence $\mathrm{KL}(q^\star\|p_T)$ between the target and the generated law by three terms,
\[
 \underbrace{\text{Score-estimation error}}_{\text{integrated training risk}}
 \;+\;
 \underbrace{\text{Initialization error}}_{\text{finite noising horizon}}
 \;+\;
 \underbrace{\text{Discretization error}}_{\text{sampling grid}}.
\]
The first term is controlled by our estimation result, the second by the terminal time $T$ of the forward process, and the third by the mesh of the grid (\Cref{thm:efficient_sampling}). No early stopping is needed, since the neighboring-ratio bound of \Cref{ass:flip} propagates to every $t>0$ (\Cref{lem:score_bounds}), and the proof relies on a martingale property of the true score along the reverse process. As a result, the same trained network can be run on different grids, trading accuracy for computational cost as inference budgets vary.

\item \emph{An end-to-end rate that is minimax optimal.}  Combining the estimation and sampling bounds with $T=\log(n\sD\log S)$ and $\Delta\le n^{-1}$ gives
\[
 \mathrm{KL}(q^\star\|p_T)
 \;\lesssim_{\log}\;
 \frac{S^d+x}{n},
 \qquad x\ge1,
\]
with probability at least $1-e^{-x}$ (\Cref{thm:efficient_sampling}). Here $\Delta$ is the largest sampling step, and $\lesssim_{\log}$ hides logarithmic factors and a multiplicative factor that is polynomial in $\sD$ and $R$ for fixed $d$, so that no constant depends exponentially on $\sD$. In other words, the interaction order $d$ takes the place of the dimension $\sD$ in the exponent of $S$, and the unstructured rate $S^{\sD}/n$ improves to $S^d/n$. In particular, for fixed $S$ the unstructured rate no longer vanishes once $\sD$ is of order $\log n$, whereas our bound still does. This rate is minimax optimal in $n$ and $S$ up to logarithmic factors. To see this, consider a Markov random field in which one potential depends on $x^1,\dots,x^d$ and all other coordinates are free of these. Sampling from such a distribution is at least as hard as sampling from its marginal on $[S]^d$, and for this marginal the lower bound of order $S^d/n$, up to logarithmic factors, follows from \citet{cho2026minimax}.

\item \emph{Numerical experiments.}  We test the method on Potts chains, frustrated Ising ladders, and categorical trees, all with $d=2$ and with $\sD$ up to $32$ (\Cref{sec:numerics}). At $\sD=32$, weight-sharing score networks, including a convolutional one, use fewer parameters than a fully connected network and achieve lower score risk and lower edge-pair total variation. A single trained network also serves every sampling grid we tried. At small $\sD$, the fully connected network catches up as $n$ grows, while at $\sD=32$ the gap persists, in line with the theory.

\end{enumerate}

\subsection{Literature review}\label{sec:Litrev}

Most theoretical work on discrete diffusion studies sampling given an accurate score. For uniform noising, this includes exact simulation by uniformization \citep{chen2024uniformization}, discrete-time schemes \citep{zhang2025discretetime}, stochastic-integral analyses of $\tau$-leaping \citep{ren2024stochastic}, sharper guarantees for first-order samplers \citep{liang2025samplers,dmitriev2026sharp,conforti2025dynamics,pham2025dmpm}, higher-order methods \citep{ren2025fast}, Gibbs corrections \citep{liang2026gadd}, and vocabulary-size-independent bounds \citep{kan2026vocab}. Masking diffusion has developed a parallel theory \citep{liang2025absorbing,liang2026masked,li2025convergence}, together with work on optimal schedules and data-adaptive sampling \citep{chen2025schedules,dmitriev2026adaptive,wainwright2026igc,wainwright2026ugc}. The empirical and algorithmic behavior of uniform, masking, and related corruption mechanisms has also been compared in \citet{shi2024simplified,vonrutte2025scaling,ou2025radd}. These results largely take the score-estimation error as an input while our focus is on deriving such an error bound from finite data under structural assumptions on the target.

For continuous diffusion models, finite-sample score learning and its conversion into distributional guarantees are well studied \citep{block2020,suzuki023diffusion,chen2023score}, with more recent work relating generative modeling to low-dimensional geometry and transport structure \citep{kumar2025likelihood,liang2024denoising,deb2025noregret}. As mentioned earlier, the corresponding theory for  discrete diffusions is more recent. The works of \citep{srikanth2025sample,cho2026minimax,wakasugi2025statesize} all feature a convergence rate of $S^{\sD}/n$ for unstructured $q^\star$. Our assumption is instead combinatorial where low-order Markov random field interactions in $q^\star$ lead to a faster rate governed by $S^d$, where $d\ll D$.

The aforementioned structure is classical in graphical models. Markov random fields factor distributions into local potentials which have been leveraged in \citep{hammersley1971,besag1974,wainwright2008graphical}.  Their local structure can make graphical model learning tractable \citep{bresler2015,vuffray2016,klivans2017}, and can yield neural density-estimation rates governed by local rather than ambient dimension \citep{vandermeulen2024dimension}. These results exploit the factorization of $q^\star$ itself, whereas discrete diffusion requires learning the score of the noised law, where the same factorization property is generally lost. A closely related continuous state space result is \citet{mei2023denoising}, where graphical structure and smoothness enable efficient score approximation. In our discrete setting, due to the absence of smoothness, we rely on our pinning decomposition from \eqref{eq:pin} and the observation that the partial marginals of $q^\star$ can be expressed using a smaller number of shared parameters than $S^\sD$. 

\subsection{Organization}\label{sec:organize}
In \Cref{sec:problem}, we set up the forward chain, uniform noising, the reverse chain with its concrete score, and the weight-sharing network class. In \Cref{sec:method}, we describe score learning with the denoising and implicit score entropy losses, together with the $\tau$-leaping sampler. In \Cref{sec:theory}, we introduce Markov random fields and the flip-ratio condition, and we present the pinning decomposition together with the approximation, estimation, sampling, and end-to-end results. In \Cref{sec:numerics}, we report numerical experiments. Appendix~\ref{app:numerical} contains further details of the numerical experiments, and the remaining appendices contain the proofs of the main results together with auxiliary lemmas.
\subsection{Notation}\label{sec:nota}
We write $[S]=\{1,\dots,S\}$ and $\cX=[S]^{\sD}$.  A state is $\x=(x^1,\dots,x^{\sD})$, and for $\sm\subseteq[\sD]$ we write $\x_\sm=(x^i)_{i\in\sm}$.  We write $\x\oplus_i c$ for the state obtained by shifting coordinate $i$ by $c$ modulo $S$ (\Cref{sec:reverse}), $\Ham$ for the Hamming distance, $\mathbf{1}\{\cdot\}$ for indicators, $\mathrm{KL}(p\|q)$ for the Kullback--Leibler divergence, and $a\vee b=\max(a,b)$.  The sign $\lesssim$ hides constants, $O(\cdot)$ and $o(\cdot)$ are the Bachmann--Landau symbols, $\widetilde O(\cdot)$ also hides logarithmic factors, and $O_{\mathbb P}(\cdot)$ and $\widetilde O_{\mathbb P}(\cdot)$ are their stochastic counterparts.
\section{Problem formulation}\label{sec:problem}

Let $q^\star$ be a unknown distribution on the discrete product space $\cX\defn[S]^\sD$, where $[S]\defn\{1,\dots,S\}$. Our goal is to construct a sampler for $q^\star$ from observations $\x_0^{(1)},\ldots,\x_0^{(n)} \overset{\mathrm{i.i.d.}}{\sim}q^\star$. We use continuous-time discrete diffusion \citep{campbell2022}, in which a forward Markov process progressively corrupts the data and a sampler is obtained by learning its time reversal.

\subsection{Forward CTMC on a discrete product space}
\label{subsec:forward}

We model the forward noising process by a continuous-time Markov chain $(\x_t)_{t\ge0}$ on $\cX$, with $\x_0\sim q^\star$. Its dynamics are specified by a time-homogeneous rate matrix $Q\in\bR^{|\cX|\times|\cX|}$. For $\x\neq\y$, the entry $Q(\x,\y)\ge0$ gives the transition rate from $\x$ to $\y$, while $Q(\x,\x)=-\sum_{\y\neq\x}Q(\x,\y)$ ensures that total probability is conserved.

Write $q_t(\y)\defn\Pr(\x_t=\y)$. Regarded as a column vector, $q_t$ satisfies the \emph{Kolmogorov forward equation} \citep{feller1940integro, anderson2012continuous}
\begin{equation}\label{eq:kolmogorov}
  \frac{\mathrm d q_t}{\mathrm dt}
  =
  Q^\top q_t,
  \qquad
  q_0=q^\star.
\end{equation}
We choose the forward dynamics to act separately on each coordinate, with a common token-level rate matrix $Q^{\tok}\in\bR^{S\times S}$. Each jump changes one coordinate, at a rate depending only on its current and proposed values. Writing
$\Ham(\x,\y)\defn|\{i\in[\sD]:x^i\neq y^i\}|$ for the Hamming distance, we set
\begin{equation}\label{eq:Q_cases}
  Q(\x,\y)
  =
  \begin{cases}
    Q^{\tok}(x^i,y^i),
      & \text{if $\x$ and $\y$ differ only in coordinate $i$},
      \\[3pt]
    \displaystyle\sum_{i=1}^{\sD}Q^{\tok}(x^i,x^i),
      & \text{if $\x=\y$},
      \\[3pt]
    0,
      & \text{if $\Ham(\x,\y)\ge2$}.
  \end{cases}
\end{equation}
Conditional on $\x_0$, the coordinates evolve independently; the coordinates under $q^\star$ need not be independent. Since $Q$ is time-homogeneous,
\eqref{eq:kolmogorov} has the solution
$q_t=e^{tQ^\top}q^\star$.
The transition probability from a fixed initial state
$\x_0$ to $\y$ at time $t$ is given by the
\emph{conditional forward kernel}
\begin{equation}\label{eq:cond_kernel}
  q_{t|0}(\y\mid\x_0)
  \defn
  \bigl[e^{tQ}\bigr]_{\x_0,\y}.
\end{equation}
Averaging over the initial state gives
$q_t(\y)=\sum_{\x_0\in\cX}
q_{t|0}(\y\mid\x_0)q^\star(\x_0)$.

The form of $Q$ in \eqref{eq:Q_cases} allows us to evaluate
conditional transition probabilities without forming the
full transition matrix $e^{tQ}$.
In the product basis $e_{x^1}\otimes\cdots\otimes e_{x^\sD}$,
where $e_a$ is the $a$-th standard basis vector of $\bR^S$,
the rate matrix has the Kronecker-sum representation
\[
  Q
  =
  \bigoplus_{i=1}^{\sD}Q^{\tok}
  \defn
  \sum_{i=1}^{\sD}
  \underbrace{\Id_S\otimes\cdots\otimes\Id_S}_{i-1}
  \otimes Q^{\tok}\otimes
  \underbrace{\Id_S\otimes\cdots\otimes\Id_S}_{\sD-i}.
\]
Indeed, the $i$-th summand has entries
\[
  \bigl[
    \Id_S^{\otimes(i-1)}
    \otimes Q^{\tok}
    \otimes\Id_S^{\otimes(\sD-i)}
  \bigr]_{\x,\y}
  =
  Q^{\tok}(x^i,y^i)
  \prod_{j\neq i}\mathbf{1}\{x^j=y^j\}.
\]
The identity factors require all coordinates other than $i$
to agree, giving exactly the contribution of coordinate $i$
to \eqref{eq:Q_cases}.

The summands act on different tensor factors and therefore
commute, giving
$e^{tQ}=\bigotimes_{i=1}^{\sD}e^{tQ^{\tok}}$.
Define the token-level transition kernel by
$q_{t|0}^{\tok}(b\mid a)\defn[e^{tQ^{\tok}}]_{a,b}$.
Taking the $(\x_0,\y)$ entry of this tensor product and using
\eqref{eq:cond_kernel} yields
\begin{equation}\label{eq:qt0_factor}
  \boxed{
    q_{t|0}(\y\mid\x_0)
    =
    \prod_{i=1}^{\sD}
    q_{t|0}^{\tok}(y^i\mid x_0^i)
  }.
\end{equation}
A conditional transition probability can therefore be evaluated from the $S\times S$ matrix $e^{tQ^{\tok}}$ by multiplying $\sD$ entries.

\paragraph{Uniform noising.}
Throughout the paper, we use uniform noising for the forward process. Each coordinate changes from $a$ to $b$ at rate $1/S$ for every pair of distinct values $a,b\in[S]$. The corresponding token-level rate matrix is
\[
  Q^{\tok}
  =
  \frac1S\bOne\bOne^\top-\Id_S,
\]
where $\bOne\in\bR^S$ is the vector of ones.
Since
$(S^{-1}\bOne\bOne^\top)^2=S^{-1}\bOne\bOne^\top$,
the matrix $-Q^{\tok}$ is idempotent. A direct computation (Lemma~\ref{lem:idempotent} in \Cref{app:prelim}) gives
\[
  e^{tQ^{\tok}}
  =
  e^{-t}\Id_S
  +
  \frac{1-e^{-t}}{S}\bOne\bOne^\top.
\]
The token-level transition probabilities are
\begin{equation}\nonumber
  q_{t|0}^{\tok}(b\mid a)
  =
  e^{-t}\mathbf{1}\{a=b\}
  +\frac{1-e^{-t}}{S}.
\end{equation}
In \eqref{eq:qt0_factor}, the factor for coordinate $i$
equals $[1+(S-1)e^{-t}]/S$ when $y^i=x_0^i$ and
$(1-e^{-t})/S$ otherwise.
Thus, $q_{t|0}(\y\mid\x_0)$ depends on $\x_0$ and $\y$
only through their Hamming distance.
Writing
$\alpha_t\defn(1-e^{-t})/[1+(S-1)e^{-t}]$,
we obtain, for $t>0$,
\begin{equation}\label{eq:qt0_hamming}
  \boxed{
    q_{t|0}(\y\mid\x_0)
    =
    \left(
      \frac{1+(S-1)e^{-t}}{S}
    \right)^{\!\sD}
    \alpha_t^{\,\Ham(\y,\x_0)}
  }.
\end{equation}
As $t\to\infty$, $\alpha_t\to1$, so the conditional kernel
converges to $S^{-\sD}$ for every $\x_0,\y\in\cX$.
Consequently, the marginal law $q_t$ converges to the
uniform distribution on $\cX$.

\subsection{The reverse process and the score}\label{sec:reverse}

The time-reversed process $(\overleftarrow{\x}_t)_{t\in[0,T]}$ has marginal law $\overleftarrow{q}_t=q_{T-t}$, starting from $q_T$ at time zero and reaching $q^\star$ at time $T$. For $0\le t<T$, its marginal law satisfies the Kolmogorov equation \citep{feller1940integro, anderson2012continuous}
\[
  \frac{\mathrm{d}\overleftarrow{q}_t}{\mathrm{d}t}
  =
  \overleftarrow{Q}_t^{\,\top}\overleftarrow{q}_t,
  \qquad
  \overleftarrow{q}_0=q_T,
\]
with off-diagonal transition rates
\begin{equation}\label{eq:reversal}
  \overleftarrow{Q}_t(\x,\y)
  =
  Q(\y,\x)
  \frac{q_{T-t}(\y)}{q_{T-t}(\x)},
  \qquad \x\neq\y,
\end{equation}
and diagonal entries chosen so that each row sums to zero. By \eqref{eq:qt0_hamming}, $q_{T-t}(\x)>0$ for $t<T$, so these ratios are well defined. The forward rates are known, whereas the probability ratios depend on the data distribution $q^\star$.

\paragraph{The concrete score.}
We define the \emph{concrete score} for $t\ge0$ by
\[
  s_t^\star(\y,\x)
  \defn
  \frac{q_t(\y)}{q_t(\x)},
  \qquad
  \x,\y\in[S]^\sD,\quad q_t(\x)>0.
\]
By \eqref{eq:Q_cases} and \eqref{eq:reversal}, the reverse process changes only one coordinate at a time. To index these transitions, write $a\oplus c\defn((a+c-1)\bmod S)+1$ for $a\in[S]$ and $c\in\mathbb{Z}$, and define
\[
  \x\oplus_i c
  \defn
  (x^1,\dots,x^{i-1},x^i\oplus c,x^{i+1},\dots,x^\sD).
\]
Each of the $\sD(S-1)$ states differing from $\x$ in exactly one coordinate has this form for a unique pair $(i,c)\in[\sD]\times[S-1]$. Under uniform noising, $Q(\x\oplus_i c,\x)=1/S$, so
\[
  \overleftarrow{Q}_t(\x,\x\oplus_i c)
  =
  \frac1S s_{T-t}^\star(\x\oplus_i c,\x),
  \qquad
  i\in[\sD],\quad c\in[S-1].
\]
For $t>0$, averaging \eqref{eq:qt0_hamming}
over $\x_0\sim q^\star$ expresses the marginal law as
\[
  q_t(\x)
  =
  \left(
    \frac{1+(S-1)e^{-t}}{S}
  \right)^{\!\sD}
  \bE_{\x_0\sim q^\star}
  \bigl[
    \alpha_t^{\,\Ham(\x,\x_0)}
  \bigr].
\]
The factor outside the expectation does not depend on $\x$
and cancels in the score ratio:
\begin{equation}\label{eq:score_uniform}
  \boxed{
    s_t^\star(\x\oplus_i c,\x)
    =
    \frac{
      \bE_{\x_0\sim q^\star}
      \bigl[
        \alpha_t^{\,\Ham(\x\oplus_i c,\x_0)}
      \bigr]
    }{
      \bE_{\x_0\sim q^\star}
      \bigl[
        \alpha_t^{\,\Ham(\x,\x_0)}
      \bigr]
    }
  ,
  \qquad i\in[\sD],\quad c\in[S-1].
  }
\end{equation}

\subsection{Neural score model}\label{subsec:network_class}

The score~\eqref{eq:score_uniform} depends on the unknown data distribution $q^\star$ and cannot be evaluated directly. We approximate it with a weight-sharing neural network trained on samples from $q^\star$. First we recall the definition of a standard sparse ReLU network before introducing its weight-sharing variant. 

\begin{definition}[Sparse ReLU network]\label{def:sparsenet}
A ReLU network with $\sL\in\N$ hidden layers, input
dimension $n_0$, hidden-layer widths $n_1,\ldots,n_\sL$,
and output dimension $n_{\sL+1}$ is the map
\[
  \sN_\rho(\x\mid\btheta)
  \defn
  \mathbf{A}_{\sL+1}\circ\rho\circ\mathbf{A}_\sL
  \circ\cdots\circ\rho\circ\mathbf{A}_1(\x),
\]
where
$\mathbf{A}_\ell(\x)=\W_\ell\x+\bb_\ell$,
$\W_\ell\in\R^{n_\ell\times n_{\ell-1}}$,
$\bb_\ell\in\R^{n_\ell}$, and $\rho$ applies
$\max(\cdot,0)$ componentwise.
The vector $\btheta$ collects all weights and biases.
We call $\sL$ the depth and
$\max_{0\le\ell\le\sL+1}n_\ell$ the width.
The sparsity and magnitude are, respectively,
\[
  \begin{aligned}
    |\btheta|_0
\defn
    \sum_{\ell=1}^{\sL+1}
    \left(
      |\mathrm{vec}(\W_\ell)|_0+|\bb_\ell|_0
    \right),
    % \\
    \qquad
    |\btheta|_\infty
    \defn
    \max_{1\le\ell\le\sL+1}
    \max\left(
      |\mathrm{vec}(\W_\ell)|_\infty,
      |\bb_\ell|_\infty
    \right).
  \end{aligned}
\]
We write $\bTheta_{n_0,n_{\sL+1}}(\sL,\sW,\sS,\sB)$ for the class of parameter vectors with depth $\sL$, width at most $\sW$, sparsity at most $\sS$, and magnitude at most $\sB$. For a fully \emph{connected neural network}, we will omit the sparsity parameter $\sS$.
\end{definition}

\begin{definition}[Weight-sharing neural network class]
\label{def:wsnn}
  Fix a fully connected neural network architecture $(\sL,\sW,\sB)$. Let $\Sws$ be the effective sparsity and let 
  \[
    \omega:
    [-\sB,\sB]^{\Sws}
    \longrightarrow
    \bigl\{
      \btheta:
      |\btheta|_\infty\le\sB\;
    \bigr\},
  \]
  be a Lipschitz map with Lipschitz constant
  \[
    \Lw
    \defn
    \sup_{\bv\neq\bv'}
    \frac{
      \|\omega(\bv)-\omega(\bv')\|_\infty
    }{
      \|\bv-\bv'\|_\infty
    },
    \qquad
    \bv,\bv'\in[-\sB,\sB]^{\Sws}.
  \]
  The \emph{weight-sharing neural network class} is
  \[
    \bTheta^{\mathrm{ws}}
    (\sL,\sW,\sB,\Sws,\omega)
    \defn
    \bigl\{
      \sN_\rho(\cdot\mid\omega(\bv)):
      \bv\in[-\sB,\sB]^{\Sws}
    \bigr\}.
  \]
The vector $\bv$ contains the $\Sws$ trainable coordinates, and $\omega(\bv)$ is the resulting network parameter vector.  
\end{definition}

% Requires \usepackage{tikz} and \usetikzlibrary{arrows.meta}.
% Insert outside remark/theorem, minipage, table, or another figure.
\begin{figure}[!htbp]
% \begin{figure}[h]
\centering
\begin{tikzpicture}[
  x=1cm,y=1cm,>=Stealth,
  neuron/.style={circle,draw=black!65,fill=white,minimum size=7mm,inner sep=1pt},
  weightA/.style={->,thick,draw=blue!65!black},
  weightB/.style={->,thick,dashed,draw=red!65!black},
  labelbox/.style={fill=white,inner sep=1.5pt,font=\small},
  font=\small
]
% (a) A width-two filter applied at three spatial locations.
\node[anchor=west,font=\bfseries] at (-.5,3.65)
  {(a) Convolution: sharing across positions};
\foreach \j/\xx in {1/0,2/2,3/4,4/6}
  \node[neuron] (x\j) at (\xx,1) {$x_{\j}$};
\foreach \j/\xx in {1/1,2/3,3/5}
  \node[neuron] (y\j) at (\xx,2.8) {$y_{\j}$};
\foreach \j/\k in {1/2,2/3,3/4}{
  \draw[weightA] (x\j) -- node[labelbox,pos=.48] {$a$} (y\j);
  \draw[weightB] (x\k) -- node[labelbox,pos=.48] {$b$} (y\j);
}
\node[anchor=west,align=left,text width=5.2cm] at (7.1,2.45)
  {$y_j=\rho(ax_j+bx_{j+1})$\\[6pt]
   $\bv=(a,b)$\\[3pt]
   $\omega(\bv)=(a,b,a,b,a,b)$\\[6pt]
   Two free coefficients; six uses.};
\draw[black!20] (-.5,.35) -- (12.4,.35);
% (b) The same scalar recurrent cell used at three time steps.
\node[anchor=west,font=\bfseries] at (-.5,-.3)
  {(b) Recurrence: sharing across time steps};
\node[neuron] (h0) at (0,-1.6) {$h_0$};
\foreach \t/\xx in {1/2,2/4,3/6}{
  \node[neuron] (h\t) at (\xx,-1.6) {$h_{\t}$};
  \node[neuron] (z\t) at (\xx,-3.2) {$x_{\t}$};
  \draw[weightA] (z\t) -- node[labelbox] {$u$} (h\t);
}
\foreach \j/\k in {0/1,1/2,2/3}
  \draw[weightB] (h\j) -- node[labelbox,above] {$w$} (h\k);
\node[anchor=west,align=left,text width=5.2cm] at (7.1,-2.15)
  {$h_t=\rho(ux_t+wh_{t-1})$\\[6pt]
   $\bv=(u,w)$\\[3pt]
   $\omega(\bv)=(u,w,u,w,u,w)$\\[6pt]
   Two free coefficients; six uses.};
\end{tikzpicture}
\caption{Weight sharing across spatial positions in a CNN and time steps
in an unrolled RNN. Matching edge labels denote shared coefficients,
replicated by $\omega$. Here $\rho$ is ReLU; biases and zero entries
are omitted.}
\label{fig:wsnn_sharing}
\end{figure}

\begin{remark}[Examples of weight-sharing]
Several standard architectures admit this representation. Placing the coordinates of $\bv$ at a prescribed set of $\Sws$ parameter positions and setting the remaining entries to zero gives a fixed-support subclass of the sparse networks considered by \citet{schmidt2020nonparametric}, with $\sS=\Sws$. For a convolutional layer, $\bv$ stores a filter, and $\omega$ places the same coefficients at each spatial location \citep{lecun1998gradient}. For example, the layer
\[
 y_j=\rho\!\left(\sum_{\ell=1}^{r}a_\ell x_{j+\ell-1}\right),
 \qquad j=1,\ldots,m,
\]
has $\Sws=r$ free coefficients and $\sS=mr$ weight positions. Increasing the number of locations at which the filter is applied does not increase the number of free parameters. 

Recurrent networks similarly reuse parameters across time \citep{elman1990finding}. For the ReLU recurrence 
\[
 h_t=\rho(Ux_t+Wh_{t-1}+b),\qquad t=1,\ldots,m,
\]
$\bv$ contains the entries of $(U,W,b)$, and $\omega$ copies this parameter block into each of the $m$ recurrent cells in the unrolled computation.  Figure~\ref{fig:wsnn_sharing} illustrates convolutional and recurrent sharing. The parameter maps in these examples copy coordinates, possibly with a sign, and insert fixed entries. They satisfy  $\Lw=1$.
\end{remark}

\smallskip\noindent
For the score model, the network takes $(\x,t)$ as input and has $\sD(S-1)$ outputs, one for each distinct single-coordinate neighbour of $\x$. Writing $s^{\btheta}(\x,t)=\sN_\rho((\x,t)\mid\btheta)$, we denote its $(i,c)$-th component by
\[
  s_t^{\btheta}(\x\oplus_i c,\x),
  \qquad i\in[\sD],\quad c\in[S-1].
\]
We write $s$ for a time-dependent score function and $s_t$ for the score at time $t$. For training and statistical analysis, we restrict the estimated scores to a fixed positive range. For $\sR\ge1$, define
\begin{small}
\begin{equation}\label{eq:ws-restricted}
  \bTheta^{\mathrm{ws}}_{\sR}
  (\sL,\sW,\sB,\Sws,\omega)
  \defn
  \bigl\{
    \sN_\rho(\cdot\mid\btheta)
    \in
    \bTheta^{\mathrm{ws}}(\sL,\sW,\sB,\Sws,\omega)
    :
    s_t^{\btheta}(\x\oplus_i c,\x)
    \in[\sR^{-1},\sR]
    \text{ for all }(\x,t,i,c)
  \bigr\}.
\end{equation}
\end{small}
The next section defines the training objective over this class.

\section{Method}\label{sec:method}
We first estimate the concrete score $s_t^\star$ from observations drawn from $q^\star$ and denote the estimate by $\widehat s_t$. We then substitute $\widehat s$ for $s^\star$ in the reverse rates and generate samples by simulating the resulting chain with $\tau$-leaping.

\subsection{Score estimation}\label{subsec:optimization}

We estimate the score within the range-restricted network class $\bTheta_{\sR}^{\mathrm{ws}}$ defined in~\eqref{eq:ws-restricted}. Under uniform noising, the reverse rates depend only on scores associated with single-coordinate transitions,  of which has forward rate $1/S$. We measure score estimation error using the \emph{score entropy loss} of \citet{lou2024discrete}, built from $K(u)\defn u-\log u-1$ for $u>0$. \citet{lou2024discrete} motivate this choice by noting that the squared-error loss used in earlier work \citep{meng2022concrete} insufficiently penalizes zero or negative score estimates and can lead to unstable training.

\smallskip
\noindent
For $s^{\btheta}\in\bTheta_{\sR}^{\mathrm{ws}}$
and $0<t\le T$, define
\begin{equation}\label{eq:master_obj}
\begin{aligned}
  \cL_t(s^{\btheta})
  &\defn
  \frac1S
  \bE_{\x\sim q_t}
  \left[
    \sum_{i=1}^{\sD}
    \sum_{\substack{c\in[S]\\c\neq S}}
    s_t^\star(\x\oplus_i c,\x)
    K\left(
      \frac{s_t^{\btheta}(\x\oplus_i c,\x)}
           {s_t^\star(\x\oplus_i c,\x)}
    \right)
  \right],
  % \\
  \quad
  \cL(s^{\btheta})
  % &
  \defn
  \int_0^T \cL_t(s^{\btheta})\,\mathrm dt.
\end{aligned}
\end{equation}
The shift $c=S$ is excluded because it leaves $\x$ unchanged.  Since $K(u)\ge0$, with equality only at $u=1$, $\cL_t(s^{\btheta})\ge0$ is nonnegative and vanishes exactly when $s_t^{\btheta}=s_t^\star$ for every transition.

To obtain a training objective, we rewrite the population loss
using the known forward transition probabilities. Expanding each
summand, with the state arguments suppressed, gives
\begin{equation}\label{eq: K expand}
     s_t^\star K\left(\frac{s_t^{\btheta}}{s_t^\star}\right)
  =
  s_t^{\btheta}
  -
  s_t^\star\log s_t^{\btheta}
  +
  s_t^\star\log s_t^\star
  -
  s_t^\star. 
\end{equation}
The last two terms in~\eqref{eq: K expand} do not depend on $\btheta$. To rewrite the remaining terms, define the
\emph{conditional score}
\[
  s_{t|0}^\star(\y,\x\mid\x_0)
  \defn
  \frac{q_{t|0}(\y\mid\x_0)}
       {q_{t|0}(\x\mid\x_0)},
  \qquad 0<t\le T.
\]
This ratio can be evaluated for each observed initial state from~\eqref{eq:qt0_hamming}, without knowing $q^\star$. Averaging over $\x_0\sim q^\star$ and $\x\sim q_{t|0}(\cdot\mid\x_0)$ preserves the expectation of the first term. A direct computation (Lemma~\ref{lem:cond_score} in Appendix~\ref{app:prelim}) shows that replacing $s_t^\star$ by $s_{t|0}^\star$ under this joint expectation also preserves the expectation of the second term. Hence
\begin{equation}\label{eq:obj_division}
\begin{aligned}
  &\cL_t(s^{\btheta})
   \\={}&
  \frac1S
  \bE_{\x_0\sim q^\star}
  \bE_{\x\sim q_{t|0}(\cdot\mid\x_0)}
  \Biggl[
    \sum_{i=1}^{\sD}
    \sum_{\substack{c\in[S]\\c\neq S}}
    \Bigl\{
      s_t^{\btheta}(\x\oplus_i c,\x)
      -
      s_{t|0}^\star(\x\oplus_i c,\x\mid\x_0)
      \log s_t^{\btheta}(\x\oplus_i c,\x)
    \Bigr\}
  \Biggr]
  +\Xi_t
  \\
  ={}&
  \cL_{\mathrm{cond},t}(s^{\btheta})+\Xi_t,
\end{aligned}
\end{equation}
where $\Xi_t$ collects the last two terms of the expansion, averaged and summed as in~\eqref{eq:master_obj}.
Since $\Xi_t$ does not depend on $\btheta$, the conditional loss
\[
  \cL_{\mathrm{cond}}(s^{\btheta})
  \defn
  \int_0^T
  \cL_{\mathrm{cond},t}(s^{\btheta})\,\mathrm dt
\]
has the same minimizers as $\cL$.

\paragraph{Empirical score learning.}
We estimate $\cL_{\mathrm{cond}}$ by averaging over the observations and randomly sampled times. We observe $\x_0^{(1)},\ldots,\x_0^{(n)} \overset{\mathrm{i.i.d.}}{\sim}q^\star$ and, independently, draw $\xi_1,\ldots,\xi_n \overset{\mathrm{i.i.d.}}{\sim}\operatorname{Unif}(0,T]$. Let $\ell_t(s;\x_0)$ be an unbiased training loss for $\cL_{\mathrm{cond},t}(s)$, that is,

\begin{equation}\label{eq:loss_mean_identity}
  \bE_{\x_0\sim q^\star}
  \bigl[\ell_t(s;\x_0)\bigr]
  =
  \cL_{\mathrm{cond},t}(s)
  =
  \cL_t(s)-\Xi_t.
\end{equation}
We estimate the score by
\begin{equation}\label{eq:emp_score_loss}
\begin{aligned}
  \widehat\cL_n^\ell(s^{\btheta})
  &\defn
  \frac{T}{n}
  \sum_{k=1}^n
  \ell_{\xi_k}(s^{\btheta};\x_0^{(k)}),
  % \\
  \qquad
  \widehat\btheta
  % &
  \in
  \argmin_{\btheta:\,s^{\btheta}\in\bTheta_{\sR}^{\mathrm{ws}}}
  \widehat\cL_n^\ell(s^{\btheta}),
  \qquad
  \widehat s\defn s^{\widehat\btheta}.
\end{aligned}
\end{equation}
The factor $T$ ensures that $\bE[\widehat\cL_n^\ell(s)]=\cL_{\mathrm{cond}}(s)$. We consider two representations of $\ell_t$: denoising score entropy (DSE) and implicit score entropy (ISE) \citep{lou2024discrete,chen2024uniformization}.

For a fixed observation $\x_0$ and $0<t\le T$, the DSE
loss is the conditional expression in~\eqref{eq:obj_division}:
\[
\begin{aligned}
  \ell_t^{\mathrm{DSE}}(s;\x_0)
  \defn{}&
  \frac1S
  \bE_{\x\sim q_{t|0}(\cdot\mid\x_0)}
  \Biggl[
    \sum_{i=1}^{\sD}
    \sum_{\substack{c\in[S]\\c\neq S}}
    \Bigl\{
      s_t(\x\oplus_i c,\x)
      -
      s_{t|0}^\star(\x\oplus_i c,\x\mid\x_0)
      \log s_t(\x\oplus_i c,\x)
    \Bigr\}
  \Biggr].
\end{aligned}
\]
Under uniform noising, the factors for unchanged coordinates
cancel in the conditional-score ratio.
Using~\eqref{eq:qt0_hamming}, we obtain
\begin{equation}\label{eq:cond_score_unif}
  s_{t|0}^{\star}(\x\oplus_i c,\x\mid\x_0)
  =
  \alpha_t^{
    \mathbf{1}\{x^i\oplus c\neq x_0^i\}
    -
    \mathbf{1}\{x^i\neq x_0^i\}
  }
  \in\{\alpha_t^{-1},1,\alpha_t\}.
\end{equation}
The ratio equals $\alpha_t^{-1}$ when the transition
restores the initial symbol $x_0^i$, $\alpha_t$ when
it changes that symbol to another value, and $1$ otherwise.
This specifies the conditional score needed to evaluate
the DSE integrand.

The ISE representation removes the conditional score
from the logarithmic term.
For $c\in[S]$ with $c\neq S$, write
$\x\ominus_i c\defn\x\oplus_i(S-c)$.
For each $i$ and $c$, expanding the logarithmic expectation
gives
\[
\begin{aligned}
&\bE_{\x\sim q_{t|0}(\cdot\mid\x_0)}
\left[
  s_{t|0}^\star(\x\oplus_i c,\x\mid\x_0)
  \log s_t(\x\oplus_i c,\x)
\right]
=
\bE_{\y\sim q_{t|0}(\cdot\mid\x_0)}
\left[
  \log s_t(\y,\y\ominus_i c)
\right].
\end{aligned}
\]
The last display uses the change of variables $\y=\x\oplus_i c$, whose inverse is $\x=\y\ominus_i c$. Applying this identity to the logarithmic terms defines the ISE loss:
\[
\begin{aligned}
  \ell_t^{\mathrm{ISE}}(s;\x_0)
  \defn{}&
  \frac1S
  \bE_{\x\sim q_{t|0}(\cdot\mid\x_0)}
  \Biggl[
    \sum_{i=1}^{\sD}
    \sum_{\substack{c\in[S]\\c\neq S}}
    \Bigl\{
      -
      \log s_t(\x,\x\ominus_i c)
    \Bigr\}
  \Biggr].
\end{aligned}
\]
Equation~\eqref{eq:obj_division} establishes \eqref{eq:loss_mean_identity} for DSE, and the preceding change of variables establishes it for ISE. Either representation can therefore be used in the empirical objective~\eqref{eq:emp_score_loss}.

The computational distinction arises when we estimate the conditional expectations by Monte Carlo. We can approximate either expectation by averaging its integrand over independent noisy states drawn from the conditional forward kernel in~\eqref{eq:qt0_hamming}. For each sampled state, evaluating the full DSE summand requires one network evaluation, which returns all single-coordinate scores at that state. The full ISE summand also requires network evaluations at neighbouring states for its logarithmic terms \citep{lou2024discrete,chen2024uniformization}. Our theoretical analysis in the next section uses the exact conditional expectations.

\subsection{$\tau$-leaping sampling}\label{subsec:tau_leaping}
We use $\tau$-leaping to simulate the estimated reverse process \citep{campbell2022}. The sampler (\Cref{alg:tau_leaping}) advances along a prescribed time grid and allows several jumps within each step. Let
\[
  0=t_0<t_1<\cdots<t_N=T
\]
be a grid on the reverse-time interval $[0,T]$, and denote the sampler's state at $t_k$ by $\z_k\in[S]^{\sD}$. We initialize the sampler from the uniform stationary distribution of the forward process:
\[
  \z_0\sim\operatorname{Unif}([S]^{\sD}).
\]
\begin{algorithm}[htpb!]
  \caption{$\tau$-leaping for uniform noising}
  \label{alg:tau_leaping}
  \begin{algorithmic}[1]
    \Require A grid $0=t_0<t_1<\cdots<t_N=T$ and the estimated
    score $\widehat s$
    \State Draw
    \[
      \z_0
      \sim
      \operatorname{Unif}([S]^{\sD}).
    \]
    \For{$k=0,\ldots,N-1$}
      \ForAll{$(i,c)\in[\sD]\times[S-1]$}
        \State Independently draw
        \[
          P_{k,i,c}
          \sim
          \operatorname{Poisson}\left(
            \frac1S
            \int_{t_k}^{t_{k+1}}
            \widehat s_{T-t}
            (\z_k\oplus_i c,\z_k)\,\mathrm dt
          \right).
        \]
      \EndFor
      \For{$i=1,\ldots,\sD$}
        \State Set
        \begin{equation}\label{eq:evolve}
            z_{k+1}^i \gets z_k^i \oplus \sum_{c\in[S-1]} cP_{k,i,c}.
        \end{equation}
      \EndFor
    \EndFor
    \State \Return $\z_N$.
  \end{algorithmic}
\end{algorithm}
On each interval $[t_k,t_{k+1})$, we fix the state entering the score at $\z_k$ and retain the time dependence of $\widehat s_{T-t}$. For $k=0,\ldots,N-1$, conditional on $\z_k$, define the coordinate rate matrices $Q_{k,t}^i\in\bR^{S\times S}$ by
\begin{equation}\label{eq:tau_rate}
  Q_{k,t}^i(a,a\oplus c)
  =
  \frac1S
  \widehat s_{T-t}(\z_k\oplus_i c,\z_k),
  \qquad
  a\in[S],\quad c\in[S-1],\quad i\in[\sD],
\end{equation}
for $t\in[t_k,t_{k+1})$, with diagonal entries chosen so that every row sums to zero. Conditional on $\z_k$, coordinate $i$ starts from $z_k^i$, and the coordinates evolve independently under these rate matrices. Their values at $t_{k+1}$ form $\z_{k+1}$. For each coordinate, the rate of a shift depends on time and $\z_k$. %, but not on the current coordinate value $a$ as in \eqref{eq:tau_rate}. 
% Since additive shifts commute, their counts determine the endpoint without requiring their occurrence times or order. 
Applying \Cref{lem:poisson_representation} (see \Cref{app:poisson}) with
\[
  h=t_{k+1}-t_k,
  \qquad
  \lambda_c(r)
  =
  \frac1S
  \widehat s_{T-t_k-r}(\z_k\oplus_i c,\z_k),
  \qquad 0\le r<h,
\]
shows that the law of $z_{k+1}$ given $z_k$ (in \eqref{eq:evolve}) is the same as that of a Markov process evolving according to \eqref{eq:kolmogorov}, initialized at $z_k$ with rate given by the matrix in \eqref{eq:tau_rate}. At each step, the $\sD(S-1)$ counts $P_{k,i,c}$ are independent conditional on $\z_k$
and can be drawn in parallel. \Cref{alg:tau_leaping} applies their combined shifts. We denote the law of the final output by
\begin{equation}\label{eq:finalaw}
  p_T\defn\operatorname{Law}(\z_N).
\end{equation}

\section{Theoretical results}\label{sec:theory}

We now give guarantees for score approximation, score estimation, and sampling.  We begin by imposing a natural structure on the data distribution $q^\star$. As mentioned in the Introduction, many popular models for discrete distributions exhibit spatial, temporal, network based, and other forms of local dependence structures. This observation leads us to Markov random fields as defined below.

\begin{definition}[Markov random field of order $d$]
\label{def:mrf}
  Let $1\le d\le\sD$, and define
  \[
    \Fd\defn\{\sm\subseteq[\sD]:|\sm|=d\}.
  \]
  For $\x\in[S]^\sD$ and $\sm\in\Fd$, write   $\x_\sm=(x^i)_{i\in\sm}\in[S]^d$.   We say that $q^\star$ belongs to $\MRF_d$ if there exist   potentials $\psi_\sm:[S]^d\to[0,1]$, $\sm\in\Fd$, such that
  \begin{equation}\label{eq:qMRF}
    q^\star(\x)
    =
    \prod_{\sm\in\Fd}\psi_\sm(\x_\sm),
    \qquad
    \sum_{\x\in[S]^\sD}q^\star(\x)=1.
  \end{equation}
  It is easy to check that the representation uses at most  $\binom{\sD}{d}S^d\le\sD^dS^d$ potential values. For fixed $d$, this number grows polynomially in $\sD$ and $S$. This is in contrast with an unstructured $q^\star$ which would require $S^\sD-1$ values to be fully specified. By definition $\MRF_1$ is the class of product distributions and the classes are nested:
\[
  \MRF_1\subseteq\MRF_2\subseteq\cdots\subseteq\MRF_\sD.
\]
This is because one can obtain an order-$(d+1)$ representation for any $d<\sD$,
by assigning each $d$-element set to a $(d+1)$-element super set.
\end{definition}
\noindent\paragraph{Examples.}
We present three standard classes of discrete distributions that fall naturally in the aforementioned class of Markov random fields.

\smallskip
\noindent\emph{1. $k$-gram models.}
A $k$-gram model on $[S]^D$ has the factorization
\[
q^\star(x)
=
q^\star_{1:k-1}(x_1,\ldots,x_{k-1})
\prod_{j=k}^D
p_j\!\left(
    x_j
    \,\middle|\,
    x_{j-k+1},\ldots,x_{j-1}
\right), 
\]
with potentials $q^\star_{1:k-1}$, $\{p_j\}_{j=k}^D$
Each factor involves at most $k$ coordinates, so
$q^\star\in\MRF_k$. This includes finite-memory language models (see \cite{shannon1948,liu2024infini,brants2007large}).

\smallskip
\noindent\emph{2. Potts models.}
Let $G=([D],E)$ be a graph. A ferromagnetic Potts model (see \cite{potts1952,costeniuc2005complete,malekian2026ldp}) can be written as
\[
q^\star(x)
=
\frac{1}{Z}
\exp\left\{
    -\beta \sum_{\{i,j\}\in E}
    \mathbf{1}\{x_i\neq x_j\}
\right\},
\qquad
\beta\geq 0,
\]
where $Z$ is the partition function. Each interaction involves only the
pair $(x_i,x_j)$, and hence $q^\star\in\MRF_2$.

\smallskip
\noindent\emph{3. Bayesian networks.}
Let $G$ be a directed acyclic graph on $[D]$, and let
$\pa(i)$ denote the parents of node $i$. A Bayesian network (see \cite{liu2018nonparametric,koller2009graphical})  factorizes as
\[
q^\star(x)
=
\prod_{i=1}^D
p_i\!\left(x_i\,\middle|\,x_{\pa(i)}\right).
\]
If $r:=\max_{i\in[D]}|\pa(i)|$, then every factor depends on at most $r+1$ coordinates. Therefore,
in the sense of Definition~\ref{def:mrf}, we have $q^\star\in\MRF_{r+1}$.

\paragraph{Bounded neighboring ratios.}
We also control how much the forward chain probabilities  change when one coordinate changes via the following assumption.
\begin{assumption}[Bounded flip ratio]\label{ass:flip}
  The distribution $q^\star$ has full support, and there is   a constant $R\ge1$ such that
  \begin{equation}\label{eq:flip_ratio}
    R^{-1}
    \le
    \frac{q^\star(\x\oplus_i c)}{q^\star(\x)}
    \le
    R
  \end{equation}
  for every $\x\in[S]^\sD$, $i\in[\sD]$, and   $c\in[S-1]$.
\end{assumption}

\begin{remark}[Interpretation of the ratio bound]\label{rem:flip}
Under full support, \eqref{eq:flip_ratio} is equivalent to
\[
  \left|
    \log q^\star(\x\oplus_i c)-\log q^\star(\x)
  \right|
  \le \log R.
\]
Changing the coordinates on which $\x$ and $\y$ differ, one at a time, yields
\[
  \left|\log q^\star(\y)-\log q^\star(\x)\right|
  \le (\log R)\Ham(\x,\y).
\]
Thus $\log q^\star$ is Lipschitz with respect to Hamming distance. Equivalently,
  \(
    R^{-\Ham(\x,\y)}
    \le
    {q^\star(\y)}/{q^\star(\x)}
    \le
    R^{\Ham(\x,\y)}
  \)
The continuous analogue is a bound on the magnitude of the score. For a positive, continuously differentiable density $p^\star$ on $\bR^m$, the bound $\|\nabla\log p^\star(\x)\|_2\le L$ for every $\x$ implies
\[
  \left|\log p^\star(\y)-\log p^\star(\x)\right|
  \le L\|\y-\x\|_2,
\]
by integrating the gradient along the line segment from $\x$ to $\y$. Assumption~\ref{ass:flip} imposes the corresponding bound with Hamming distance in place of Euclidean distance and $\log R$ in place of $L$. Continuous diffusion analyses also study Lipschitz regularity of the score, which controls how the score varies with the state \citep{mena2025rectified, stephanovitch2026lipschitz}.
\end{remark}

\subsection{Score approximation}\label{subsec:score_approx}

The estimation results of the next subsection require the network
class to contain a good approximation of the true score.  We now show
that the MRF structure of $q^\star$ yields such an approximation. Recall that $\alpha_t=\frac{1-e^{-t}}{1+(S-1)e^{-t}}$. Define
\begin{equation}\label{eq:g_def}
  g_t(\x)
  :=
  \bE_{\x_0\sim q^\star}
  \left[
    \alpha_t^{\Ham(\x,\x_0)}
  \right]
  =
  \sum_{\x_0\in[S]^\sD}
  \alpha_t^{\Ham(\x,\x_0)}
  q^\star(\x_0).
\end{equation}
By~\eqref{eq:score_uniform},
\begin{equation}\label{eq:score_ratio_g}
  s_t^\star(\x\oplus_i c,\x)
  =
  \frac{g_t(\x\oplus_i c)}{g_t(\x)}.
\end{equation}
Thus it is enough to approximate $g_t$ and its values at neighboring states.  Under Assumption~\ref{ass:flip}, the true score takes values in $[R^{-1},R]$ for every $t>0$ (\Cref{lem:score_bounds} in Appendix~\ref{app:prelim}), which is the range used in the range-restricted network class.

\paragraph{Pinning decomposition.}
We next express $g_t$ through the marginals of $q^\star$.  Set
$\mu_t:=1-\alpha_t$.  For $\sm\subseteq[\sD]$, let
\begin{align}\label{eq:marg-def}
  A_\sm(\x_\sm)
  :=
  \Pr_{\x_0\sim q^\star}
  \bigl((\x_0)_\sm=\x_\sm\bigr)
  =
  \sum_{\substack{\x_0\in[S]^\sD\\
                  (x_0)^i=x^i,\;i\in\sm}}
  q^\star(\x_0)
\end{align}
be the marginal of $q^\star$ on the coordinates in $\sm$, with $A_\emptyset\equiv1$.  The following decomposition holds for every distribution $q^\star$, and the Markov random field structure enters only through the marginals $A_\sm$.

\begin{proposition}[Pinning decomposition]\label{prop:pinning}
  For every $t>0$ and $\x\in[S]^\sD$,
  \begin{equation}\label{eq:pinning}
    g_t(\x)
    =
    \sum_{\sm\subseteq[\sD]}
    \mu_t^{|\sm|}
    \alpha_t^{\sD-|\sm|}
    A_\sm(\x_\sm).
  \end{equation}
  Equivalently, $g_t(\x)=\bE_\sm\bigl[A_\sm(\x_\sm)\bigr]$, where $\sm$ contains each coordinate independently with probability $\mu_t$.  Grouping the sets $\sm$ by their size,
  \begin{equation}\label{eq:pinning_levels}
    g_t(\x)
    =
    \sum_{k=0}^{\sD}
    \binom{\sD}{k}\mu_t^{k}\alpha_t^{\sD-k}\,
    \bar A_k(\x),
    \qquad
    \bar A_k(\x)
    :=
    \binom{\sD}{k}^{-1}
    \sum_{|\sm|=k}A_\sm(\x_\sm).
  \end{equation}
\end{proposition}

\begin{proof}
  For each coordinate, $\alpha_t^{\mathbf{1}\{x_0^i\neq x^i\}}=\alpha_t+\mu_t\mathbf{1}\{x_0^i=x^i\}$.  Taking the product over $i\in[\sD]$ and expanding gives
  \(
    \alpha_t^{\Ham(\x,\x_0)}
    =
    \sum_{\sm\subseteq[\sD]}\mu_t^{|\sm|}\alpha_t^{\sD-|\sm|}\prod_{i\in\sm}\mathbf{1}\{x_0^i=x^i\}.
  \)
  Finally we take expectations over $\x_0\sim q^\star$ and use $\bE\prod_{i\in\sm}\mathbf{1}\{x_0^i=x^i\}=A_\sm(\x_\sm)$ to get~\eqref{eq:pinning}.  The weights $\mu_t^{|\sm|}\alpha_t^{\sD-|\sm|}$ are the probabilities of the independent pinning scheme, which gives the expectation form, following which we collect the sets of size $k$ which yields~\eqref{eq:pinning_levels}.
\end{proof}

\Cref{prop:pinning} separates the two sources of variation in the score.  In~\eqref{eq:pinning_levels}, time enters only through the Binomial$(\sD,\mu_t)$ weights, which are explicit and do not depend on $q^\star$, while the data distribution enters only through the averages $\bar A_k$ of its partial marginals, which do not depend on $t$.  As $t\to0$, the weights concentrate on $k=\sD$ and $g_t(\x)$ approaches $q^\star(\x)$, while as $t\to\infty$ they concentrate on $k=0$ and $g_t(\x)$ approaches one.  In between, the noise level only changes how much weight each level $k$ receives.  As a result, the same target quantities serve every noise level, and a single network with time-independent weights can approximate the score uniformly over $(0,T]$.  This is what allows us to train one network with random times and to choose the sampling grid after training (\Cref{subsec:efficient_sampling}).

The score of a continuous diffusion behaves differently.  For the Ornstein--Uhlenbeck forward process, $q_t$ is the convolution of a rescaled copy of $q^\star$ with a Gaussian whose variance depends on $t$, so that $\nabla\log q_t$ depends on $t$ and on $q^\star$ jointly through this convolution.  Analyses of continuous diffusion and flow-based models handle this coupling through the smoothness of the target and the Euclidean geometry of its support \citep{liang2024denoising,deb2025noregret,kumar2026flow}.  To the best of our knowledge, no exact separation into explicit time weights and time-free functionals of the data distribution is available in the continuous setting.  Categorical data have neither derivatives nor a Euclidean geometry, and \Cref{prop:pinning} replaces these tools with an exact combinatorial description of the score in terms of the marginals $A_\sm$.

By \cref{prop:pinning} approximating $g_t$ reduces to evaluating the partial marginals $A_\sm$.  When $q^\star$ is a Markov random field of order $d$, these marginals are sums of products of certain smaller tables which naturally leads to weight-sharing networks. A more detailed description of this construction is provided in \Cref{subsec:proof_technique}.

In our first result, we prove an uniform approximation on the bound for the discrete score.

\begin{theorem}[Score approximation]\label{thm:score}
  Suppose that $q^\star\in\MRF_d$, that $1\le d\le D$, and that Assumption~\ref{ass:flip} holds
  with constant $R\ge1$.  Let $S\ge2$ and $T>0$. For every $0<\varepsilon\le1$, denote
  \(
    \ell_\varepsilon
    \defn
    \log\!\left(
      {e\sD(SR)^{3\sD}}/{\varepsilon}
    \right).
  \)
  There exists a weight-sharing ReLU network
  \[
    s^{\mathrm{NN}}
    \in
    \bTheta_R^{\mathrm{ws}}
    (\sL,\sW,\sB,\Sws,\omega)
  \]
  with input $(\x,t)\in[S]^\sD\times(0,T]$ such that
  \begin{equation}\label{eq:score-err}
    \sup_{t\in(0,T]}
    \max_{\substack{\x\in[S]^\sD\\
                    i\in[\sD],\;c\in[S-1]}}
    \left|
      s_t^\star(\x\oplus_i c,\x)
      -
      s_t^{\mathrm{NN}}(\x\oplus_i c,\x)
    \right|
    \le
    \varepsilon.
  \end{equation}
  The network parameters satisfy
  \[
  \begin{aligned}
    \Sws
    \le
    C_d
    \left[
      \sD^dS^d
      +
      \sD^{3d+2}\ell_\varepsilon
      +
      \ell_\varepsilon^4
    \right]&,
    \;\;
    \sL
    \le
    C_d\sD^{2d+1}\ell_\varepsilon^2,
    \;\;
    \sW
    \le
    C_d\sD^{2d+2}S
    \left[
      (S+1)^{\sD-d-1}
      +
      \ell_\varepsilon^3
    \right],
    \\
    \sB
    &\le
    \exp\!\left(
      C_d\sD^d\ell_\varepsilon^2
    \right),
    \qquad
    \Lw\le\sD^{d},
  \end{aligned}
  \]
  where $C_d>0$ depends only on $d$.
\end{theorem}
\noindent
The proof of \cref{thm:score} also shows that when $D\le 2d+1$, a sparse ReLU network (see \cref{def:sparsenet}) with sparsity $\le C_d\left[\sD^dS^d+\sD^{3d+2}\ell_\varepsilon+\ell_\varepsilon^4
    \right]$ also achieves the bound in \eqref{eq:score-err}. The term $\sD^dS^d$ in \cref{thm:score} accounts for the stored MRF and marginal
tables and their evaluation parameters.  The remaining terms
in the effective sparsity account for the shared networks
used to approximate products, time-dependent coefficients,
and the reciprocal as detailed in \cref{subsec:proof_technique}. 

To pass from the uniform error to the population risk, note that, for
every $i\in[\sD]$ and $c\in[S-1]$, we have
\[
\begin{aligned}
  \bE_{\x\sim q_t}
  \left[
    s_t^\star(\x\oplus_i c,\x)
  \right]
  &=
  \sum_{\x\in[S]^\sD}
  q_t(\x)
  \frac{q_t(\x\oplus_i c)}{q_t(\x)}
  =
  \sum_{\x\in[S]^\sD}
  q_t(\x\oplus_i c)
  =
  1,
\end{aligned}
\]
where the last equality follows because
$\x\mapsto\x\oplus_i c$ is a bijection.  Therefore
\begin{equation}\label{eq:score-mass}
  \bE_{\x\sim q_t}
  \left[
    \frac1S
    \sum_{i=1}^{\sD}
    \sum_{c\in[S-1]}
    s_t^\star(\x\oplus_i c,\x)
  \right]
  =
  \frac{(S-1)\sD}{S}
  \le
  \sD.
\end{equation}
The range bound in Lemma~\ref{lem:score_bounds}, together with \eqref{eq:score-mass}, gives the corresponding approximation error in the population loss.

\begin{theorem}[Approximation error in population risk]
\label{thm:bias}
  Suppose that $q^\star\in\MRF_d$, $1\le d\le D$, and  Assumption~\ref{ass:flip} holds  with constant $R\ge1$.  Let $S\ge2$ and $T>0$.   For every $0<\varepsilon\le1$, define
  \(
    \ell_{\varepsilon,T}
    \defn
    \log\!\left(
      {e\sD(SR)^{3\sD}(1\vee T)}/{\varepsilon}
    \right).
  \)
  There exists a weight-sharing ReLU network
  \[
    s^{\mathrm{NN}}
    \in
    \bTheta_R^{\mathrm{ws}}
    (\sL,\sW,\sB,\Sws,\omega)
  \]
  with input $(\x,t)\in[S]^\sD\times(0,T]$ and outputs
  $s_t^{\mathrm{NN}}(\x\oplus_i c,\x)$,
  indexed by $i\in[\sD]$ and $c\in[S-1]$, satisfying
  \begin{small}
  \[
    \cL_t(s^{\mathrm{NN}})
    \le
    \frac{\varepsilon}{1\vee T}
    \quad\text{for every }t\in(0,T],
    \qquad
    \cL(s^{\mathrm{NN}})
    =
    \int_0^T
    \cL_t(s^{\mathrm{NN}})\,\mathrm dt
    \le
    \varepsilon.
  \]
  The network parameters satisfy
  \[
  \begin{aligned}
    \Sws
    \le
    C_d
    \left[
      \sD^dS^d
      +
      \sD^{3d+2}\ell_{\varepsilon,T}
      +
      \ell_{\varepsilon,T}^4
    \right]&,
    \;\;
    \sL
    \le
    C_d\sD^{2d+1}\ell_{\varepsilon,T}^2,
    \;\;
    \sW
    \le
    C_d\sD^{2d+2}S
    \left[
      (S+1)^{\sD-d-1}
      +
      \ell_{\varepsilon,T}^3
    \right],
    \\
    \sB
    &\le
    \exp\!\left(
      C_d\sD^d\ell_{\varepsilon,T}^2
    \right),
    \;\;
    \Lw\le\sD^{d},
  \end{aligned}
  \]
  \end{small}
  where $C_d>0$ depends only on $d$.
\end{theorem}

\subsection{Proof technique}\label{subsec:proof_technique}
We now describe the construction behind \Cref{thm:score,thm:bias}.  The complete proofs are given in Appendix~\ref{app:score_approximation}.

\paragraph{Stored and computed marginals.}
The decomposition~\eqref{eq:pinning} reduces the approximation of
$g_t$ to the partial marginals $A_\sm$.  When $|\sm|\le d$, the
marginal is a table with at most $S^d$ entries.  Under the standing condition $\sD\ge2d+2$, storing all such tables requires at most
\[
  \sum_{j=0}^d\binom{\sD}{j}S^j
  \le
  1+d \binom{\sD}{d} S^d
\]
values, which is of the same order as the number of entries in the
MRF potentials.

For $|\sm|>d$, storing the full table $A_\sm$ would require
$S^{|\sm|}$ entries.  We instead compute its value at the given input
$\x_\sm$.  Substituting the factorization~\eqref{eq:qMRF}, the
potentials involving only coordinates in $\sm$ can be evaluated
immediately.  The remaining coordinates are summed out:

\begin{equation}\label{eq:Am_exp}
  A_\sm(\x_\sm)
  =
  \prod_{\substack{\sm'\in\Fd\\ \sm'\subseteq\sm}}
  \psi_{\sm'}(\x_{\sm'})
  \times
  \sum_{\z\in[S]^{\sm^c}}
  \prod_{\substack{\sm'\in\Fd\\ \sm'\not\subseteq\sm}}
  \psi_{\sm'}
  \bigl(
    \x_{\sm'\cap\sm},
    \z_{\sm'\cap\sm^\com}
  \bigr).
\end{equation}
When at most $d$ coordinates remain to be summed, this expression can
be evaluated directly.  Otherwise, we write the sum as a sequence of
one-coordinate sums.

\paragraph{Eliminating one coordinate at a time.}
Let
\(
  \sm^\com=\{j_1<\cdots<j_r\}.
\)
Each remaining potential depends on one or more of these free
coordinates.  Assign the potential to the last such coordinate in the
ordering above. Let
\[
  \Phi_p
  \bigl(
    b,z^{j_1},\ldots,z^{j_{p-1}};\x_\sm
  \bigr)
  =
  \prod_{\substack{\sm'\in\Fd\\
                   \sm'\not\subseteq\sm\\
                   \max(\sm'\cap\sm^\com)=j_p}}
  \psi_{\sm'}
  \bigl(
    \x_{\sm'\cap\sm},
    \z_{\sm'\cap\sm^\com}
  \bigr)
  \Bigg|_{z^{j_p}=b}.
\]
Thus $\Phi_p$ is the product of the potentials assigned to $j_p$,
after setting $z^{j_p}=b$.  Every remaining potential appears in
exactly one $\Phi_p$. Starting from $h_r = 1$, define, for $p=r,\ldots,1$,
\begin{equation}\label{eq:elim_rec_main}
\begin{aligned}
  &h_{p-1}
  \bigl(
    z^{j_1},\ldots,z^{j_{p-1}};\x_\sm
  \bigr)
  :=
  \sum_{b\in[S]}
  \Phi_p
  \bigl(
    b,z^{j_1},\ldots,z^{j_{p-1}};\x_\sm
  \bigr)
  h_p
  \bigl(
    z^{j_1},\ldots,z^{j_{p-1}},b;\x_\sm
  \bigr).
\end{aligned}
\end{equation}
This is simply the original sum written as nested sums.  At $p=r$, there are no coordinates after $j_r$ to sum over, so we set $h_r\equiv1$.  Suppose that $h_p$ is the sum over $j_{p+1},\ldots,j_r$ of the potentials assigned to those coordinates. Multiplying by $\Phi_p$ inserts exactly the potentials assigned to $j_p$, and summing over $b=z^{j_p}$ eliminates that coordinate. A downward induction therefore shows that $h_{p-1}$ is the sum over the coordinates $j_p,\ldots,j_r$, together with the potentials assigned to these coordinates.  At $p=1$, all coordinates in $\sm^\com$ have been summed out, and hence
\[
  h_0(\,;\x_\sm)
  =
  \sum_{\z\in[S]^{\sm^\com}}
  \prod_{\substack{\sm'\in\Fd\\
                   \sm'\not\subseteq\sm}}
  \psi_{\sm'}
  \bigl(
    \x_{\sm'\cap\sm},
    \z_{\sm'\cap\sm^\com}
  \bigr),
\]
which is sum on the right display of \eqref{eq:Am_exp}.
The appendix writes the factors in $\Phi_p$ more explicitly and gives
the formal induction; see Proposition~\ref{prop:elim_correct} in
Appendix~\ref{app:score_approximation}.

\paragraph{Repeated use of the tables.}
At step $p$, the quantity $\Phi_p$ is evaluated at the $S^p$
assignments of
\[
  (z^{j_1},\ldots,z^{j_{p-1}},b)\in[S]^p.
\]
A table $\psi_{\sm'}$ appearing in $\Phi_p$ depends on at most $d$
of these variables.  If it depends on $\ell\le d$ of them, then only
$S^\ell\le S^d$ different entries of the table are used, and each
entry is repeated over the remaining $p-\ell$ variables.  It therefore
appears $S^{p-\ell}$ times.  Thus, for $p>d$, the recursion may use
$S^p\gg S^d$ positions even though each table contains at most $S^d$
different entries.

For example, let $\sD=6$, $d=2$, and
$\sm=\{1,2,3\}$, so that
$\sm^\com=\{4,5,6\}$.  At the step that eliminates coordinate
$6$,
\begin{equation}\label{eq:Phi_example}
\begin{aligned}
  \Phi_3(b,z^4,z^5;\x_\sm)
  ={}&
  \psi_{\{1,6\}}(x^1,b)
  \psi_{\{2,6\}}(x^2,b)
  \psi_{\{3,6\}}(x^3,b)
  \times
  \psi_{\{4,6\}}(z^4,b)
  \psi_{\{5,6\}}(z^5,b).
\end{aligned}
\end{equation}
The right-hand side is evaluated at $S^3$ triples $(z^4,z^5,b)$.
For fixed $\x_\sm$, each of the first three tables contributes only
$S$ values, each repeated $S^2$ times.  Each of the last two tables
contributes $S^2$ values, each repeated $S$ times.  Hence the $S^3$
positions in~\eqref{eq:Phi_example} are filled by repeated entries of
tables containing at most $S^2$ values.  \Cref{fig:bottleneck} illustrates this reuse.

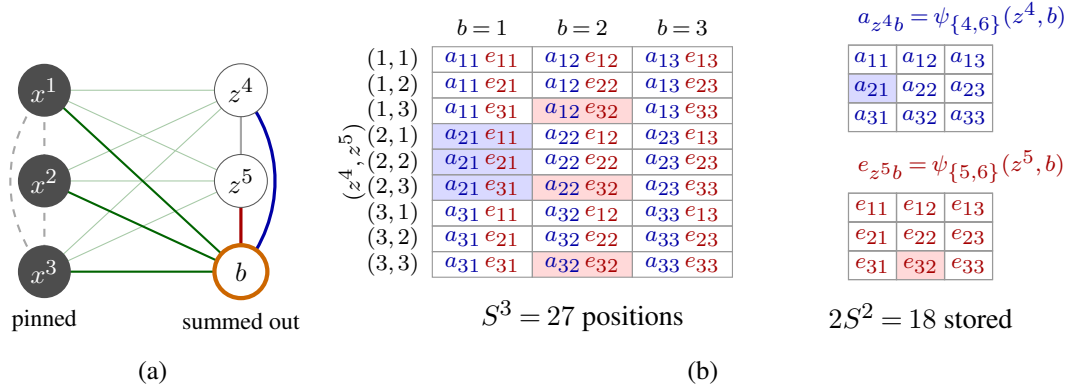
\begin{figure}[!htbp]
\centering
\begin{minipage}[b]{0.29\linewidth}
\centering
\begin{tikzpicture}[x=1cm,y=1cm,
    pv/.style={circle,fill=black!70,text=white,minimum size=7mm,inner sep=0pt,font=\small},
    fv/.style={circle,draw=black!65,fill=white,minimum size=7mm,inner sep=0pt,font=\small}]
  \coordinate (p1) at (0,2.4); \coordinate (p2) at (0,1.2); \coordinate (p3) at (0,0);
  \coordinate (f4) at (2.6,2.4); \coordinate (f5) at (2.6,1.2); \coordinate (f6) at (2.6,0);
  \draw[black!35,dashed,thick] (p1)--(p2) (p2)--(p3);
  \draw[black!35,dashed,thick] (p1) to[bend right=40] (p3);
  \foreach \i in {1,2,3} \foreach \j in {4,5} \draw[darkgreen!35] (p\i)--(f\j);
  \foreach \i in {1,2,3} \draw[darkgreen,thick] (p\i)--(f6);
  \draw[black!35,thick] (f4)--(f5);
  \draw[red!65!black,very thick] (f5)--(f6);
  \draw[blue!65!black,very thick] (f4) to[bend left=40] (f6);
  \node[pv] at (p1) {$x^1$}; \node[pv] at (p2) {$x^2$}; \node[pv] at (p3) {$x^3$};
  \node[fv] at (f4) {$z^4$}; \node[fv] at (f5) {$z^5$};
  \node[fv,draw=orange!80!black,line width=1.6pt] at (f6) {$b$};
  \node[font=\footnotesize] at (0,-0.65) {pinned};
  \node[font=\footnotesize] at (2.6,-0.65) {summed out};
\end{tikzpicture}\\[4pt]
{\small (a)}
\end{minipage}\hfill
\begin{minipage}[b]{0.69\linewidth}
\centering
\resizebox{\linewidth}{!}{%
\begin{tikzpicture}[x=1cm,y=1cm]
  \def\cw{1.25} \def\ch{0.32}
  \foreach \rw in {3,4,5} \fill[blue!15] (0,-\rw*\ch-\ch/2) rectangle (\cw,-\rw*\ch+\ch/2);
  \foreach \rw in {2,5,8} \fill[red!15] (\cw,-\rw*\ch-\ch/2) rectangle (2*\cw,-\rw*\ch+\ch/2);
  \foreach \zf in {1,2,3} {
    \foreach \zs in {1,2,3} {
      \pgfmathtruncatemacro{\rw}{(\zf-1)*3+\zs-1}
      \node[font=\scriptsize,anchor=east] at (-0.05,-\rw*\ch) {$(\zf,\zs)$};
      \foreach \bv in {1,2,3} {
        \draw[black!40] ({(\bv-1)*\cw},{-\rw*\ch-\ch/2}) rectangle ({\bv*\cw},{-\rw*\ch+\ch/2});
        \node[font=\scriptsize] at ({(\bv-0.5)*\cw},-\rw*\ch)
          {$\textcolor{blue!65!black}{a_{\zf\bv}}\,\textcolor{red!65!black}{e_{\zs\bv}}$};
      }
    }
  }
  \foreach \bv in {1,2,3} \node[font=\scriptsize] at ({(\bv-0.5)*\cw},0.4) {$b=\bv$};
  \node[font=\scriptsize,rotate=90] at (-1.0,-1.28) {$(z^4,z^5)$};
  \node[font=\small] at (1.5*\cw,-3.2) {$S^3=27$ positions};
  \def\tx{5.2} \def\tw{0.6} \def\th{0.36}
  \foreach \off/\lt/\col/\ttl in {0/a/blue!65!black/{$a_{z^4b}=\psi_{\{4,6\}}(z^4,b)$}, -1.85/e/red!65!black/{$e_{z^5b}=\psi_{\{5,6\}}(z^5,b)$}} {
    \node[font=\scriptsize,text=\col,anchor=south west] at (\tx,\off+0.2) {\ttl};
    \foreach \i in {1,2,3} \foreach \j in {1,2,3} {
      \draw[black!40] ({\tx+(\j-1)*\tw},{\off-(\i-1)*\th-\th/2}) rectangle ({\tx+\j*\tw},{\off-(\i-1)*\th+\th/2});
    }
  }
  \fill[blue!15] (\tx,-\th-\th/2) rectangle (\tx+\tw,-\th+\th/2);
  \fill[red!15] (\tx+\tw,-1.85-2*\th-\th/2) rectangle (\tx+2*\tw,-1.85-2*\th+\th/2);
  \foreach \off/\lt/\col in {0/a/blue!65!black, -1.85/e/red!65!black} {
    \foreach \i in {1,2,3} \foreach \j in {1,2,3} {
      \draw[black!40] ({\tx+(\j-1)*\tw},{\off-(\i-1)*\th-\th/2}) rectangle ({\tx+\j*\tw},{\off-(\i-1)*\th+\th/2});
      \node[font=\scriptsize,text=\col] at ({\tx+(\j-0.5)*\tw},{\off-(\i-1)*\th}) {$\lt_{\i\j}$};
    }
  }
  \node[font=\small] at (\tx+1.5*\tw,-3.2) {$2S^2=18$ stored};
\end{tikzpicture}}\\[4pt]
{\small (b)}
\end{minipage}
\caption{The bottleneck in the example~\eqref{eq:Phi_example}, with $\sD=6$, $d=2$, $\sm=\{1,2,3\}$, and $S=3$.  (a) Summing out the sixth coordinate $b=z^6$ involves the three tables linking it to the pinned coordinates (green), which carry the input $\x_\sm$, and the tables $\psi_{\{4,6\}}$ (blue) and $\psi_{\{5,6\}}$ (red), which do not depend on the input.  Dashed edges join pinned coordinates and are evaluated directly, and $\psi_{\{4,5\}}$ (grey) enters only at the next step.  (b) The step touches the $S^3=27$ positions $(z^4,z^5,b)$, and the weight at each position is the product $a_{z^4b}e_{z^5b}$ of an entry of $\psi_{\{4,6\}}$ and an entry of $\psi_{\{5,6\}}$.  The shaded entry $a_{21}$ fills three positions, and so does $e_{32}$, so the $2S^2=18$ stored numbers fill all $27$ positions.}
\label{fig:bottleneck}
\end{figure}

\paragraph{From the recursion to the score network.}
The network follows the recursion~\eqref{eq:elim_rec_main} --- products
in $\Phi_p$ are approximated by multiplication networks, and the sum
over $b\in[S]$ is a linear layer.  The final value $h_0$ gives the sum
in~\eqref{eq:Am_exp}, and hence $A_\sm(\x_\sm)$.  Combining these
marginals in~\eqref{eq:pinning} gives $g_t(\x)$.  The same construction
is applied to $\x$ and $\x\oplus_i c$, after which a reciprocal and
multiplication network approximates
\[
  \frac{g_t(\x\oplus_i c)}{g_t(\x)}
  =
  s_t^\star(\x\oplus_i c,\x).
\]
Weight sharing represents the repetition illustrated in
\eqref{eq:Phi_example} where the entries of each $S^2$ table are stored
once and reused across the $S^3$ positions in which they occur.  More
generally, it reuses at most $S^d$ table entries across the $S^p$
positions of step $p$ with $1\le p\le D$.

\subsection{Score estimation}\label{subsec:score_estimation}

The previous subsection constructs a score with small population risk in the network class. We now bound the risk of the empirical minimizer of~\eqref{eq:emp_score_loss} in terms of the best risk in this class and an error arising from finite sample approximation. Throughout this subsection, we take $\sR=R$ (i.e., the truncation range of the weight-sharing network is equal to the bound on the score from \cref{ass:flip}).

Let $\x_0\sim q^\star$ and $\xi\sim\operatorname{Unif}(0,T]$ be independent. For $s^{\btheta}\in\bTheta_R^{\mathrm{ws}}$, define
\[
  \Delta\ell_{\btheta}(\x_0,\xi)
  \defn
  T\left\{
    \ell_\xi(s^{\btheta};\x_0)
    -\ell_\xi(s^\star;\x_0)
  \right\}.
\]
The reference term $\ell_\xi(s^\star;\x_0)$ does not depend on the candidate score, so subtracting it leaves the empirical minimizer unchanged. We impose the following conditions on these loss differences.

\begin{assumption}[Regularity of the training loss]
\label{ass:training_loss}
  There are constants $B_\ell,V_\ell,L_\ell>0$ such that,
  for every $s^{\btheta},s^{\btheta'}\in\bTheta_R^{\mathrm{ws}}$,
  \[
    \bE\bigl[\Delta\ell_{\btheta}(\x_0,\xi)\bigr]
    =\cL(s^{\btheta}),
    \qquad
        \bE\bigl[\Delta\ell_{\btheta}(\x_0,\xi)^2\bigr]
    \le V_\ell\,\cL(s^{\btheta}),
    \qquad
    \bigl|\Delta\ell_{\btheta}(\x_0,\xi)\bigr|
    \le B_\ell,
  \]
  where the absolute bound holds for every
  $\x_0\in[S]^\sD$ and $\xi\in(0,T]$.
  For every such pair $(\x_0,\xi)$, we also have
  \[
  \begin{aligned}
    &\left|
      \Delta\ell_{\btheta}(\x_0,\xi)
      -\Delta\ell_{\btheta'}(\x_0,\xi)
    \right|
    \le
    L_\ell \!
    \sup_{\substack{t\in(0,T],\;\x\in[S]^\sD\\
                    i\in[\sD],\;c\in[S],\;c\neq S}}
    \left|
      s_t^{\btheta}(\x\oplus_i c,\x)
      -s_t^{\btheta'}(\x\oplus_i c,\x)
    \right|.
  \end{aligned}
  \]
\end{assumption}

The mean identity follows from~\eqref{eq:loss_mean_identity}
and $\cL(s^\star)=0$. The uniform and second-moment bounds
control fluctuations of the empirical loss; the second-moment
bound becomes tighter for scores with smaller population risk.
The Lipschitz bound transfers uniform approximations of the
score functions to their loss differences. Below we verify these conditions for the DSE and ISE losses
introduced in Section~\ref{subsec:optimization}.
For DSE, averaging over the noisy state controls the
conditional-score terms even as $t$ approaches zero.
Together with the bounds on the true and estimated scores,
this allows bounds uniform over the training times.
\begin{proposition}[DSE and ISE]\label{prop:dse_ise_estimation}
  Suppose Assumption~\ref{ass:flip} holds. The DSE and ISE
  losses defined in Section~\ref{subsec:optimization}, using
  exact conditional expectations, satisfy
  Assumption~\ref{ass:training_loss} with
  \[
    B_\ell=3T\sD R,
    \qquad
    V_\ell=8T\sD R^4,
    \qquad
    L_\ell=2T\sD R.
  \]
\end{proposition}

\paragraph{Entropy of the search class.} To control the empirical loss uniformly over the search class, we approximate the class by a finite collection of score networks in the supremum norm. For $n\ge2$, Lemma~\ref{lem:covering} bounds the logarithm of the number of networks needed at accuracy of order $1/n$ by a constant multiple of
\begin{equation}\label{eq:est_entropy}
  \Lambda_n
  \defn
  3\Sws\sL
  \log\!\bigl(
    8n(S\vee T)(\sB\vee1)(\Lw\vee1)\sL\sW
  \bigr).
\end{equation}
Note that the factor $\Sws$ counts the trainable coordinates in the sharing map.

\begin{theorem}[Oracle inequality]\label{thm:score_est}
  Suppose Assumption~\ref{ass:training_loss} holds and $n\ge2$. Let $\widehat s$ minimize the empirical
  loss~\eqref{eq:emp_score_loss} over
  $\bTheta_R^{\mathrm{ws}}$.
  Then, for every $x\ge1$, with probability at least
  $1-e^{-x}$ over the training observations and sampled times,
  \[
    \cL(\widehat s)
    \le
    2\inf_{s\in\bTheta_R^{\mathrm{ws}}}\cL(s)
    +
    C\frac{(B_\ell+V_\ell)(\Lambda_n+x)}{n},
  \]
  where $C>0$ is an absolute constant.
\end{theorem}

The infimum measures the approximation error of the search
class, while the second term bounds the estimation error.
Under Assumption~\ref{ass:flip},
Proposition~\ref{prop:dse_ise_estimation} verifies the
required loss conditions for DSE and ISE.
Combining these loss bounds with the network bounds in
Theorem~\ref{thm:bias} yields the following rate.

\begin{corollary}[Score-estimation rate]
\label{cor:score_est_rate}
  Suppose $q^\star\in\MRF_d$, $1\le d\le D$ and  Assumption~\ref{ass:flip} holds with constant $R\ge1$.
  Let $n\ge2$, and define
  \[
    \ell_n
    \defn
    \log\!\left(e\sD(SR)^{3\sD}(1\vee T)n\right).
  \]
  Choose $\bTheta_R^{\mathrm{ws}}$ as in
  Theorem~\ref{thm:bias} with $\varepsilon=1/n$, and let
  $\widehat s$ minimize the empirical
  loss~\eqref{eq:emp_score_loss} over this class,
  using either DSE or ISE.
  Then, for every $x\ge1$, with probability at least $1-e^{-x}$,
  \begin{equation}\label{eq:score_est_rate}
    \cL(\widehat s)
    \le
    C_d\frac{T\sD R^4}{n}
    \left[
      \sD^{3d+1}\ell_n^4
      \left\{
        \sD^dS^d
        +
        \sD^{3d+2}\ell_n
        +
        \ell_n^4
      \right\}
      +
      x
    \right],
  \end{equation}
  where $C_d>0$ depends only on $d$.
\end{corollary}

For fixed $d$, $\sD$, $R$, and $T$, this bound implies
\[
  \cL(\widehat s)
  =
  \widetilde O_{\mathbb P}\!\left(\frac{S^d}{n}\right),
\]
where $\widetilde O_{\mathbb P}$ suppresses logarithmic
factors in $n$ and $S$. The interaction order $d$ determines
the power of the alphabet size $S$ in the rate.
The full bound also retains the dependence on the ambient
dimension $\sD$, which is only polynomial for fixed $d$.

\subsection{Sampling guarantees}\label{subsec:efficient_sampling}

We now combine the score-estimation bound with Algorithm~\ref{alg:tau_leaping} to obtain a guarantee for the generated distribution.

\begin{theorem}[End-to-end sampling rate]
\label{thm:efficient_sampling}
  Suppose that $q^\star\in\MRF_d$, $1\le d\le\sD$, and
  Assumption~\ref{ass:flip} holds with constant $R\ge1$.
  Let $S\ge2$, $n\ge2$, and $T>0$, and let $\widehat s$ be
  the DSE or ISE estimator of Corollary~\ref{cor:score_est_rate}.
  Write
  \[
    \ell_n
    \defn
    \log\!\left(e\sD(SR)^{3\sD}(1\vee T)n\right).
  \]
  For a grid $0=t_0<\cdots<t_N=T$, set
  $\Delta=\max_{0\le k<N}(t_{k+1}-t_k)$, and let $p_T$ be
  the output law in~\eqref{eq:finalaw}, conditional on the fitted
  score and the chosen grid.

  For every $x\ge1$, with probability at least $1-e^{-x}$ over
  the training observations and training times, the bound
  \begin{equation}\label{eq:end_to_end_sampling}
  \begin{aligned}
    \mathrm{KL}(q^\star\,\|\,p_T)
    \le{}&
    C_d\frac{T\sD R^4}{n}
    \left[
      \sD^{3d+1}\ell_n^4
      \left\{
        \sD^dS^d+\sD^{3d+2}\ell_n+\ell_n^4
      \right\}
      +x
    \right]
    \\
    &+e^{-T}\sD\log S
    +\Delta\left(\frac32T\sD^2R+\sD\log R\right)
  \end{aligned}
  \end{equation}
  holds simultaneously for every such grid, where $C_d$ depends
  only on $d$.   In particular, for fixed $d$, $\sD$, and $R$, choosing
  \(
    T=\log(n\sD\log S),
    \Delta\le n^{-1},
  \)
  gives, on the same event,
  \[
    \mathrm{KL}(q^\star\,\|\,p_T)
    =\widetilde O\!\left(\frac{S^d+x}{n}\right),
  \]
  where $\widetilde O$ suppresses logarithmic factors in $n$ and $S$.
\end{theorem}
The first term in \eqref{eq:end_to_end_sampling} is the statistical error term which bounds the integrated score risk $\cL(\widehat s)$ of the network learned from $n$ observations. The initialization term $e^{-T}\sD\log S$ accounts for starting the sampler from the uniform distribution instead of the unknown $q_T$. Training controls the integrated score error over the whole time interval, while the sampler evaluates the score using the state $\z_k$ recorded at the beginning of each step, as in~\eqref{eq:tau_rate} and~\eqref{eq:evolve}. The term $3T\sD^2R\Delta/2$ accounts for this difference. The remaining $\sD\log R\,\Delta$ term bounds the discretization error in the entropy decrease along the true reverse process. This decrease represents a gain in information relative to the uniform distribution.  We prove Theorem~\ref{thm:efficient_sampling} using this decomposition, established in Proposition~\ref{prop:sampling_error} (Appendix~\ref{app:efficient_sampling}).

The theorem establishes a sampling guarantee from $n$ observations, with dependence $S^d$ on the vocabulary size up to logarithmic factors and polynomial dependence on $\sD$ for fixed $d$, $R$, and $T$. The $\tau$-leaping guarantees of \citet{ren2024stochastic,dmitriev2026sharp} assume bounds on the score error at the points of a prescribed sampling grid. Here we control the integrated score risk of the estimated score from training data and use it to bound the error of the generated distribution. The same estimated score satisfies the sampling guarantee for every grid chosen after training and before generating samples. The grid need not match the sampled training times and can be refined without retraining the network.

\section{Numerical Experiments}\label{sec:numerics}

We examine the score estimation and sampling performance of weight-sharing score networks on categorical distributions with local dependence. We vary the number of training samples and the training and sampling grids to study their effects on score estimation and sample generation.

% \subsection{Experimental design}\label{subsec:num_design}
\paragraph{Models and methods.} We use Potts chains, frustrated Ising ladders, and tree-structured categorical Markov random fields, each at two dimensions. These are pairwise graphical models \citep{wainwright2008graphical}, so $d=2$ throughout. This allows us to compare performance at smaller and larger dimensions while keeping the interaction order fixed. Table~\ref{tab:num_models} lists the dimensions and vocabulary sizes. Appendix~\ref{app:num_targets} specifies the target distributions.

\begin{table}[htpb!]
  \centering
  \caption{Target distributions used in the experiments.
  Here $\sD$ is the dimension, $S$ is the vocabulary size,
  and $d$ is the interaction order.}
  \label{tab:num_models}
  \begin{tabular}{lccc@{\qquad}ccc}
    \toprule
    & \multicolumn{3}{c}{Smaller dimension}
    & \multicolumn{3}{c}{Larger dimension}\\
    \cmidrule(lr){2-4}\cmidrule(lr){5-7}
    Model & $\sD$ & $S$ & $d$ & $\sD$ & $S$ & $d$\\
    \midrule
    Potts chain       & 8 & 3 & 2 & 32 & 3 & 2\\
    Ising ladder      & 8 & 2 & 2 & 32 & 2 & 2\\
    Categorical tree  & 7 & 3 & 2 & 32 & 3 & 2\\
    \bottomrule
  \end{tabular}
\end{table}

We compare three weight-sharing score networks, LSE-WSNN, Conv-WSNN, and Butterfly-WSNN, with a fully connected score network denoted MLP. LSE-WSNN uses log-sum-exp operations and shares local factor coefficients across states with the same local configuration. Conv-WSNN shares convolutional filters across positions \citep{lecun1998gradient}. Butterfly-WSNN uses butterfly connectivity \citep{dao2019learning} and applies the same transformation to each pair of coordinate features within a layer. MLP has three hidden layers of width 128. All four methods estimate the same score ratios and are trained by DSE. Architecture details and parameter counts are provided in Appendix~\ref{app:num_arch}.

\paragraph{Training and evaluation.} For each target, we train the four methods on $n\in\{2048,8192\}$ observations and set the terminal time to $T=6$. We vary the number of training-time nodes $K\in\{64,128,256\}$. Each estimated score is then used for sampling by time-dependent $\tau$-leaping (\Cref{alg:tau_leaping}) on quadratically spaced grids with $N\in\{32,64,128,256\}$ reverse-time intervals. Varying $K$ requires separate training, whereas varying $N$ uses the same estimated score without retraining. The training and sampling procedures are specified in Appendix~\ref{app:num_protocol}. We repeat each experiment ten times. Within each repetition, the methods use the same training observations, noising draws, and held-out evaluations. We report means and standard errors across repetitions.

Score accuracy is measured by the normalized integrated risk $\cL(\widehat s)/(T\sD)$, estimated on held-out observations using the exact target score. Sampling accuracy is measured by squared maximum mean discrepancy (MMD) with a Hamming kernel \citep{gretton2012kernel}, mean edge-pair total variation, and the Wasserstein distance $W_1$ between energy distributions \citep{peyre2019computational}. These comparisons use 512 generated observations and 8192 reference observations, except for MMD, which uses 512 observations from each distribution. The unbiased estimator of squared MMD can take negative values. Evaluation details are provided in Appendix~\ref{app:num_metrics}. The sampling metrics are defined in~\eqref{eq:num_edge_tv}, \eqref{eq:num_energy_w1}, and~\eqref{eq:num_mmd}.

\subsection{Score and distributional accuracy}
\label{subsec:num_accuracy}
Tables~\ref{tab:num_high} and~\ref{tab:num_low} report results for $n=8192$ and $K=N=256$. Additional metrics and comparisons across sample sizes, training grids, and sampling grids are provided in Appendix~\ref{app:num_grids}.

\begin{table}[htpb!]
\centering\footnotesize\setlength{\tabcolsep}{4pt}
\caption{Larger examples ($\sD=32$) at $n=8192$ and $K=N=256$. Entries are mean $\pm$ SE \\ over ten repetitions. Bold marks the smallest mean for each target and metric.}\label{tab:num_high}
\begin{tabular}{lrrrr}\toprule
Method & Risk $\times10^3 \downarrow$ & Edge TV $\downarrow$ & MMD$^2\times10^3\downarrow$ & $W_1(H)\downarrow$\\\midrule
\csname @@input\endcsname numerical_tables/main_high.tex \bottomrule
\end{tabular}
\end{table}

\begin{table}[htpb!]
\centering\footnotesize\setlength{\tabcolsep}{4pt}
\caption{Smaller examples at $n=8192$ and $K=N=256$. Entries are mean $\pm$ SE \\ over ten repetitions. Bold marks the smallest mean for each target and metric.}\label{tab:num_low}
\begin{tabular}{lrrrr}\toprule
Method & Risk $\times10^3 \downarrow$ & Edge TV $\downarrow$ & MMD$^2\times10^3 \downarrow$ & $W_1(H) \downarrow$\\\midrule
\csname @@input\endcsname numerical_tables/main_low.tex \bottomrule
\end{tabular}
\end{table}

Table~\ref{tab:num_high} shows that, at $\sD=32$, all three weight-sharing networks achieve lower mean score risk and edge-pair TV than MLP across the three targets. Conv achieves the lowest mean values for all four metrics on the Potts chain, while LSE does so on the Ising ladder and categorical tree. The comparison is less uniform for squared MMD. Butterfly has similar mean squared MMD to MLP on Potts and Ising, while Conv has higher mean squared MMD than MLP on Ising.

Table~\ref{tab:num_low} reports the smaller examples. Conv again achieves the lowest mean values for all four metrics on the Potts chain, and LSE does so on the categorical tree. On the Ising ladder, MLP has the lowest mean values for all three sampling metrics, while LSE has the lowest score risk. Across these examples, the advantage of weight sharing is more pronounced at $\sD=32$, while MLP remains competitive at smaller dimensions.

\subsection{Training settings and sampling resolution}\label{subsec:num_grids}

\paragraph{Training grids.} Table~\ref{tab:num_training_grids} compares score risk across training grids for $\sD=32$ and $n=8192$. The number of optimization updates and the number of sampled times per update are held fixed in this comparison. The ordering of the four methods is unchanged across $K\in\{64,128,256\}$, and most risk values vary little. Conv on the Potts chain shows the largest change, with mean risk increasing from $4.836\times10^{-3}$ at $K=64$ to $6.714\times10^{-3}$ at $K=256$. Thus, increasing $K$ does not consistently improve score estimation.

\begin{table}[htpb!]
\centering\small
\setlength{\tabcolsep}{5pt}
\renewcommand{\arraystretch}{1.12}
% \caption{Estimated normalized score risk $\cL(\widehat s)/(T\sD)$,
% multiplied by $10^3$, for $\sD=32$ and $n=8192$. Entries are mean $\pm$ SE
% across ten repetitions. Bold denotes the smallest mean in each
% row.}
\caption{Score risk ($\times10^3$) across training grids
for $\sD=32$ and $n=8192$.
Entries are mean $\pm$ SE \\ over ten repetitions.
Bold marks the smallest mean in each row.}
\label{tab:num_training_grids}
\begin{tabular}{rrrrr}
\toprule
$K$ & LSE & Conv & Butterfly & MLP \\
\midrule
\multicolumn{5}{l}{\textit{Potts chain}} \\[2pt]
64 & $10.756\,\pm\,0.137$ & $\mathbf{4.836\,\pm\,0.153}$ & $53.005\,\pm\,0.892$ & $56.128\,\pm\,0.352$ \\
128 & $10.708\,\pm\,0.117$ & $\mathbf{5.225\,\pm\,0.165}$ & $53.107\,\pm\,0.820$ & $56.231\,\pm\,0.421$ \\
256 & $10.786\,\pm\,0.113$ & $\mathbf{6.714\,\pm\,0.913}$ & $52.858\,\pm\,0.915$ & $56.175\,\pm\,0.423$ \\
\addlinespace[6pt]
\multicolumn{5}{l}{\textit{Ising ladder}} \\[2pt]
64 & $\mathbf{1.388\,\pm\,0.022}$ & $9.958\,\pm\,0.553$ & $13.325\,\pm\,0.561$ & $19.597\,\pm\,0.446$ \\
128 & $\mathbf{1.358\,\pm\,0.031}$ & $9.821\,\pm\,0.429$ & $13.405\,\pm\,0.734$ & $19.936\,\pm\,0.361$ \\
256 & $\mathbf{1.379\,\pm\,0.016}$ & $10.073\,\pm\,0.643$ & $13.238\,\pm\,0.519$ & $19.469\,\pm\,0.332$ \\
\addlinespace[6pt]
\multicolumn{5}{l}{\textit{Categorical tree}} \\[2pt]
64 & $\mathbf{1.384\,\pm\,0.024}$ & $15.856\,\pm\,0.245$ & $15.288\,\pm\,0.170$ & $17.407\,\pm\,0.132$ \\
128 & $\mathbf{1.386\,\pm\,0.030}$ & $15.915\,\pm\,0.273$ & $15.220\,\pm\,0.157$ & $17.411\,\pm\,0.121$ \\
256 & $\mathbf{1.356\,\pm\,0.033}$ & $15.882\,\pm\,0.238$ & $15.260\,\pm\,0.152$ & $17.354\,\pm\,0.113$ \\
\addlinespace[6pt]
\bottomrule
\end{tabular}
\end{table}

\paragraph{Sampling grids.}  Figure~\ref{fig:num_sampling} compares sampling grids using the same estimated scores, with $\sD=32$, $n=8192$, and $K=256$. For each method, the curves show $\overline E(N)-\overline E(32)$, where $\overline E(N)$ is the mean error across ten repetitions. The top row shows squared MMD multiplied by $10^3$, and the bottom row shows mean edge-pair TV. Negative values indicate improvement over $N=32$.

Increasing $N$ from 32 to 256 reduces mean edge-pair TV for all four networks on all three targets. Mean squared MMD decreases in nine of these twelve comparisons, with Conv, Butterfly, and MLP on the Potts chain as the exceptions. The MMD changes are less consistent across intermediate grids, and several error bars include zero. These results show that sampling-grid refinement can improve generation without retraining, although the improvement depends on the evaluation metric. The change from $N=32$ to $N=256$ requires eight times as many reverse-time intervals and score evaluations.

\begin{figure}[htpb!]
\centering\includegraphics[width=0.8\linewidth]{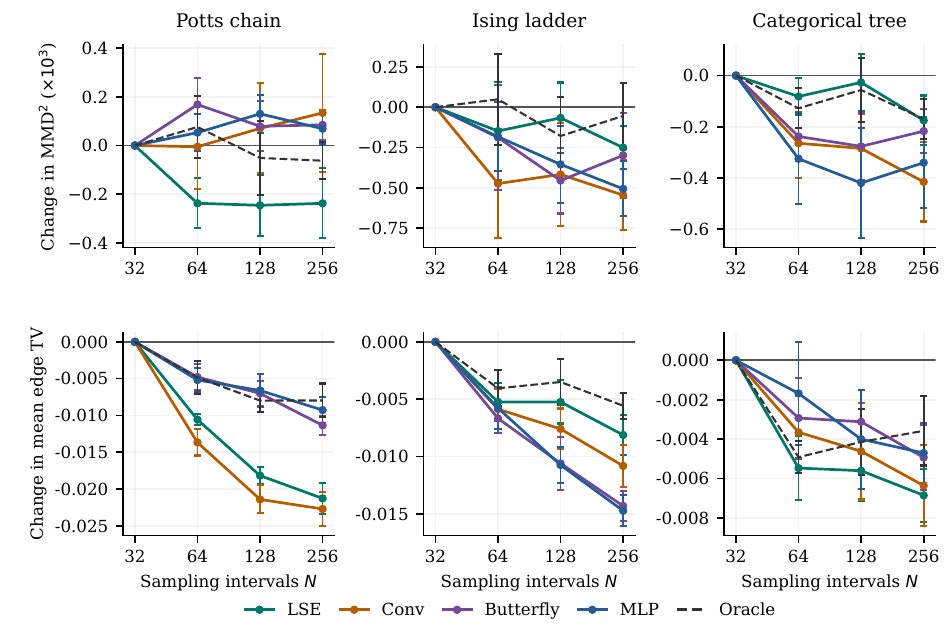}
% \caption{Effect of sampling-grid refinement for $\sD=32$, with $n=8192$ and $K=256$. Each curve shows $\overline E(N)-\overline E(32)$, where $E$ is squared Hamming MMD (top, multiplied by $10^3$) or mean edge TV (bottom). Negative values indicate improvement relative to $N=32$; the vertical scales differ across panels. Error bars are SE of the within-repetition differences across ten repetitions. Dashed curves use the exact target score in the same $\tau$-leaping sampler. The horizontal axis is logarithmic.}
\caption{Changes from $N=32$ in squared MMD
(top, $\times10^3$) and mean edge-pair TV (bottom),
with $\sD=32$, $n=8192$, and $K=256$.
Negative values indicate improvement.
Error bars show SE of paired differences over ten repetitions.
Dashed curves use the exact target score.
Vertical scales differ across panels.}
\label{fig:num_sampling}
\end{figure}

\paragraph{Summary.}
Overall, the experiments show a clearer advantage of weight sharing in the larger examples, while MLP remains competitive at smaller dimensions. They also show that the same estimated score can be used on different sampling grids, with finer grids improving edge-pair accuracy without additional training.

\bibliographystyle{imsart-nameyear}
\bibliography{references}

\newpage
\appendix

\begin{center}
{\large\bfseries APPENDIX}
\end{center}
\addcontentsline{toc}{section}{Appendix}
\vspace{0.5em}

\appendix

The appendices are organized as follows.
Appendix~\ref{app:numerical} gives the targets, architectures, protocol, metrics, and additional results of the numerical experiments.
Appendix~\ref{app:prelim} collects standard lemmas on the forward kernel and the true score.
Appendix~\ref{app:poisson} gives the Poisson representation of single-coordinate jump chains used in the sampling analysis.
Appendix~\ref{app:efficient_sampling} proves the sampling bound of \Cref{thm:efficient_sampling}.
Appendix~\ref{sec:proof_estimation} proves the oracle inequality and the score-estimation rate.
Appendix~\ref{app:nn_auxiliary} collects auxiliary ReLU network constructions and covering bounds.
Appendix~\ref{app:score_approximation} proves the score-approximation results, \Cref{thm:score,thm:bias}.

\section{Details of the numerical experiments}\label{app:numerical}

\subsection{Target distributions}\label{app:num_targets}
The six target distributions have the form
\begin{equation}\label{eq:num_target}
 q^\star(\x)
 =
 Z^{-1}\exp\left\{
 \sum_{i=1}^{\sD}h_i(x^i)
 +
 \sum_{\{i,j\}\in E}J_{ij}(x^i,x^j)
 \right\}.
\end{equation}
All interactions involve at most two coordinates.
The implementation uses labels $0,\ldots,S-1$,
which are relabeled as $1,\ldots,S$ here.

\begin{table}[htbp]\centering\small
\caption{Dimensions, graph sizes, and numbers $p$ of target-feature coefficients. All examples have interaction order $d=2$.}\label{tab:num_targets}
\begin{tabular}{llrrrr}\toprule
Target & Regime & $\sD$ & $S$ & $|E|$ & $p$\\\midrule
Potts chain & Smaller &8&3&7&7\\
Ising ladder & Smaller &8&2&10&18\\
Categorical tree & Smaller &7&3&6&38\\
Potts chain & Larger &32&3&31&31\\
Ising ladder & Larger &32&2&46&78\\
Categorical tree & Larger &32&3&31&188\\\bottomrule
\end{tabular}\end{table}

\paragraph{Potts chains.}
The graph has edges $\{i,i+1\}$, $1\le i<\sD$, and
\[
 q^\star(\x)\propto
 \exp\left\{\sum_{i=1}^{\sD-1}J_i\mathbf1\{x^i=x^{i+1}\}\right\},
 \qquad J_i\in[1.8,2.4].
\]
The couplings are drawn once from the uniform distribution on this interval. There are no node fields. This gives local agreement with invariance under a common permutation of the categories. The larger realization preserves the seed-101 construction used in the original experiment, including the intervening random-number draw after each coupling; the saved coefficient arrays specify the realization exactly.

\paragraph{Frustrated Ising ladders.}
Let $\sigma_i=2x^i-3\in\{-1,1\}$. The graph is a two-row ladder with $\sD/2$ columns, and
\[
 q^\star(\x)\propto
 \exp\left\{\sum_{\{i,j\}\in E}J_{ij}\sigma_i\sigma_j\right\}.
\]
The magnitudes are drawn uniformly from $[0.6,1.2]$. Alternating horizontal edges in the top row have negative couplings; the other edges have positive couplings. Thus two of the three elementary plaquettes are frustrated in the smaller ladder, and eight of fifteen in the larger ladder. A plaquette is frustrated when the product of its coupling signs is negative. Target node fields are zero, but the fitted factor family allows them. Frustration changes the preferred configurations without changing the pairwise Markov property.

\paragraph{Categorical trees.}
Let $C\in\bR^{S\times(S-1)}$ have orthonormal columns orthogonal to the constant vector. It is obtained by a QR factorization of the first $S-1$ columns of $I_S-S^{-1}\mathbf1\mathbf1^\top$. The log density, up to normalization, is
\[
 \sum_i\sum_{a=1}^{S-1}u_{ia}C_{x^i,a}
 +\sum_{\{i,j\}\in E}\sum_{a,b=1}^{S-1}v_{ijab}C_{x^i,a}C_{x^j,b}.
\]
Edge coefficients are drawn independently from $N(0,0.85^2)$; node coefficients are drawn from the same distribution and multiplied by $0.3$. This represents general pairwise categorical interactions, rather than only category agreement. The smaller tree has edges $(1,2),(1,3),(2,4),(2,5),(3,6),(6,7)$. For the larger tree, each additional vertex attaches to one preceding vertex under a fixed random construction. Target coefficients are held fixed across repetitions. Samples and noised scores are evaluated exactly by tree messages, or by grouping ladder columns into a chain. The larger targets therefore do not require an MCMC reference sample.

\paragraph{Relation to the model assumptions.}
Every target is strictly positive. For the log density $F(\x)=\sum_i h_i(x^i)+\sum_{\{i,j\}}J_{ij}(x^i,x^j)$, the constant
\[
 R=\exp\left\{\max_{\x,i,c}
 |F(\x\oplus_i c)-F(\x)|\right\}<\infty
\]
satisfies the bounded neighboring-ratio condition. However, pairwise Gibbs factorization alone does not imply the paper's stronger convention that the already normalized mass function is a product of factors each bounded by one.

For completeness, the saved Ising targets exhibit this distinction. Write a hypothetical binary pair factor as
\[
 \log\psi_{ij}(\sigma_i,\sigma_j)
 =a_{ij}+b_{ij}\sigma_i+c_{ij}\sigma_j+d_{ij}\sigma_i\sigma_j.
\]
Equality of the product to the Ising mass function forces $d_{ij}=J_{ij}$ on edges, $d_{ij}=0$ on nonedges, and $\sum_{i<j}a_{ij}=-\log Z$, by uniqueness of the binary multilinear expansion. Averaging the two configurations with $\sigma_i\sigma_j=\operatorname{sign}(d_{ij})$ shows that
$\max\log\psi_{ij}\ge a_{ij}+|d_{ij}|$. Thus $\psi_{ij}\le1$ for every pair would require
\[
 \log Z\ge\sum_{\{i,j\}\in E}|J_{ij}|.
\]
The saved smaller and larger ladders instead have, respectively,
\[
 (\log Z,\sum|J_{ij}|)=(9.3966,10.2703),
 \qquad (37.5626,41.8942).
\]
They are pairwise MRFs in the usual Gibbs sense, but do not satisfy that additional normalized-factor restriction. The stated value $d=2$ describes their interaction order; it is not a claim that all assumptions of the current theory hold for these targets.

\subsection{Score architectures and parameter sharing}\label{app:num_arch}
Write $F_{\bv}(\x)=A(\x)^\top\bv$ for the target's feature family, with $p$ coefficients as in Table~\ref{tab:num_targets}. All fitted coefficients are learned from the clean training observations. In particular, knowing the candidate graph does not reveal its realized coefficients.

\paragraph{LSE-WSNN.}
Let $P_t(a,b)=e^{-t}\mathbf1\{a=b\}+(1-e^{-t})/S$. With trainable $\bv_0,\bv_1,\bv_2\in\bR^p$ and $a_\ell=\log(1+e^{r_\ell})>0$, define
\begin{align*}
 \lambda_0(\x)&=A(\x)^\top\bv_0,\\
 \lambda_\ell(\x,t)&=
 \log\sum_{\z\in[S]^{\sD}}
 \left\{\prod_{i=1}^{\sD}P_{a_\ell t}(x^i,z^i)\right\}
 e^{\lambda_{\ell-1}(\z,t)}+A(\x)^\top\bv_\ell,
 \quad \ell=1,2,\\
 \widehat s_t(\x\oplus_i c,\x)
 &=\exp\{\lambda_2(\x\oplus_i c,t)-\lambda_2(\x,t)\}.
\end{align*}
A factor coefficient is reused across all global states containing its local assignment. Different graph edges have different coefficients. Each $a_\ell$ is shared over coordinates, and the same parameters are used at every time. There are $3p+2$ trainable parameters. Small examples use full-state contractions; larger examples evaluate the same recursion by exact elimination on hidden states $(z_{0i},z_{1i})\in[S]^2$, with ladder sites grouped by columns. These are two evaluators of the same model, not different fitted architectures.

The family contains the exact noised target score: setting $\bv_0$ to the target coefficients, $\bv_1=\bv_2=0$, and $a_1+a_2=1$ gives the uniform semigroup applied to the unnormalized target mass. This is a representability statement, not a training initialization. The actual initialization is $\bv_0=\bv_1=\bv_2=0$ and $a_1=a_2\simeq1/2$. Exact log-sum-exp operations are used throughout. A finite ReLU approximation of these operations and its complexity bounds are not implemented in this experiment.

\paragraph{Common neural features.}
Conv, Butterfly, and MLP receive the one-hot state and twelve time features: $t/T$ and eleven normalized ReLU hinges with knots geometrically spaced from $0.001T$ to $0.9T$. Conv and Butterfly also receive eight fixed position features, consisting of $r_i,r_i^2$ and $\sin(\pi\omega r_i),\cos(\pi\omega r_i)$ for $\omega\in\{1,2,4\}$ and equally spaced $r_i\in[-1,1]$. Raw outputs are mapped to positive scores by saturating to $[-8,8]$ and linearly interpolating the exponential at 129 equally spaced knots. This fixed continuous piecewise-affine map has a finite ReLU realization and range $[e^{-8},e^8]$.

\paragraph{Conv-WSNN and Butterfly-WSNN.}
Both have a shared affine stem with 32 channels and a shared affine head with $S-1$ outputs per site. Conv has residual width-three convolutions with dilations $1,2,4,\ldots$, stopped when their sum reaches $\sD-1$. Each block applies ReLU after division of the residual sum by $\sqrt2$. There are three blocks in the smaller examples and five in the larger examples. Filters are shared across positions, with separate parameters at different layers and zero padding at boundaries.

Butterfly pads the sequence to the next power of two, applies a fixed permutation, and uses one pair-mixing layer for each binary index bit. Each layer applies the same residual affine map on $\bR^{64}$ to every pair of 32-channel site vectors, followed by division by $\sqrt2$ and ReLU. Layers have distinct parameters. The permutation is inverted and padded outputs are discarded. This routing allows global communication without supplying the MRF graph.

For both models, unrolling the computation repeats the same learned coefficients in multiple affine positions. The corresponding coefficient-placement map is Lipschitz with constant at most one in the maximum norm used in the WSNN definition. This concerns parameter placement, not a uniform input-Lipschitz bound for the fitted network. No parameter-box projection or spectral-norm constraint is enforced during training. LSE also has no enforced common output range matching the neural positive map. The empirical models are therefore not asserted to implement every restriction of the theoretical search class.

\paragraph{MLP.}
The comparator concatenates the $\sD S$ one-hot entries with the twelve time features and uses three fully connected ReLU layers of width 128. Its $\sD(S-1)$ outputs use the same positive map. No coefficients are tied across site positions. This is the model labeled \texttt{MLP-score} in the archived results.

\begin{table}[htbp]\centering\small\setlength{\tabcolsep}{5pt}
\caption{Trainable parameter counts. Parameter budgets are not matched.}\label{tab:num_parameters}
\begin{tabular}{lrrrr}\toprule
Example & LSE & Conv & Butterfly & MLP\\\midrule
\csname @@input\endcsname numerical_tables/parameters.tex \bottomrule
\end{tabular}\end{table}

\subsection{Training and sampling protocol}\label{app:num_protocol}
For uniform noising, let $a_{t,i,c}(\x\mid\x_0)$ denote the known conditional score ratio. The implemented objective is the normalized, centered DSE
\begin{equation}\label{eq:num_dse}
 L_{\mathrm{DSE}}(\widehat s)
 =\frac1{T\sD}\int_0^T\bE_{\x_0}\bE_{\x\mid\x_0}
 \frac1S\sum_{i,c}
 \left[\widehat s_t(\x\oplus_i c,\x)-1
 -a_{t,i,c}(\x\mid\x_0)\log\widehat s_t(\x\oplus_i c,\x)\right]dt.
\end{equation}
Subtracting one per directed move makes the loss of the constant score one equal to zero. The loss can therefore be negative. It is not a negative log-likelihood; subtracting its population value at the true score yields $\cL(\widehat s)/(T\sD)$.

Set $u_j=(j+1/2)/K$, $\tau_j=Tu_j^2$, and $w_j=2u_j$, for $0\le j<K$. These weights follow from $dt/T=2u\,du$. At each update, 512 clean observations are drawn with replacement from the training set. Four training-grid indices are sampled independently and uniformly; each is assigned to a block of 128 observations. Every observation is then corrupted independently at its assigned time. The update averages
\[
 \frac{w_j}{S\sD}\sum_{i,c}
 \left[\widehat s_{\tau_j}-1-a_{\tau_j,i,c}\log\widehat s_{\tau_j}\right].
\]
Its expectation, conditional on the training data, is the midpoint-grid approximation to~\eqref{eq:num_dse}. It is not exactly the continuous-time loss, nor does it retain the exact conditional expectation in the theoretical estimator. Increasing $K$ refines the available quadrature nodes without increasing the four time draws per update.

The neural hidden layers use Kaiming-normal initialization, with zero biases and a zero output head, so that the initial score is one. Each model receives 200 DSE warm-up updates at $K=128$. Each subsequent branch starts from its own model's warm weights and resets Adam before 800 additional updates. The learning rate is $10^{-3}$, Adam moments are $(0.9,0.999)$, the stabilizer is $10^{-8}$, and gradient norms are clipped at 5. There is no weight decay, learning-rate schedule, early stopping, or checkpoint selection. The final iterate is evaluated. Equal warm-up budgets do not imply equal warm-up accuracy.

All samplers start from the uniform law. In the paper's reverse-time notation, the grid is
\[
 t_j=T\{1-(1-j/N)^2\},\qquad 0\le j\le N.
\]
In both dimensional regimes, the state is frozen during each interval and the time-dependent intensities are approximated by three-point Gauss--Legendre quadrature. Independent Poisson draws then give the modular updates. This requires $3N$ batched score evaluations and $N\sD(S-1)$ Poisson draws per generated observation.

For the supplementary coordinate-splitting evaluation at smaller dimension, scores are evaluated at the midpoint of each interval for every state. For each coordinate and fixed setting of the remaining coordinates, an $S\times S$ reverse generator is formed. Its matrix exponential is applied to the probability vector, successively over coordinates. This computes the full law of a coordinate-splitting approximation. It is neither exact simulation of the continuous-time reverse chain nor the modular Poisson sampler used in the main comparison. Full-law computations use double precision; learned forward computations mainly use single precision. Small probability denominators use a $10^{-30}$ numerical safeguard.

For each target and repetition, smaller training sets are prefixes of the largest set. Architecture-specific initialization seeds are distinct, while private draw generators give the methods the same minibatch indices, training-time indices, and forward corruptions. There are 180 dataset tasks, 720 warm-up fits, and 2160 final fits. Every recorded fit reached its assigned update count. Sessions 1, 3, and 4 used vectorized T4 execution; Session 2 used sequential A100 execution. Vectorization keeps separate parameters, Adam moments, losses, and gradient clipping for each fit. The available end-to-end runtime logs are incomplete, so we do not report a hardware speedup or a complete runtime ranking.

The low-dimensional $\tau$-leaping evaluation reuses the final fitted networks and the original held-out score results. It covers $n\in\{2048,8192\}$, all three training grids, all four sampling grids, and all ten repetitions. The original low-dimensional experiments additionally include $n\in\{128,512\}$; their score evaluations and coordinate-splitting results are retained as supplementary results.

\subsection{Evaluation and uncertainty}\label{app:num_metrics}
Each repetition has 256 independent clean held-out observations and 16 fixed midpoint nodes in the transformed time coordinate. There is one corrupted observation per clean observation and node. The same evaluation bank is used across methods and grids. The known target score is used only for evaluation. The score-risk estimator approximates $\cL(\widehat s)/(T\sD)$ by averaging
\[
 \frac{w_j}{S\sD}\sum_{i,c}
 \left[\widehat s_{\tau_j}-s_{\tau_j}^\star
 -s_{\tau_j}^\star\log(\widehat s_{\tau_j}/s_{\tau_j}^\star)\right].
\]
Score MSE averages the weighted squared score difference over $\sD(S-1)$ outputs; log-score RMSE is the square root of the corresponding weighted log-score MSE. Centered DSE is averaged over time for each clean observation before calculating its held-out mean and observation-level SD. These observation-level quantities include corruption noise and differ from variability across the ten fitted repetitions.

For a generated law $p$ and target $q^\star$, one-site and edge-pair TV are
\[
 \operatorname{TV}_i=\frac12\sum_a|p_i(a)-q_i^\star(a)|,
 \qquad
 \operatorname{TV}_{ij}=\frac12\sum_{a,b}|p_{ij}(a,b)-q_{ij}^\star(a,b)|.
\]
The mean edge-pair total variation is
\begin{equation}\label{eq:num_edge_tv}
 \operatorname{TV}_{\mathrm{edge}}(p,q^\star)
 :=
 \frac{1}{|E|}
 \sum_{\{i,j\}\in E}\operatorname{TV}_{ij}.
\end{equation}
We retain means and maxima over sites or true edges. For $\tau$-leaping at both dimensions, these quantities
are estimated from empirical marginals using 512 generated
and 8192 reference observations. The same sample counts are used for edge agreement and the observable Wasserstein distances.

% Energy Wasserstein distance is the one-dimensional $W_1$ between the laws of $H(\x)=-\sum_i h_i(x^i)-\sum_{\{i,j\}}J_{ij}(x^i,x^j)$.

We compare the distributions of the energy
\[
 H(\x)
 =
 -\sum_i h_i(x^i)
 -\sum_{\{i,j\}\in E}J_{ij}(x^i,x^j)
\]
using the one-dimensional Wasserstein distance \citep{peyre2019computational},
\begin{equation}\label{eq:num_energy_w1}
 W_1(H)
 :=
 \int_{\mathbb R}
 \left|F_{H,p}(u)-F_{H,q^\star}(u)\right|\,\mathrm du.
\end{equation}
Here $F_{H,p}$ and $F_{H,q^\star}$ are the distribution functions of $H(\x)$ under $p$ and $q^\star$, respectively.

The order parameter is
\[
 m(\x)=\frac{S\max_a\sum_i\mathbf1\{x^i=a\}/\sD-1}{S-1}.
\]
Its distribution is also compared by $W_1$. At $S=2$ this is absolute magnetization; signed magnetization is recorded separately. Wasserstein distances between energy distributions are
compared within each target, since the energy scale
changes across targets.

The categorical kernel is
\[
 k(\x,\y)=\frac14\sum_{\gamma\in\{0.5,1,2,4\}}
 \exp\{-\gamma\Ham(\x,\y)/\sD\}.
\]
For two independent samples $X_1,\ldots,X_m\sim p$ and
$Y_1,\ldots,Y_r\sim q^\star$, the unbiased estimator
of squared MMD \citep{gretton2012kernel} is
\begin{equation}\label{eq:num_mmd}
 \widehat{\operatorname{MMD}}_{\mathrm u}^{\,2}
 :=
 \frac{\sum_{a\ne b}k(X_a,X_b)}{m(m-1)}
 +\frac{\sum_{a\ne b}k(Y_a,Y_b)}{r(r-1)}
 -\frac{2\sum_{a,b}k(X_a,Y_b)}{mr}.
\end{equation}
 Both samples are capped at 512. Negative estimates are retained. This sample-based calculation is used at both dimensions, including when full laws are available. Thus low-dimensional MMD remains noisy even when joint TV is evaluated without empirical histograms.

For the supplementary coordinate-splitting evaluation of the smaller examples, joint TV, forward KL, Jensen--Shannon divergence and squared Hellinger distance are
\begin{align*}
 \operatorname{TV}(q^\star,p)&=\tfrac12\sum_{\x}|q^\star(\x)-p(\x)|,\qquad
 \operatorname{KL}(q^\star\|p)=\sum_{\x}q^\star(\x)\log\frac{q^\star(\x)}{p(\x)},\\
 \operatorname{JS}(q^\star,p)&=\tfrac12\operatorname{KL}(q^\star\|m)
 +\tfrac12\operatorname{KL}(p\|m),\qquad m=(q^\star+p)/2,\\
 H^2(q^\star,p)&=\tfrac12\sum_{\x}(\sqrt{q^\star(\x)}-\sqrt{p(\x)})^2.
\end{align*}
Natural logarithms are used. Here $p$ is the full law of the numerical splitting sampler. These are discrepancies for the coordinate-splitting law; they are not joint discrepancies for the $\tau$-leaping sampler. No empirical high-dimensional joint KL is substituted for these quantities.

For ten repetition values $z_r$, tables show $\bar z\pm s/\sqrt{10}$, where $s$ is the sample SD. Paired differences are formed within a repetition before calculating their means and pointwise 95\% confidence intervals based on the
$t$ distribution with nine degrees of freedom. These intervals describe the fixed target realizations and are not adjusted for multiple comparisons. We do not treat different sampling grids as independent fitted repetitions. 
% The complete saved metrics are supplied in the accompanying data file; likelihood bounds and classifier two-sample scores were not evaluated.

\subsection{Additional comparisons and grid effects}\label{app:num_grids}

Table~\ref{tab:num_additional} reports additional metrics
at $n=8192$ and $K=N=256$.
Table~\ref{tab:num_paired} reports paired differences
from MLP, and Table~\ref{tab:num_oracle} reports results
using the exact target score.
Sample-based metrics use $\tau$-leaping.
Joint TV, forward KL, Jensen--Shannon divergence, and
squared Hellinger distance are computed from the
coordinate-splitting law for the smaller examples.

\begingroup\scriptsize\setlength{\tabcolsep}{3pt}
\begin{longtable}{lrrrr}
% \caption{Additional metrics at $n=8192$, $K=256$ and, for sample metrics, $N=256$. Entries are mean $\pm$ SE; bold denotes the smallest mean for accuracy metrics. Observation-level DSE SD is descriptive and is not ranked.}\label{tab:num_additional}\\
\caption{Additional metrics at $n=8192$, $K=256$ and, for sample metrics, $N=256$. Sample-based quantities use $\tau$-leaping; the four joint discrepancies use the low-dimensional coordinate-splitting law. Entries are mean $\pm$ SE; bold denotes the smallest mean for accuracy metrics. Observation-level DSE SD is descriptive and is not ranked.} \label{tab:num_additional}\\
\toprule Metric & LSE & Conv & Butterfly & MLP\\\midrule\endfirsthead
\multicolumn{5}{l}{\textit{Table \thetable\ continued}}\\\toprule Metric & LSE & Conv & Butterfly & MLP\\\midrule\endhead
\midrule\multicolumn{5}{r}{\textit{Continued on next page}}\\\endfoot\bottomrule\endlastfoot
\csname @@input\endcsname numerical_tables/additional.tex 
\end{longtable}\endgroup

\begingroup\scriptsize\setlength{\tabcolsep}{5pt}
\begin{longtable}{lrr}
\caption{Paired differences from MLP for $\tau$-leaping at $n=8192$, $K=N=256$: mean difference [pointwise 95\% interval]. Negative values favor the shared model. MMD-squared differences are multiplied by $10^3$.}\label{tab:num_paired}\\
\toprule Method & Edge TV difference & MMD$^2$ difference $\times10^3$\\\midrule\endfirsthead
\multicolumn{3}{l}{\textit{Table \thetable\ continued}}\\\toprule Method & Edge TV difference & MMD$^2$ difference $\times10^3$\\\midrule\endhead
\midrule\multicolumn{3}{r}{\textit{Continued on next page}}\\\endfoot\bottomrule\endlastfoot
\csname @@input\endcsname numerical_tables/paired.tex 
\end{longtable}\endgroup

\begin{table}[htbp]\centering\footnotesize\setlength{\tabcolsep}{4pt}
\caption{Oracle-score evaluation at $N=256$, using the $n=8192$ records. Joint TV is evaluated on the full coordinate-splitting law and is available only for the smaller examples. All other columns use $\tau$-leaping, with 512 generated observations, 8192 reference observations for edge TV and energy $W_1$, and 512 reference observations for MMD. Sample-based entries are mean $\pm$ SE across ten repetitions; deterministic joint TV is shown without SE.}\label{tab:num_oracle}
\begin{tabular}{lrrrr}\toprule
Example & Joint TV & Edge TV & MMD$^2\times10^3$ & $W_1(H)$\\\midrule
\csname @@input\endcsname numerical_tables/oracle.tex \bottomrule
\end{tabular}\end{table}

For the sample-based metrics, the oracle uses the exact target score with the same $\tau$-leaping scheme and sampling grid. The metrics capture different discrepancies. For the categorical tree with $\sD=32$, LSE has near-zero estimated squared MMD, while its Wasserstein distance between energy distributions is $1.305$, compared with $0.197$ for the oracle.

Figure~\ref{fig:num_riskn} reports score risk over every original training sample size. Figure~\ref{fig:num_mmdn} reports $\tau$-leaping MMD for $n\in\{2048,8192\}$ at both dimensions. Figure~\ref{fig:num_training} gives MMD at every training grid; Figure~\ref{fig:num_samplinglow} complements Figure~\ref{fig:num_sampling} with the smaller examples under the same $\tau$-leaping scheme. The MMD curves are retained even when their uncertainty prevents a clear ordering. Since update counts are fixed, neither $n$ nor $K$ alone determines optimization accuracy. The original low-dimensional MMD curves, including $n\in\{128,512\}$, are retained separately in Figure~\ref{fig:num_splitting_n}.

\begin{figure}[!htbp]
\centering
\includegraphics[width=\linewidth]{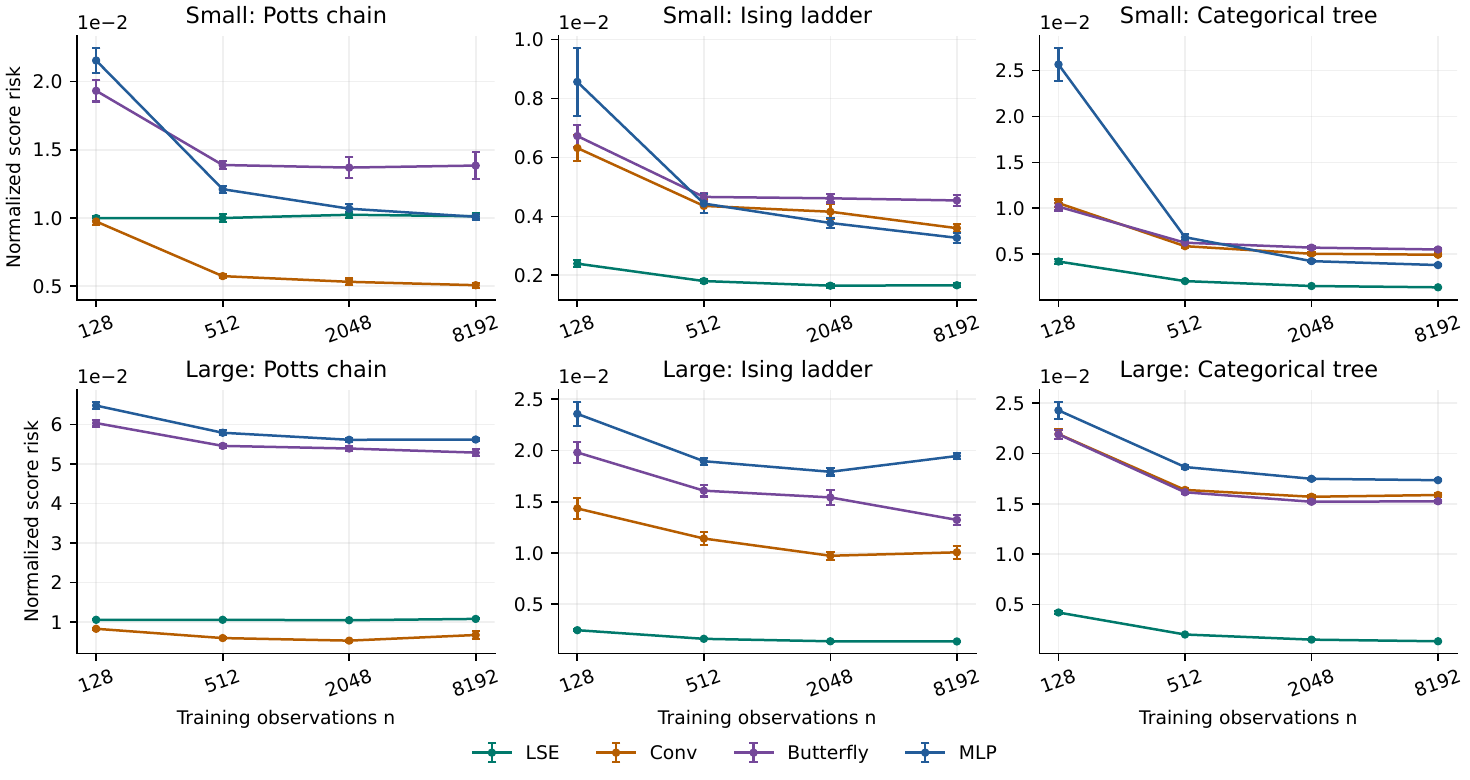}
\caption{Normalized score risk versus sample size, with $K=256$. Error bars are SE across ten repetitions. Panels have separate vertical scales.}
\label{fig:num_riskn}
\end{figure}

\begin{figure}[!htbp]
\centering\includegraphics[width=\linewidth]{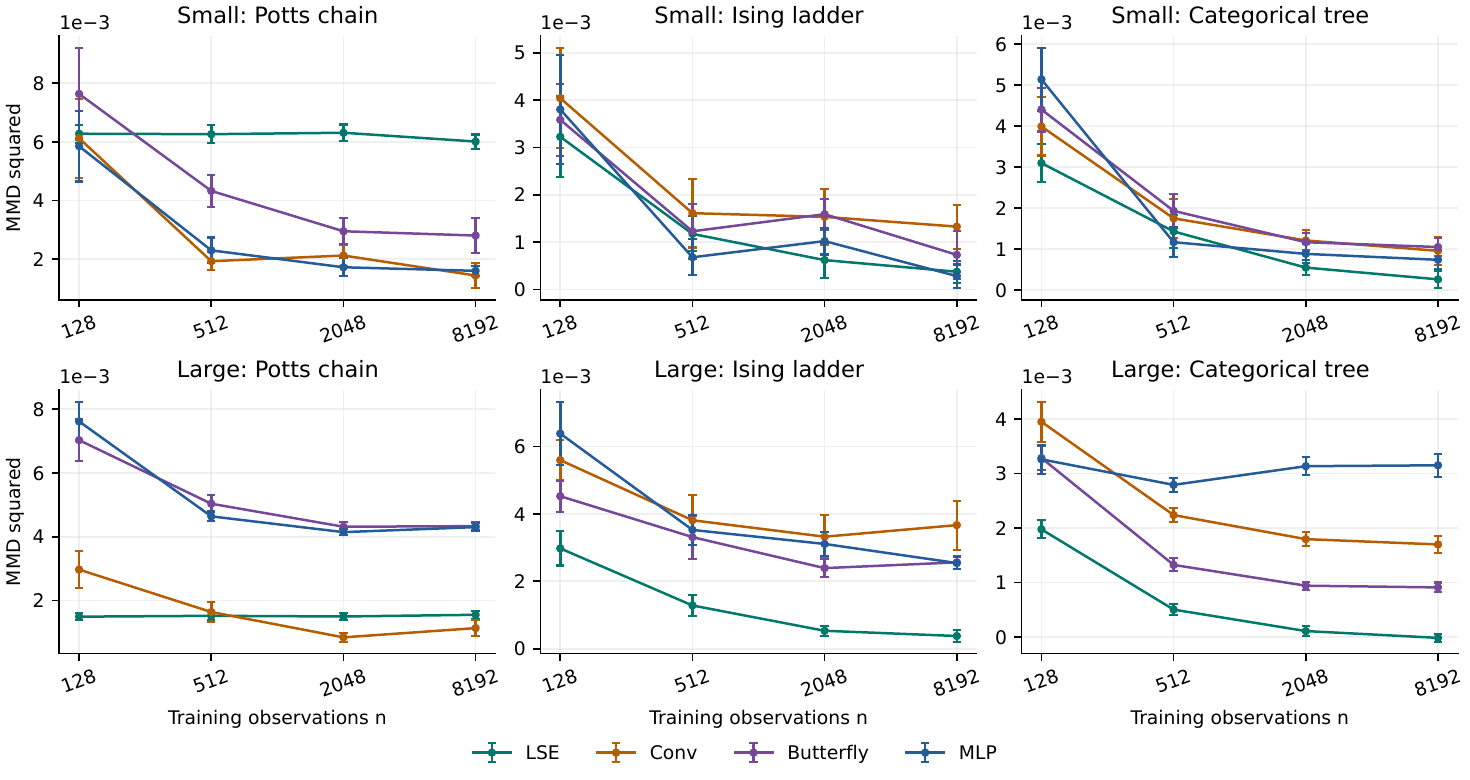}
\caption{Unbiased squared Hamming MMD versus sample size, with $K=N=256$. Both dimensions use $\tau$-leaping at $n\in\{2048,8192\}$. Error bars are SE across ten repetitions; negative estimates are retained.}
\label{fig:num_mmdn}
\end{figure}

\begin{figure}[!htbp]
\centering\includegraphics[width=\linewidth]{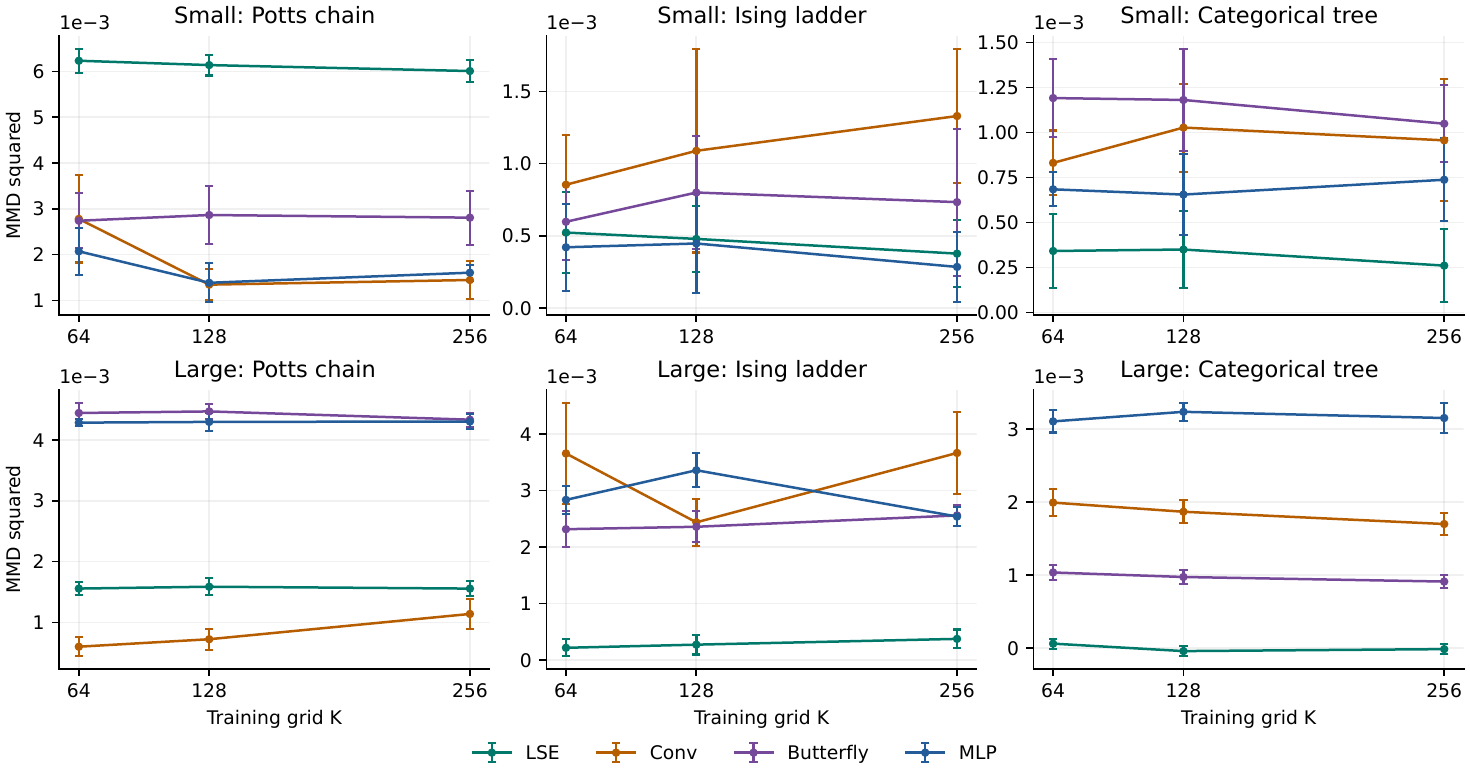}
\caption{Unbiased squared Hamming MMD versus the training-grid size $K$, with $n=8192$ and $N=256$. Both dimensions use $\tau$-leaping. Every branch has the same refinement budget. Error bars are SE.}\label{fig:num_training}
\end{figure}

\begin{figure}[!htbp]
\centering\includegraphics[width=\linewidth]{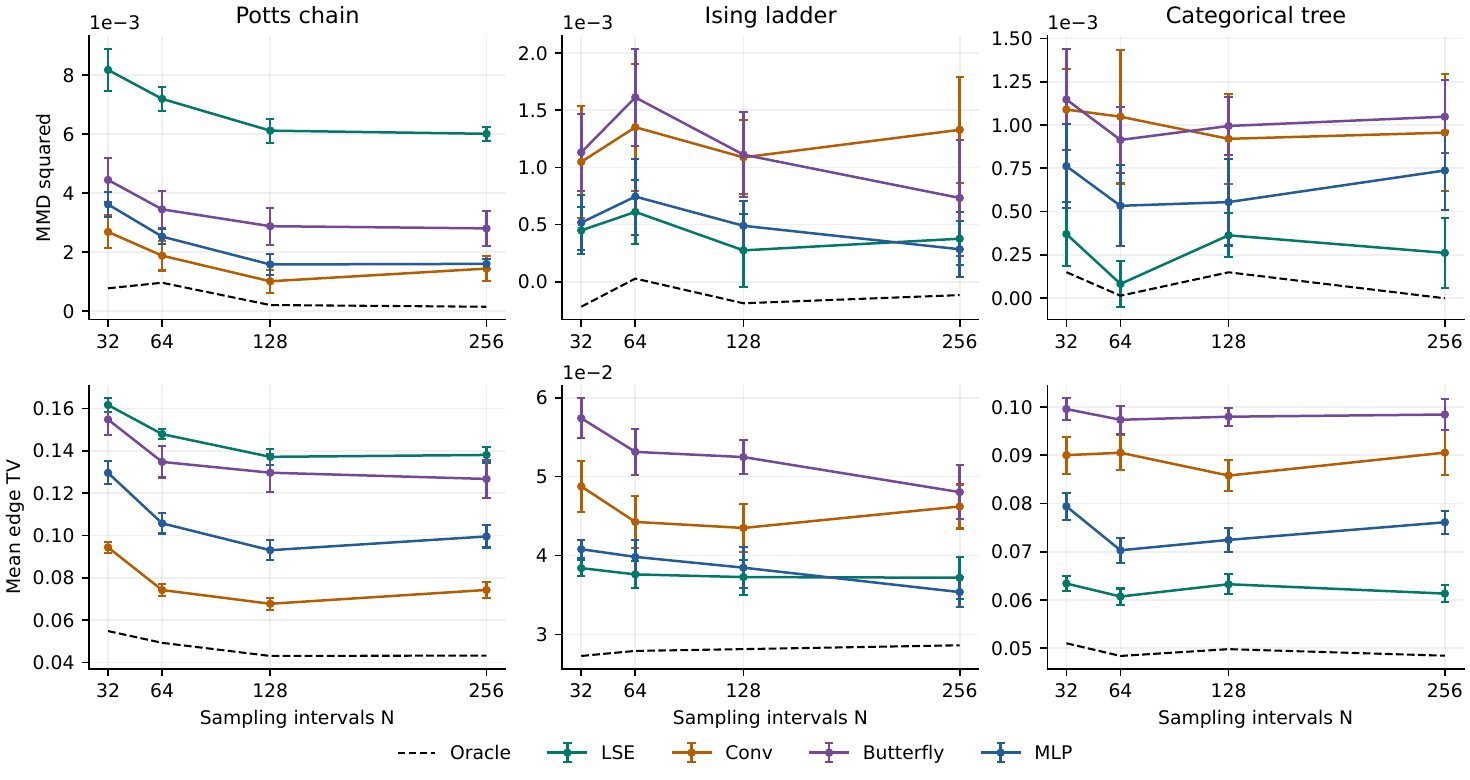}
\caption{Sampling-grid refinement for the smaller examples, with $n=8192$, $K=256$, using $\tau$-leaping. Top: squared MMD from 512 observations per distribution. Bottom: mean edge TV from 512 generated and 8192 reference observations. Error bars are SE; dashed curves use the exact score in the same $\tau$-leaping sampler.}\label{fig:num_samplinglow}
\end{figure}

\begin{figure}[!htbp]
\centering
\includegraphics[width=\linewidth]{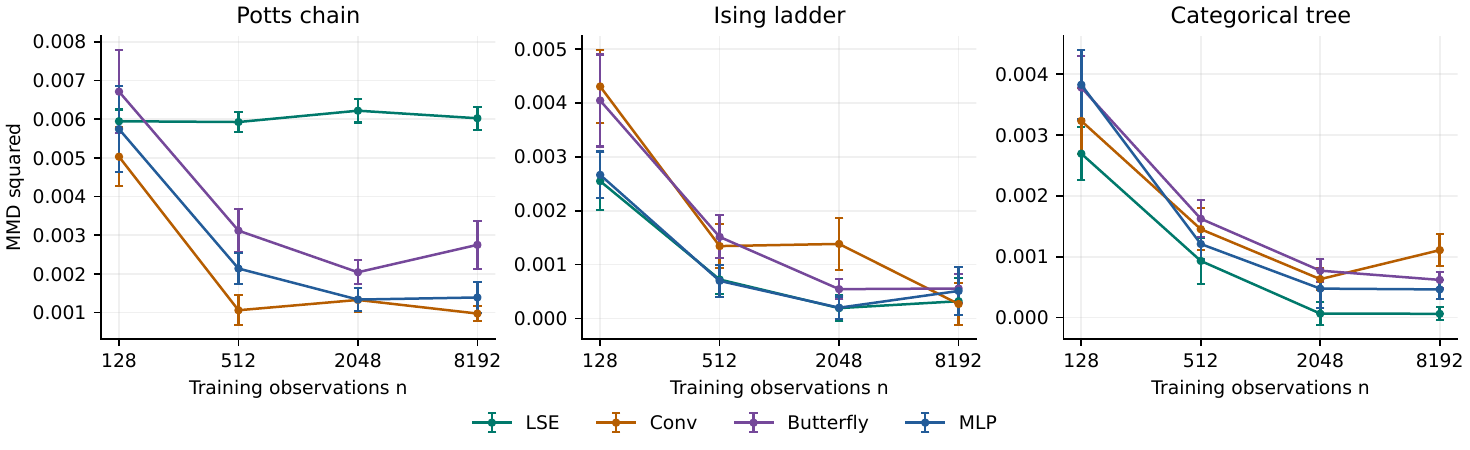}
\caption{Supplementary MMD curves for the original low-dimensional coordinate-splitting sampler, with $K=N=256$ and $n\in\{128,512,2048,8192\}$. Each MMD estimate uses 512 observations from each distribution. Error bars are SE across ten repetitions.}
\label{fig:num_splitting_n}
\end{figure}

\clearpage
\section{Preliminary lemmas}\label{app:prelim}

This appendix collects standard facts about the forward kernel and the true score that are used in \Cref{sec:problem,sec:theory}.

\begin{lemma}[Factorization of the conditional kernel]\label{lem:kernel_factor}
  Let $Q$ be given by~\eqref{eq:Q_cases}.  Then $e^{tQ}=\bigotimes_{i=1}^{\sD}e^{tQ^{\tok}}$ for every $t\ge0$, and the conditional kernel factorizes as in~\eqref{eq:qt0_factor}.
\end{lemma}

\begin{proof}
In the product basis $e_{x^1}\otimes\cdots\otimes e_{x^\sD}$,
where $e_a$ is the $a$-th standard basis vector of $\bR^S$,
the rate matrix has the Kronecker-sum representation
\[
  Q
  =
  \bigoplus_{i=1}^{\sD}Q^{\tok}
  \defn
  \sum_{i=1}^{\sD}
  \underbrace{\Id_S\otimes\cdots\otimes\Id_S}_{i-1}
  \otimes Q^{\tok}\otimes
  \underbrace{\Id_S\otimes\cdots\otimes\Id_S}_{\sD-i}.
\]
Indeed, the $i$-th summand has entries
\[
  \bigl[
    \Id_S^{\otimes(i-1)}
    \otimes Q^{\tok}
    \otimes\Id_S^{\otimes(\sD-i)}
  \bigr]_{\x,\y}
  =
  Q^{\tok}(x^i,y^i)
  \prod_{j\neq i}\mathbf{1}\{x^j=y^j\}.
\]
The identity factors require all coordinates other than $i$
to agree, giving exactly the contribution of coordinate $i$
to~\eqref{eq:Q_cases}.
The summands act on different tensor factors and therefore
commute, which gives
$e^{tQ}=\bigotimes_{i=1}^{\sD}e^{tQ^{\tok}}$.
Taking the $(\x_0,\y)$ entry of this tensor product and using~\eqref{eq:cond_kernel} yields~\eqref{eq:qt0_factor}.
\end{proof}

\begin{lemma}[Idempotent collapse]\label{lem:idempotent}
  Let $\cV$ be a finite set and $Q^{\tok}\in\bR^{|\cV|\times|\cV|}$ a rate matrix on $\cV$.  If $\mathsf{R}\defn-Q^{\tok}$ is idempotent, that is $\mathsf{R}^2=\mathsf{R}$, then
  $\mathsf{R}^n=\mathsf{R}$ for all $n\ge1$ and
  \begin{equation}\label{eq:exp_idempotent}
    e^{tQ^{\tok}}
    =
    e^{-t\mathsf{R}}
    =
    \Id+(e^{-t}-1)\mathsf{R}.
  \end{equation}
\end{lemma}

\begin{proof}
  Idempotency for all $n\ge2$ follows by induction, since
  $\mathsf{R}^{n+1}
   =\mathsf{R}^n\mathsf{R}
   =\mathsf{R}^2
   =\mathsf{R}$.
  Hence
  \[
    e^{-t\mathsf{R}}
    =
    \Id+
    \sum_{n=1}^{\infty}
    \frac{(-t)^n}{n!}\mathsf{R}^n
    =
    \Id+(e^{-t}-1)\mathsf{R}.
  \]
\end{proof}

\begin{lemma}[Uniform noising kernel]\label{lem:uniform_kernel}
  Let $Q^{\tok}=\mathsf{P}-\Id_S$ with $\mathsf{P}\defn\frac1S\bOne\bOne^\top$.  Then
  \begin{equation}\label{eq:exp_uniform}
    e^{tQ^{\tok}}
    =
    e^{-t}\Id_S+(1-e^{-t})\mathsf{P},
  \end{equation}
  so that
  \begin{equation}\label{eq:tok_kernel_uniform_cases}
    q_{t|0}^{\tok}(b\mid a)
    =
    \begin{cases}
      \dfrac{1+(S-1)e^{-t}}{S},
        & \text{if } b=a,\\[6pt]
      \dfrac{1-e^{-t}}{S},
        & \text{if } b\neq a,
    \end{cases}
  \end{equation}
  and
  \begin{equation}\label{eq:qt0_hamming_product}
    q_{t|0}(\y\mid\x_0)
    =
    \biggl(
      \frac{1+(S-1)e^{-t}}{S}
    \biggr)^{\!\sD-\Ham(\y,\x_0)}
    \biggl(
      \frac{1-e^{-t}}{S}
    \biggr)^{\!\Ham(\y,\x_0)}.
  \end{equation}
\end{lemma}

\begin{proof}
Since
\[
  \mathsf{P}^2
  =
  \frac{1}{S^2}
  \bOne(\bOne^\top\bOne)\bOne^\top
  =
  \frac{1}{S}\bOne\bOne^\top
  =
  \mathsf{P},
\]
the matrix $\mathsf{R}=-Q^{\tok}=\Id_S-\mathsf{P}$ satisfies
\[
  \mathsf{R}^2
  =
  \Id_S-2\mathsf{P}+\mathsf{P}^2
  =
  \Id_S-\mathsf{P}
  =
  \mathsf{R}.
\]
\Cref{lem:idempotent} then gives
\[
  e^{tQ^{\tok}}
  =
  \Id_S+(e^{-t}-1)(\Id_S-\mathsf{P})
  =
  e^{-t}\Id_S+(1-e^{-t})\mathsf{P},
\]
which is~\eqref{eq:exp_uniform}.  Reading off entries via $[\mathsf{P}]_{a,b}=1/S$ gives~\eqref{eq:tok_kernel_uniform_cases}.  As a check, $q_{t|0}^{\tok}(a\mid a)+(S-1)q_{t|0}^{\tok}(b\mid a)=1$, the kernel is the identity as $t\to0^+$, and it converges to $1/S$ as $t\to\infty$.  Substituting~\eqref{eq:tok_kernel_uniform_cases} into~\eqref{eq:qt0_factor} and grouping the $\sD-\Ham(\y,\x_0)$ matching coordinates and the $\Ham(\y,\x_0)$ differing ones gives~\eqref{eq:qt0_hamming_product}.  Dividing and multiplying by $\bigl(\frac{1+(S-1)e^{-t}}{S}\bigr)^{\Ham(\y,\x_0)}$ gives the form~\eqref{eq:qt0_hamming}.
\end{proof}

\begin{lemma}[Conditioning the score]\label{lem:cond_score}
  Let $0<t\le T$.  For any function $f$,
  \[
    \bE_{\x\sim q_t}
    \left[
      \sum_{\y\neq\x}
      s_t^\star(\y,\x)f(\x,\y,t)
    \right]
    =
    \bE_{\x_0\sim q^\star}
    \bE_{\x\sim q_{t|0}(\cdot\mid\x_0)}
    \left[
      \sum_{\y\neq\x}
      s_{t|0}^\star(\y,\x\mid\x_0)f(\x,\y,t)
    \right],
  \]
  where the \emph{conditional score} is
  \[
    s_{t|0}^\star(\y,\x\mid\x_0)
    =
    \frac{q_{t|0}(\y\mid\x_0)}
         {q_{t|0}(\x\mid\x_0)}.
  \]
\end{lemma}
\begin{proof}
  We have
  \begin{align*}
    \bE_{\x\sim q_t}\Bigl[\sum_{\y\neq\x} s_t^{\star}(\y,\x)\,f(\x,\y,t)\Bigr]
    &= \sum_{\x} q_t(\x)\sum_{\y\neq\x}\frac{q_t(\y)}{q_t(\x)}\,f(\x,\y,t)
     = \sum_{\y} q_t(\y)\sum_{\x\neq\y} f(\x,\y,t)\\
    &= \sum_{\y}\sum_{\x_0} q^\star(\x_0)\,q_{t|0}(\y\mid\x_0)
       \sum_{\x\neq\y} f(\x,\y,t)\\
    &= \sum_{\x_0} q^\star(\x_0)\sum_{\x} q_{t|0}(\x\mid\x_0)
       \sum_{\y\neq\x}\frac{q_{t|0}(\y\mid\x_0)}{q_{t|0}(\x\mid\x_0)}\,
       f(\x,\y,t),
  \end{align*}
  which is the right-hand side.
\end{proof}

\begin{lemma}[Score bounds]\label{lem:score_bounds}
  For every $t>0$, $\x\in[S]^\sD$, $i\in[\sD]$, and
  $c\in[S-1]$,
  \begin{equation}\label{eq:score_bounds}
    \max\{\alpha_t,R^{-1}\}
    \le
    s_t^\star(\x\oplus_i c,\x)
    \le
    \min\{\alpha_t^{-1},R\}.
  \end{equation}
  The bounds involving $\alpha_t$ hold for every data distribution.
  The bounds involving $R$ follow from
  Assumption~\ref{ass:flip}.
\end{lemma}

\begin{proof}
  Let $\y=\x\oplus_i c$, and recall $g_t$ from~\eqref{eq:g_def} and the ratio~\eqref{eq:score_ratio_g}.  Changing one coordinate changes the Hamming
  distance to any $\x_0$ by at most one.  Hence
  \[
    \alpha_t g_t(\x)
    \le
    g_t(\y)
    \le
    \alpha_t^{-1}g_t(\x).
  \]
  For the bound involving $R$, change variables in~\eqref{eq:g_def} to obtain
  \[
    g_t(\y)
    =
    \sum_{\x_0\in[S]^\sD}
    \alpha_t^{\Ham(\x,\x_0)}
    q^\star(\x_0\oplus_i c)
    \le
    Rg_t(\x).
  \]
  Applying the same upper bounds to the reverse pair gives the lower
  bounds.
\end{proof}
\section{Poisson representation}\label{app:poisson}
Let $S\ge2$, and consider the state space $[S]$ with modular addition $\oplus$ as introduced in Section~\ref{sec:reverse}. We index nonzero shifts by $c\in[S]$ with $c\neq S$.

\begin{lemma}[Poisson representation]
\label{lem:poisson_representation}
  Let $\lambda_c:[0,h]\to[0,\infty)$ be integrable for
  $c\in[S-1]$.  Let $(X_r)_{0\le r\le h}$ be the continuous-time
  Markov chain on $\mathbb Z_S$ whose rate matrix at time $r$
  satisfies
  \[
    Q_r(a,a\oplus c)=\lambda_c(r),
    \qquad
    a\in\mathbb Z_S,\quad c\in[S-1],
  \]
  with diagonal entries chosen so that every row sums to zero.
  Conditional on $X_0=a$,
  \[
    X_h
    \overset{\mathrm d}{=}
    a\oplus\sum_{c\in[S-1]}cN_c,
  \]
  where the $N_c$ are independent and
  \(
    N_c
    \sim
    \operatorname{Poisson}\left(
      \int_0^h\lambda_c(r)\,\mathrm dr
    \right).
  \)
\end{lemma}

\begin{proof}
  Write
  \[
    q_{r|0}(b\mid a)
    =
    \Pr(X_r=b\mid X_0=a).
  \]
  For $c\in[S-1]$, let
    $\Pi_c(a,b):=\mathbf 1\{b=a\oplus c\}$.
  Then
  \[
    Q_r
    =
    \sum_{c\in[S-1]}
    \lambda_c(r)(\Pi_c-\Id),
    \qquad
    \frac{\mathrm d}{\mathrm dr}q_{r|0}
    =
    Q_r^\top q_{r|0}.
  \]
  Since the matrices $\Pi_c$ commute, so do $Q_r$ and $Q_u$ for
  every $r,u\in[0,h]$.  Hence 
  \begin{align*}
    e^{\int_0^h Q_r^\top\,\mathrm dr}
    &=
    \prod_{c\in[S-1]}
    \exp\left\{
      \left(\int_0^h\lambda_c(r)\,\mathrm dr\right)
      (\Pi_c^\top-\Id)
    \right\}
    \\ &=
    \prod_{c\in[S-1]}
    \left[
      \sum_{n_c\ge0}
      \Pr(N_c=n_c)
      (\Pi_c^\top)^{n_c}
    \right]
    =
    \bE\left[
      \prod_{c\in[S-1]}
      (\Pi_c^\top)^{N_c}
    \right].
  \end{align*}
  Moreover, $\prod_{c\in[S-1]}
    (\Pi_c^\top)^{n_c}e_a
    =
    e_{\,a\oplus\sum_{c\in[S-1]}cn_c}$.
  
  \noindent Therefore
  \begin{align*}
    q_{h|0}(b\mid a)
    =
    e_b^\top
    e^{\int_0^hQ_r^\top\,\mathrm dr}e_a
    =
    \Pr\left(
      a\oplus\sum_{c\in[S-1]}cN_c=b
    \right),
  \end{align*}
  which proves the claim.
\end{proof}

\section{Proofs for efficient sampling}
\label{app:efficient_sampling}

We first bound the sampling error in terms of the integrated
risk of the estimated score. Combining this bound with
Corollary~\ref{cor:score_est_rate} proves
Theorem~\ref{thm:efficient_sampling}.

\begin{proposition}[Sampling error decomposition]
\label{prop:sampling_error}
  Suppose Assumption~\ref{ass:flip} holds with constant $R\ge1$,
  and let $S\ge2$ and $T>0$.
  Let $\widehat s\in\bTheta_R^{\mathrm{ws}}$ be an estimated
  score in the range-restricted class
  defined in~\eqref{eq:ws-restricted}, with $\sR=R$.

  For a grid $0=t_0<\cdots<t_N=T$, chosen before generating
  the sampling trajectory, set
  \[
    \Delta
    \defn
    \max_{0\le k<N}(t_{k+1}-t_k).
  \]
  Run Algorithm~\ref{alg:tau_leaping} with $\widehat s$
  on this grid, starting from
  $\z_0\sim\operatorname{Unif}([S]^\sD)$.
  Let $p_T$ be the output law in~\eqref{eq:finalaw},
  conditional on the estimated score and the chosen grid.
  Then
  \begin{equation}\label{eq:efficient_sampling}
    \mathrm{KL}(q^\star\,\|\,p_T)
    \le
    \cL(\widehat s)
    +e^{-T}\sD\log S
    +\Delta\left(
      \frac32T\sD^2R+\sD\log R
    \right),
  \end{equation}
  where $\cL$ is defined in~\eqref{eq:master_obj}.
\end{proposition}
The proposition applies, in particular, to the DSE and ISE estimators defined by~\eqref{eq:emp_score_loss}.

Let $(\overleftarrow{\x}_t)_{0\le t\le T}$ be the exact reverse process.  Its distribution at time $t$ is $q_{T-t}$, and its rate for the shift $(i,c)$ is $\frac1S
  s_{T-t}^\star
  \bigl(
    \overleftarrow{\x}_t\oplus_i c,
    \overleftarrow{\x}_t
  \bigr)$.
\subsection{Conditional KL divergence}

\begin{lemma}[Conditional KL derivative]
\label{lem:conditional_kl_derivative}
  Let $q_t$ and $p_t$ be the distributions of two finite-state
  time-dependent Markov chains with rate matrices $Q_t$ and
  $\widetilde Q_t$, evolving according to \eqref{eq:kolmogorov}. 
  % Suppose that their off-diagonal rates are positive on the same transitions.  Then
  Suppose that their off-diagonal rates are positive on
the same transitions. At times when both distributions
are strictly positive and differentiable,
  \[
  \begin{aligned}
    \frac{\mathrm d}{\mathrm dt}
    \mathrm{KL}(q_t\,\|\,p_t)
    \le
    \bE_{X\sim q_t}
    \left[
      % \sum_{Y\neq X}
      \sum_{\substack{Y\neq X\\Q_t(X,Y)>0}}
      Q_t(X,Y)
      K\left(
        \frac{\widetilde Q_t(X,Y)}
             {Q_t(X,Y)}
      \right)
    \right].
  \end{aligned}
  \]
\end{lemma}

\begin{proof}
  Write
  \[
    r_t(x)=\frac{q_t(x)}{p_t(x)}.
  \]
  Differentiating the KL divergence and using the two Kolmogorov
  equations gives
  \[
  \begin{aligned}
    \frac{\mathrm d}{\mathrm dt}
    \mathrm{KL}(q_t\,\|\,p_t)
    =
    \sum_{x\neq y}
    p_t(x)
    \left[
      r_t(x)Q_t(x,y)
      \log\frac{r_t(y)}{r_t(x)}
      -
      \widetilde Q_t(x,y)
      \bigl(
        r_t(y)-r_t(x)
      \bigr)
    \right].
  \end{aligned}
  \]
  Transitions with $Q_t(x,y)=\widetilde Q_t(x,y)=0$ contribute zero and are omitted below. For each $x\neq y$,
  \[
  \begin{aligned}
    r_t(x)Q_t(x,y)
    \log\frac{r_t(y)}{r_t(x)}
    ={}&
    r_t(x)Q_t(x,y)
    \log
    \frac{
      r_t(y)\widetilde Q_t(x,y)
    }{
      r_t(x)Q_t(x,y)
    }
    \\
    &+
    r_t(x)Q_t(x,y)
    \log
    \frac{Q_t(x,y)}{\widetilde Q_t(x,y)}.
  \end{aligned}
  \]
  Applying $\log u\le u-1$ to the first logarithm shows that the
  corresponding summand is at most
  \[
    r_t(x)
    \left[
      \widetilde Q_t(x,y)-Q_t(x,y)
      +
      Q_t(x,y)
      \log\frac{Q_t(x,y)}{\widetilde Q_t(x,y)}
    \right].
  \]
  Multiplying by $p_t(x)$ and summing over $x\neq y$ proves the
  result.
\end{proof}

\subsection{Two identities for the true score}

Let $P_t:=e^{tQ}$ denote the forward transition matrix.  Uniform
noising is invariant under additive shifts:
\begin{equation}\label{eq:shift_invariance}
  P_t(\x\oplus_i c,\y\oplus_i c)
  =
  P_t(\x,\y).
\end{equation}

\begin{lemma}[Conditional mean of the score]
\label{lem:score_martingale_sampling}
  Let $0\le r<t\le T$.  For every $i\in[\sD]$ and
  $c\in[S-1]$,
  \[
  \begin{aligned}
    &\bE\left[
      s_{T-t}^\star
      \bigl(
        \overleftarrow{\x}_t\oplus_i c,
        \overleftarrow{\x}_t
      \bigr)
      \,\middle|\,
      \overleftarrow{\x}_r
    \right]
    =
    s_{T-r}^\star
    \bigl(
      \overleftarrow{\x}_r\oplus_i c,
      \overleftarrow{\x}_r
    \bigr).
  \end{aligned}
  \]
\end{lemma}

\begin{proof}
  Put $u=T-t$, $v=T-r$, and condition on
  $\overleftarrow{\x}_r=\x$.  The reverse transition probabilities
  satisfy
  \[
    \overleftarrow q_{t|r}(\z\mid\x)
    =
    \frac{
      q_u(\z)P_{v-u}(\z,\x)
    }{
      q_v(\x)
    }.
  \]
  Hence
  \begin{align*}
    &\bE\left[
      s_u^\star
      (\overleftarrow{\x}_t\oplus_i c,\overleftarrow{\x}_t)
      \,\middle|\,
      \overleftarrow{\x}_r=\x
    \right]=
    \frac1{q_v(\x)}
    \sum_{\z}
    q_u(\z\oplus_i c)
    P_{v-u}(\z,\x).
  \end{align*}
  Set $\w=\z\oplus_i c$.  By~\eqref{eq:shift_invariance}, we have $P_{v-u}(\w\ominus_i c,\x)
    =
    P_{v-u}(\w,\x\oplus_i c)$.
  Therefore the last display equals
  \[
    \frac1{q_v(\x)}
    \sum_{\w}
    q_u(\w)
    P_{v-u}(\w,\x\oplus_i c)
    =
    \frac{q_v(\x\oplus_i c)}{q_v(\x)},
  \]
  which is the required score.
\end{proof}

Define
\begin{align}\label{eq:hdef}
  \mathsf h(u):=u\log u-u+1
\end{align}
and
\begin{equation}\label{eq:phi_def}
  \varphi(t)
  :=
  \frac1S
  \bE_{\x\sim q_t}
  \left[
    \sum_{i=1}^{\sD}
    \sum_{c\in[S-1]}
    -\log
    s_t^\star(\x\oplus_i c,\x)
  \right].
\end{equation}

\begin{lemma}[True-score term]\label{lem:true_score_term}
  For every $t\ge0$,
  \[
  \begin{aligned}
    \frac1S
    \bE_{\x\sim q_t}
    \left[
      \sum_{i=1}^{\sD}
      \sum_{c\in[S-1]}
      \mathsf h
      \bigl(
        s_t^\star(\x\oplus_i c,\x)
      \bigr)
    \right]
    =
    \varphi(t).
  \end{aligned}
  \]
  Moreover, $\varphi(t)$ is nonnegative and non-increasing, and
  under Assumption~\ref{ass:flip},
  \[
    \varphi(0)
    \le
    \frac{\sD(S-1)}{S}\log R.
  \]
\end{lemma}

\begin{proof}
  Fix $i$ and $c$, and abbreviate
  \[
    b(\x)
    =
    \frac{q_t(\x\oplus_i c)}{q_t(\x)}.
  \]
  Since
  $q_t(\x)b(\x)=q_t(\x\oplus_i c)$,
  \begin{align*}
    \sum_{\x}q_t(\x)\mathsf h(b(\x))
    &=
    \sum_{\x}
    q_t(\x\oplus_i c)
    \log
    \frac{q_t(\x\oplus_i c)}{q_t(\x)}.
  \end{align*}
  Summing over $c$ and replacing each shift by its inverse gives the
  first identity.

  For monotonicity, let
  \[
    \mathsf N_{i,c}(\x)=\x\oplus_i c.
  \]
  The $(i,c)$-term in~\eqref{eq:phi_def} is
  \[
    \mathrm{KL}
    \left(
      q_t
      \,\middle\|\,
      (\mathsf N_{i,-c})_\#q_t
    \right).
  \]
  Uniform noising commutes with every shift.  Thus, for $u\ge0$,
  \[
    q_{t+u}=q_tP_u,
    \qquad
    (\mathsf N_{i,-c})_\#q_{t+u}
    =
    \bigl(
      (\mathsf N_{i,-c})_\#q_t
    \bigr)P_u.
  \]
  The data-processing inequality therefore shows that every term,
  and hence $\varphi$, is non-increasing.  It is nonnegative because
  it is a sum of KL divergences.

  Finally, Assumption~\ref{ass:flip} gives
  \[
    -\log
    s_0^\star(\x\oplus_i c,\x)
    \le
    \log R.
  \]
  Summing over $(i,c)$ proves the last bound.
\end{proof}

\subsection{Variation of the score risk}

For a positive function $r$ on directed neighboring pairs, let
$\cL_t(r)$ denote the population risk~\eqref{eq:master_obj} with
$r$ in place of the candidate score.

\begin{lemma}[Variation in the true process time]
\label{lem:score_risk_time}
  Suppose Assumption~\ref{ass:flip} holds and
  \[
    r(\x\oplus_i c,\x)\in[R^{-1},R].
  \]
  For $0\le u\le v\le T$,
  \begin{equation}\label{eq:score_risk_time}
    \cL_v(r)
    \le
    \cL_u(r)
    +
    \left(
      \frac{\sD(S-1)}{S}
    \right)^2
    \left(
      R-R^{-1}+2\log R
    \right)
    (v-u).
  \end{equation}
\end{lemma}

\begin{proof}
  Define
  \[
    G_r(\x)
    =
    \frac1S
    \sum_{i=1}^{\sD}
    \sum_{c\in[S-1]}
    \left[
      r(\x\oplus_i c,\x)
      -
      \log r(\x,\x\ominus_i c)
    \right].
  \]
  A change of variables in the terms weighted by the true score,
  together with Lemma~\ref{lem:true_score_term}, gives
  \begin{equation}\label{eq:risk_decomposition}
    \cL_t(r)
    =
    \bE_{\x\sim q_t}
    \bigl[
      G_r(\x)
    \bigr]
    +
    \varphi(t)
    -
    \frac{\sD(S-1)}{S}.
  \end{equation}

  For any $\x,\y\in[S]^\sD$,
  \[
    |G_r(\x)-G_r(\y)|
    \le
    \frac{\sD(S-1)}{S}
    \left(
      R-R^{-1}+2\log R
    \right).
  \]
  The total forward jump rate is
  $\sD(S-1)/S$.  Therefore
  \[
    |(QG_r)(\x)|
    \le
    \left(
      \frac{\sD(S-1)}{S}
    \right)^2
    \left(
      R-R^{-1}+2\log R
    \right).
  \]
  It follows that
  \[
  \begin{aligned}
    \left|
      \bE_{q_v}[G_r]
      -
      \bE_{q_u}[G_r]
    \right|
    &=
    \left|
      \int_u^v
      \bE_{q_t}[QG_r]
      \,\mathrm dt
    \right|
    \\
    &\le
    \left(
      \frac{\sD(S-1)}{S}
    \right)^2
    \left(
      R-R^{-1}+2\log R
    \right)
    (v-u).
  \end{aligned}
  \]
  Since $\varphi(v)\le\varphi(u)$, the result follows from
  \eqref{eq:risk_decomposition}.
\end{proof}

\subsection{Initialization}
The following result can be found in earlier work, e.g.~\cite[Theorem B.7]{ren2024discrete} and quantifies the error incurred in stopping the forward chain at time $t$. A short proof is added here for the sake of completeness.

\begin{lemma}[Initialization error]
\label{lem:uniform_initialization}
  Let $\mathsf u$ be the uniform distribution on $[S]^\sD$.  Then
  \[
    \mathrm{KL}(q_t\,\|\,\mathsf u)
    \le
    e^{-t}
    \mathrm{KL}(q^\star\,\|\,\mathsf u)
    \le
    e^{-t}\sD\log S.
  \]
\end{lemma}

\begin{proof}
  Under uniform noising, each coordinate keeps its original value
  with probability $e^{-t}$ and is otherwise replaced by an
  independent uniform value.  Let $A\subseteq[\sD]$ be the random
  set of coordinates that keep their original values.  Conditional
  on $A$, the resulting distribution is
  \[
    q_A^\star\otimes\mathsf u_{A^\com}.
  \]
  By convexity of KL divergence,
  \[
    \mathrm{KL}(q_t\,\|\,\mathsf u)
    \le
    \bE_A
    \left[
      \mathrm{KL}
      \bigl(
        q_A^\star\,\|\,\mathsf u_A
      \bigr)
    \right].
  \]
  
  If $\x_0\sim q^\star$, the chain rule for entropy gives
  \begin{align*}
      \mathrm{KL}(q_A^\star\,|\,\mathsf u_A)=\sum_{i\in A} \bE_{q_A^\star}\mathrm{KL}(q_A^\star(X_i\in \cdot\, | X_j,j<i, j\in A)\|\mathsf{u}_i).
  \end{align*}
  The fact that
  conditioning reduces entropy then gives
  \begin{align*}
    \bE_A \mathrm{KL}(q_A^\star\,|\,\mathsf u_A)
    &\le \sum_{i=1}^D \Pr(i\in A) \bE_{q_A^\star}\mathrm{KL}(q_A^\star(X_i\in \cdot\, | X_j,j<i)\|\mathsf{u}_i) \\ &=e^{-t}\sum_{i=1}^D \bE_{q_A^\star}\mathrm{KL}(q_A^\star(X_i\in \cdot\, | X_j,j<i)\|\mathsf{u}_i)=e^{-t}\mathrm{KL}(q_0\|\mathsf{u}).
  \end{align*}
  As $\mathrm{KL}(q_0\|\mathsf{u})=D\log{S}+\sum q_0\log{q_0}\le D\log{S}$. This completes the proof. 
\end{proof}

% \subsection{Proof of Theorem~\ref{thm:efficient_sampling}}
\subsection{Proof of Proposition~\ref{prop:sampling_error}}
\begin{proof}
  Condition on the estimated score $\widehat s$ and the chosen grid, and write $h_k=t_{k+1}-t_k$.
  Let
  $\overleftarrow q_{t_{k+1}|t_k}(\cdot\mid\x)$ denote the exact
  reverse transition probabilities, and let
  $p_{t_{k+1}|t_k}(\cdot\mid\x)$ denote the transition probabilities
  used by Algorithm~\ref{alg:tau_leaping}.  The latter are the
  endpoint probabilities of the time-dependent rate matrices
  specified in Section~\ref{subsec:tau_leaping}; equivalently, this
  follows from Lemma~\ref{lem:poisson_representation}.

  The data-processing inequality and the chain rule for KL divergence
  give
  \begin{align}
    \mathrm{KL}(q^\star\,\|\,p_T)
    \le{}&
    \mathrm{KL}(q_T\,\|\,\mathsf u)
    \notag\\
    &+
    \sum_{k=0}^{N-1}
    \bE_{\overleftarrow{\x}_{t_k}\sim q_{T-t_k}}
    \left[
      \mathrm{KL}
      \left(
        \overleftarrow q_{t_{k+1}|t_k}
        (\cdot\mid\overleftarrow{\x}_{t_k})
        \,\middle\|\,
        p_{t_{k+1}|t_k}
        (\cdot\mid\overleftarrow{\x}_{t_k})
      \right)
    \right].
    \label{eq:sampling_chain_rule}
  \end{align}

  Fix $k$ and write $\ell=t_k$.  For
  $t\in[\ell,t_{k+1})$, $i\in[\sD]$, and $c\in[S-1]$, set
  \[
    \theta_t
    =
    \widehat s_{T-t}
    \bigl(
      \overleftarrow{\x}_{\ell}\oplus_i c,
      \overleftarrow{\x}_{\ell}
    \bigr),
  \]
  \[
    \omega_\ell
    =
    s_{T-\ell}^\star
    \bigl(
      \overleftarrow{\x}_{\ell}\oplus_i c,
      \overleftarrow{\x}_{\ell}
    \bigr),
    \qquad
    \gamma_t
    =
    s_{T-t}^\star
    \bigl(
      \overleftarrow{\x}_t\oplus_i c,
      \overleftarrow{\x}_t
    \bigr).
  \]
  Applying Lemma~\ref{lem:conditional_kl_derivative} to each
  conditional transition in~\eqref{eq:sampling_chain_rule} gives
  \begin{align}
    \mathrm{KL}(q^\star\,\|\,p_T)
    \le
    \mathrm{KL}(q_T\,\|\,\mathsf u)
    \notag\\
    &+
    \frac1S
    \sum_{k=0}^{N-1}
    \int_{t_k}^{t_{k+1}}
    \bE
    \left[
      \sum_{i=1}^{\sD}
      \sum_{c\in[S-1]}
      \gamma_t
      K\left(
        \frac{\theta_t}{\gamma_t}
      \right)
    \right]
    \,\mathrm dt.
    \label{eq:sampling_kl_integral}
  \end{align}

  The identity
  \begin{equation}\label{eq:three_term_identity}
    \gamma K(\theta/\gamma)
    =
    \omega K(\theta/\omega)
    +
    (\omega-\gamma)\log\theta
    +
    \mathsf h(\gamma)-\mathsf h(\omega)
  \end{equation}
  separates the integrand into three terms, where $h(\cdot)$ is defined as in \eqref{eq:hdef}.

  \noindent For the first term,
  \[
  \begin{aligned}
    \frac1S
    \bE
    \left[
      \sum_{i,c}
      \omega_\ell
      K\left(
        \frac{\theta_t}{\omega_\ell}
      \right)
    \right]
    =
    \cL_{T-\ell}
    \bigl(
      \widehat s_{T-t}
    \bigr).
  \end{aligned}
  \]
  Lemma~\ref{lem:score_risk_time}, applied with
  $u=T-t$ and $v=T-\ell$, gives
  \[
  \begin{aligned}
    \cL_{T-\ell}
    \bigl(
      \widehat s_{T-t}
    \bigr)
    \le{} 
    \cL_{T-t}
    \bigl(
      \widehat s_{T-t}
    \bigr)
    +
    \left(
      \frac{\sD(S-1)}{S}
    \right)^2
    \left(
      R-R^{-1}+2\log R
    \right)
    (t-\ell).
  \end{aligned}
  \]
  Summing over the intervals and integrating gives
  \begin{small}
  \begin{align}
    &\sum_{k=0}^{N-1}
    \int_{t_k}^{t_{k+1}}
    \frac1S
    \bE
    \left[
      \sum_{i,c}
\omega_{t_k}
K\left(
  \frac{\theta_t}{\omega_{t_k}}
\right)
    \right]
    \,\mathrm dt
\le
    \cL(\widehat s)
    +
    \frac12
    \left(
      \frac{\sD(S-1)}{S}
    \right)^2
    \left(
      R-R^{-1}+2\log R
    \right)
    \sum_{k=0}^{N-1}h_k^2.
    \label{eq:sampling_T1}
  \end{align}
\end{small}
  For the second term,
  $\log\theta_t$ is determined by
  $\overleftarrow{\x}_{\ell}$.  Lemma~\ref{lem:score_martingale_sampling}
  therefore gives
  \[
    \bE\left[
      (\omega_\ell-\gamma_t)\log\theta_t
      \,\middle|\,
      \overleftarrow{\x}_{\ell}
    \right]
    =
    0.
  \]
  Hence the second term has zero expectation. For the third term, Lemma~\ref{lem:true_score_term} gives
  \[
  \begin{aligned}
    \frac1S
    \bE
    \left[
      \sum_{i,c}
      \bigl(
        \mathsf h(\gamma_t)-\mathsf h(\omega_\ell)
      \bigr)
    \right]
    =
    \varphi(T-t)-\varphi(T-\ell).
  \end{aligned}
  \]
Writing $H(q)=-\sum_{\x}q(\x)\log q(\x)$, the Kolmogorov
equation gives $\frac{\mathrm d}{\mathrm dt}H(q_t)=\varphi(t)$.
Consequently,
\[
  \sum_{k=0}^{N-1}
  \int_{t_k}^{t_{k+1}}
  \bigl[\varphi(T-t)-\varphi(T-t_k)\bigr]\,\mathrm dt
  =
  H(q_T)-H(q^\star)
  -
  \sum_{k=0}^{N-1}h_k\varphi(T-t_k).
\]
This is the error from using the entropy decrease rate
at the beginning of each step.

Since $\varphi$ is non-increasing, we have
  \begin{align}
    \sum_{k=0}^{N-1}
    \int_{t_k}^{t_{k+1}}
    \bigl[
      \varphi(T-t)-\varphi(T-t_k)
    \bigr]
    \,\mathrm dt
    &\le
    \sum_{k=0}^{N-1}
    h_k
    \bigl[
      \varphi(T-t_{k+1})
      -
      \varphi(T-t_k)
    \bigr]
    \notag\\
    &\le
    \Delta
    \bigl[
      \varphi(0)-\varphi(T)
    \bigr]
    \le
    \Delta
    \frac{\sD(S-1)}{S}
    \log R.
    \label{eq:sampling_T3}
  \end{align}

  Combining~\eqref{eq:sampling_kl_integral},
  \eqref{eq:sampling_T1}, and~\eqref{eq:sampling_T3}, and applying
  Lemma~\ref{lem:uniform_initialization}, gives
  \[
  \begin{aligned}
    \mathrm{KL}(q^\star\,\|\,p_T)
    \le{}&
    \cL(\widehat s)
    +
    e^{-T}\sD\log S
    +
    \frac12
    \left(
      \frac{\sD(S-1)}{S}
    \right)^2
    \left(
      R-R^{-1}+2\log R
    \right)
    \sum_{k=0}^{N-1}h_k^2
    \\
    &+
    \Delta
    \frac{\sD(S-1)}{S}
    \log R.
  \end{aligned}
  \]
Since
\[
  \sum_{k=0}^{N-1}h_k^2\le T\Delta,
  \qquad
  \frac{\sD(S-1)}{S}\le\sD,
  \qquad
  R-R^{-1}+2\log R\le3R,
\]
the preceding bound proves~\eqref{eq:efficient_sampling}.
\end{proof}

\begin{proof}[Proof of Theorem~\ref{thm:efficient_sampling}]
  On the event in Corollary~\ref{cor:score_est_rate},
  substitute its bound on $\cL(\widehat s)$ into
  Proposition~\ref{prop:sampling_error}.
  This proves~\eqref{eq:end_to_end_sampling}.
  The event concerns only the estimated score, and the
  proposition applies to every fixed grid. The bound
  therefore holds simultaneously for every grid chosen
  before generating the sampling trajectory.

  With $T=\log(n\sD\log S)$, the initialization term
  equals $1/n$. If $\Delta\le n^{-1}$, the discretization
  term is at most
  \[
    \frac1n
    \left(\frac32T\sD^2R+\sD\log R\right).
  \]
  For fixed $d$, $\sD$, and $R$, both $T$ and $\ell_n$
  are $O(\log(nS))$. Substitution into
  \eqref{eq:end_to_end_sampling} gives the stated rate.
\end{proof}

\section{Proofs for score estimation}
\label{sec:proof_estimation}

\subsection{Auxiliary inequalities}
% \label{app:score_estimation_aux}

\begin{lemma}[Bounds for $K$]\label{lem:K_bounds}
  For every $r>0$,
  \[
    \frac{(r-1)^2}{2(r\vee1)}
    \le
    K(r)
    \le
    \frac{(r-1)^2}{2(r\wedge1)}.
  \]
\end{lemma}

\begin{proof}
  Since
  \[
    K(r)
    =
    (r-1)-\log r
    =
    (r-1)^2
    \int_0^1
    \frac{u}{1+u(r-1)}
    \,\mathrm du,
  \]
  and
  \[
    r\wedge1
    \le
    1+u(r-1)
    \le
    r\vee1,
    \qquad u\in[0,1],
  \]
  the result follows from
  $\int_0^1u\,\mathrm du=1/2$.
\end{proof}

\begin{lemma}[Bernstein's inequality; see \cite{Boucheron2013}, Theorem 2.10]\label{lem:bernstein}
  Let $Z_1,\ldots,Z_n$ be independent and identically distributed,
  and write
  \[
    Pf=\bE[f(Z_1)],
    \qquad
    P_nf=\frac1n\sum_{k=1}^nf(Z_k).
  \]
  If $|f|\le B$ and
  $\operatorname{Var}(f(Z_1))\le\sigma^2$, then, for every $u>0$,
  \[
    \Pr\left(
      |P_nf-Pf|
      >
      \sqrt{\frac{2\sigma^2u}{n}}
      +
      \frac{2Bu}{3n}
    \right)
    \le
    2e^{-u}.
  \]
\end{lemma}

\begin{proof}
  Apply the one-sided Bernstein inequality to $f-Pf$ and
  $-(f-Pf)$, and take a union bound.
\end{proof}

\subsection{DSE and ISE}
\label{app:dse_ise_regular}

\begin{proof}[Proof of Proposition~\ref{prop:dse_ise_estimation}]
  We first compare the two losses.  Fix
  $\x_0\in[S]^\sD$, $t>0$, $i\in[\sD]$, and $c\in[S-1]$.
  Changing variables from $\x$ to
  $\z=\x\ominus_i c$ gives
  \begin{align*}
    &\bE_{\x\sim q_{t|0}(\cdot\mid\x_0)}
    \left[
      \log s_t(\x,\x\ominus_i c)
    \right]
    \\
    &\qquad=
    \sum_{\z\in[S]^\sD}
    q_{t|0}(\z\oplus_i c\mid\x_0)
    \log s_t(\z\oplus_i c,\z)
    \\
    &\qquad=
    \bE_{\z\sim q_{t|0}(\cdot\mid\x_0)}
    \left[
      s_{t|0}^\star
      (\z\oplus_i c,\z\mid\x_0)
      \log s_t(\z\oplus_i c,\z)
    \right].
  \end{align*}
  The remaining terms in DSE and ISE are identical.  Hence
  \[
    \ell_t^{\mathrm{DSE}}(s;\x_0)
    =
    \ell_t^{\mathrm{ISE}}(s;\x_0).
  \]
  It is therefore enough to verify the bounds for DSE.

  For fixed $t$ and $\x_0$, write
  \[
    d_{\btheta,t}(\x_0)
    =
    \ell_t^{\mathrm{DSE}}(s^{\btheta};\x_0)
    -
    \ell_t^{\mathrm{DSE}}(s^\star;\x_0).
  \]
  For each $(\x,i,c)$, abbreviate
  \[
    a
    =
    s_t^{\btheta}(\x\oplus_i c,\x),
    \qquad
    b
    =
    s_t^\star(\x\oplus_i c,\x),
    \qquad
    r
    =
    s_{t|0}^\star(\x\oplus_i c,\x\mid\x_0).
  \]
  Then
  \begin{equation}\label{eq:dse_loss_difference}
    d_{\btheta,t}(\x_0)
    =
    \bE_{\x\sim q_{t|0}(\cdot\mid\x_0)}
    \left[
      \frac1S
      \sum_{i=1}^{\sD}
      \sum_{c\in[S-1]}
      \left\{
        a-b-r\log\frac{a}{b}
      \right\}
    \right].
  \end{equation}

  \paragraph{Mean.}
  By~\eqref{eq:obj_division}, $\bE_{\x_0\sim q^\star}
    \bigl[
      d_{\btheta,t}(\x_0)
    \bigr]
    =
    \cL_t(s^{\btheta})$. Since $\xi\sim\operatorname{Unif}(0,T]$,
  \[
  \begin{aligned}
    \bE\bigl[
      \Delta\ell_{\btheta}(\x_0,\xi)
    \bigr]
    =
    T
    \bE_{\xi}
    \bE_{\x_0}
    \bigl[
      d_{\btheta,\xi}(\x_0)
    \bigr]
    =
    \int_0^T
    \cL_t(s^{\btheta})
    \,\mathrm dt
    =
    \cL(s^{\btheta}).
  \end{aligned}
  \]

  \paragraph{Conditional-score identity.}
  For every fixed $t,\x_0,i,c$,
  \begin{equation}\label{eq:conditional_score_mass}
    \bE_{\x\sim q_{t|0}(\cdot\mid\x_0)}
    \left[
      s_{t|0}^\star
      (\x\oplus_i c,\x\mid\x_0)
    \right]
    =
    1.
  \end{equation}
  Indeed,
  \[
  \begin{aligned}
    \sum_{\x\in[S]^\sD}
    q_{t|0}(\x\mid\x_0)
    \frac{
      q_{t|0}(\x\oplus_i c\mid\x_0)
    }{
      q_{t|0}(\x\mid\x_0)
    }
    =
    \sum_{\x\in[S]^\sD}
    q_{t|0}(\x\oplus_i c\mid\x_0)
    =
    1.
  \end{aligned}
  \]

  \paragraph{Uniform bound.}
  By Lemma~\ref{lem:score_bounds} and the range restriction,
  $a,b\in[R^{-1},R]$.  Hence
  \[
    |a-b|\le R,
    \qquad
    \left|\log\frac ab\right|
    \le
    2\log R
    \le
    2R.
  \]
  Using~\eqref{eq:conditional_score_mass} in
  \eqref{eq:dse_loss_difference},
  \[
    |d_{\btheta,t}(\x_0)|
    \le
    3\sD R.
  \]
  Therefore
  \[
    \left|
      \Delta\ell_{\btheta}(\x_0,\xi)
    \right|
    \le
    3T\sD R.
  \]

  \paragraph{Lipschitz bound.}
  Let $\btheta,\btheta'\in\bTheta_R^{\mathrm{ws}}$, and write
  $a'=s_t^{\btheta'}(\x\oplus_i c,\x)$.  Since
  $\log$ is $R$-Lipschitz on $[R^{-1},R]$,
  \[
    |\log a-\log a'|
    \le
    R|a-a'|.
  \]
  It follows from~\eqref{eq:conditional_score_mass} that
  \begin{align*}
    &
    \left|
      \Delta\ell_{\btheta}(\x_0,\xi)
      -
      \Delta\ell_{\btheta'}(\x_0,\xi)
    \right|
    \le
    2T\sD R
    \sup_{\substack{t\in(0,T],\;\x\in[S]^\sD\\
                    i\in[\sD],\;c\in[S-1]}}
    \left|
      s_t^{\btheta}(\x\oplus_i c,\x)
      -
      s_t^{\btheta'}(\x\oplus_i c,\x)
    \right|.
  \end{align*}

  \paragraph{Second moment.}
  The unweighted sum in~\eqref{eq:dse_loss_difference} has total $\le D\sD $. By~\eqref{eq:conditional_score_mass}, the same is true after
  weighting by $r$.  Cauchy--Schwarz therefore gives
  \begin{align}
    d_{\btheta,t}(\x_0)^2
    \le{}&
    2\sD\,
    \bE_{\x\sim q_{t|0}(\cdot\mid\x_0)}
    \left[
      \frac1S
      \sum_{i=1}^{\sD}
      \sum_{c\in[S-1]}
      \left\{
        (a-b)^2
        +
        r\log^2\frac ab
      \right\}
    \right].
    \label{eq:dse_second_moment_step}
  \end{align}
  Averaging over $\x_0$ and applying
  Lemma~\ref{lem:cond_score},
  \begin{align*}
    \bE_{\x_0}
    \bigl[
      d_{\btheta,t}(\x_0)^2
    \bigr]
    \le{}&
    2\sD\,
    \bE_{\x\sim q_t}
    \left[
      \frac1S
      \sum_{i=1}^{\sD}
      \sum_{c\in[S-1]}
      \left\{
        (a-b)^2
        +
        b\log^2\frac ab
      \right\}
    \right].
  \end{align*}

  By Lemma~\ref{lem:K_bounds}, we have $(a-b)^2
    \le
    2R\,bK(a/b)$.
  Moreover,
    $\left|\log\frac ab\right|
    \le
    R|a-b|$,
  and hence $b\log^2\frac ab
    \le
    2R^4\,bK(a/b)$. Since $R\ge1$, we further have 
    $\bE_{\x_0}
    \bigl[
      d_{\btheta,t}(\x_0)^2
    \bigr]
    \le
    8\sD R^4\,
    \cL_t(s^{\btheta})$. Combining the above observations, we finally have
  \begin{align*}
    \bE\left[
      \Delta\ell_{\btheta}(\x_0,\xi)^2
    \right]
    &=
    T
    \int_0^T
    \bE_{\x_0}
    \bigl[
      d_{\btheta,t}(\x_0)^2
    \bigr]
    \,\mathrm dt
    \le
    8T\sD R^4
    \cL(s^{\btheta}).
  \end{align*}

  \noindent Thus Assumption~\ref{ass:training_loss} holds with $B_\ell=3T\sD R$, $V_\ell=8T\sD R^4$, and $L_\ell=2T\sD R$.
\end{proof}

\subsection{Proof of the oracle inequality}
\label{app:oracle_inequality}

\begin{proof}[Proof of Theorem~\ref{thm:score_est}]
  Let
  \[
    Z_k=(\x_0^{(k)},\xi_k),
    \qquad
    Pf=\bE[f(Z_1)],
    \qquad
    P_nf=\frac1n\sum_{k=1}^nf(Z_k).
  \]
  Write $A:=B_\ell+V_\ell$ and 
    $\delta_n
    :=
    \frac12
    \wedge
    \frac{A}{nL_\ell}$. Let $n\ge 2$. Without loss of generality, suppose $A\ge L_\ell$. Then
  $\delta_n^{-1}\le n$. Apply Lemma~\ref{lem:covering} with
  $\sA=S\vee T$.  After replacing a center outside
  $\bTheta_R^{\mathrm{ws}}$ by a point of the class in the same
  covering ball, there are
  $\btheta_1,\ldots,\btheta_M\in\bTheta_R^{\mathrm{ws}}$ such that $\log M
    \le
    C\Lambda_n$ for some $C>0$ where $\Lambda_n$ is defined in \eqref{eq:est_entropy}, and for every $\btheta\in\bTheta_R^{\mathrm{ws}}$, there is an
  index $j(\btheta)$ satisfying
  \[
    \sup_{\substack{t\in(0,T],\;\x\in[S]^\sD\\
                    i\in[\sD],\;c\in[S-1]}}
    \left|
      s_t^{\btheta}(\x\oplus_i c,\x)
      -
      s_t^{\btheta_{j(\btheta)}}(\x\oplus_i c,\x)
    \right|
    \le
    2\delta_n.
  \]
  Assumption~\ref{ass:training_loss} then gives
  \begin{equation}\label{eq:loss_cover_error}
    \left|
      \Delta\ell_{\btheta}
      -
      \Delta\ell_{\btheta_{j(\btheta)}}
    \right|
    \le
    e_n,
    \qquad
    e_n
    =
    2L_\ell\delta_n
    \le
    \frac{2A}{n}.
  \end{equation}
For each representative, set
    $p_j
    :=
    P\Delta\ell_{\btheta_j}
    :=
    \cL(s^{\btheta_j})$. By Assumption~\ref{ass:training_loss}, we have 
    $\left|
      \Delta\ell_{\btheta_j}
    \right|
    \le
    B_\ell$ and $P\left[
      \Delta\ell_{\btheta_j}^2
    \right]
    \le
    V_\ell p_j$.
  Bernstein's inequality and a union bound over the $M$
  representatives show that, with probability at least $1-e^{-x}$,
  \begin{equation}\label{eq:finite_cover_deviation}
    \left|
      P_n\Delta\ell_{\btheta_j}
      -
      p_j
    \right|
    \le
    \frac14p_j
    +
    C
    \frac{
      A(\Lambda_n+x)
    }{
      n
    }
  \end{equation}
  simultaneously for all $j\le M$.

  Work on this event.  Let $\widehat\btheta$ be the empirical
  minimizer and fix any
  $\btheta^\circ\in\bTheta_R^{\mathrm{ws}}$.  Write
  \[
    \widehat\jmath=j(\widehat\btheta),
    \qquad
    \jmath^\circ=j(\btheta^\circ).
  \]
  Since subtracting
  $T\ell_{\xi_k}(s^\star;\x_0^{(k)})$ does not change the empirical
  minimizer, we have $P_n\Delta\ell_{\widehat\btheta}
    \le
    P_n\Delta\ell_{\btheta^\circ}$.
  Using~\eqref{eq:loss_cover_error} and
  \eqref{eq:finite_cover_deviation},
  \begin{align*}
    P\Delta\ell_{\widehat\btheta}
    \le
    P\Delta\ell_{\btheta_{\widehat\jmath}}
    +
    e_n
    &\le
    \frac43
    P_n\Delta\ell_{\btheta_{\widehat\jmath}}
    +
    C
    \frac{
      A(\Lambda_n+x)
    }{
      n
    }
    +
    e_n
    \\
    &\le
    \frac43
    P_n\Delta\ell_{\btheta_{\jmath^\circ}}
    +
    C
    \frac{
      A(\Lambda_n+x)
    }{
      n
    }
    +
    3e_n
    \\
    &\le
    \frac53
    P\Delta\ell_{\btheta_{\jmath^\circ}}
    +
    C
    \frac{
      A(\Lambda_n+x)
    }{
      n
    }
    +
    3e_n
    \\
    &\le
    \frac53
    P\Delta\ell_{\btheta^\circ}
    +
    C
    \frac{
      A(\Lambda_n+x)
    }{
      n
    }
    +
    5e_n.
  \end{align*}
  Since
  $e_n\le2A/n$ and $\Lambda_n+x\ge1$, enlarging $C$ gives
  \[
    P\Delta\ell_{\widehat\btheta}
    \le
    2
    P\Delta\ell_{\btheta^\circ}
    +
    C
    \frac{
      A(\Lambda_n+x)
    }{
      n
    }.
  \]
  By Assumption~\ref{ass:training_loss}, $P\Delta\ell_{\btheta}
    =
    \cL(s^{\btheta})$. 
  Taking the infimum over $\btheta^\circ$ proves the result.
\end{proof}

\subsection{Proof of the score-estimation rate}
\label{app:score_estimation_rate}

\begin{proof}[Proof of Corollary~\ref{cor:score_est_rate}]
  Proposition~\ref{prop:dse_ise_estimation} gives $B_\ell+V_\ell
    \le
    12T\sD R^4$ and $B_\ell+V_\ell
    \ge
    L_\ell$.
  Hence Theorem~\ref{thm:score_est} applies to both DSE and ISE. Choose the search class in Theorem~\ref{thm:bias} with
  $\varepsilon=1/n$.  Then
  \[
    \inf_{\btheta\in\bTheta_R^{\mathrm{ws}}}
    \cL(s^{\btheta})
    \le
    \frac{1\wedge T}{n}.
  \]

  \noindent The network bounds in Theorem~\ref{thm:bias}
  imply $\log\!\bigl(
      8n(S\vee T)\sB\Lw\sL\sW
    \bigr)
    \le
    C_d\sD^d\ell_n^2$ and
  \[
    \Sws\sL
    \le
    C_d\sD^{2d+1}\ell_n^2
    \left[
      \sD^dS^d
      +
      \sD^{3d+2}\ell_n
      +
      \ell_n^4
    \right].
  \]
  Consequently, \eqref{eq:est_entropy} gives
  \[
    \Lambda_n
    \le
    C_d\sD^{3d+1}\ell_n^4
    \left[
      \sD^dS^d
      +
      \sD^{3d+2}\ell_n
      +
      \ell_n^4
    \right].
  \]
  Substituting these two bounds into the oracle inequality gives
  \eqref{eq:score_est_rate}. Finally, for any $\delta\in(0,1)$, take
  $x=1\vee\log(1/\delta)$.  When $d$, $\sD$, $R$, and $T$ are
  fixed, $\ell_n$ is logarithmic in $n$ and $S$.  Since $S\ge2$,
  \eqref{eq:score_est_rate} therefore gives
  \[
    \cL(\widehat s)
    =_{\log} O_{\mathbb P}\!\left(
      \frac{S^d}{n}
    \right),
  \]
  which proves the final assertion.
\end{proof}

\section{Auxiliary neural-network results}\label{app:nn_auxiliary}

\subsection{Elementary ReLU constructions}\label{app:nn_elementary}
In this Section, we will show that neural networks can approximate certain elementary mathematical operations such as multiplication, one-hot encoding, exponentiation, etc. 

\begin{lemma}[Approximation of products; Lemma~8 of
\cite{kohlerlanger}]
\label{lem:mult}
  Let $\sA\ge1$, $r\ge2$, and $K$ be a positive integer.
  Set
  \(
    \Delta_K^{(r)}
    :=
    K+\lceil2r\log_2\sA\rceil+5r
  \)
  and
  \begin{equation}\label{eq:wp}
    \wp_{r,\sA}
    :=\,r\,
    \bigl(K+2r\log_2(8\sA)\bigr).
  \end{equation}
  There exists $ \btheta_\times^{(r)}
    \in
    \bTheta_{r,1}(L,W,S,B)$ with 
  \[
    L\lesssim  \lceil\log_2 r\rceil\,\Delta_K^{(r)}, \qquad W\lesssim r, \qquad S \lesssim \wp_{r,\sA}, \qquad B\lesssim 4^{2r}A^{2r} 
  \]
 such that
  \[
    \sup_{\x\in[-\sA,\sA]^r}
    \left|
      \sN_\rho(\x\mid\btheta_\times^{(r)})
      -
      \prod_{i=1}^r x_i
    \right|
    \le
    2^{-K}.
  \]
\end{lemma}

\begin{proof}
We will use the same network construction as in \cite[Lemma 8]{kohlerlanger}, which in turn proceeds by constructing networks to approximate $2$-fold products and stitching them together following standard series/parallel concatenations techniques from \cite[Section B.1.1]{nakada2020adaptive}. However, their result does not track the sparsity and the bound on the largest weight, which we will provide a short proof of. First, with $d=r$, $a=\sA$, and
$R=\Delta_K^{(r)}-1$, \cite[Lemma 8]{kohlerlanger} implies that there exists a ReLU network of depth
$(R+1)\,\lceil\log_2 r\rceil$ and width $18r$ whose approximation error for the $d$-fold product function on
$[-\sA,\sA]^r$ is at most $4^{4r+1}\,\sA^{4r}\,r\cdot 4^{-R}$.
Taking base-$2$ logarithms, this is at most $2^{-K}$ provided
$R\ge (1/2)[K+(8r+2)+4r\log_2\sA+\log_2 r]$. As $K\ge (K+1)/2$ and $5r-(1/2)\ge 4r+1+(1/2)\log_2 r$ for $r\ge 2$, by choice of $R$ we have
$$R=(K-(1/2))+\lceil 2r \log_2 A\rceil + (5r-(1/2))\ge (K/2)+2r \log_2 A + 4r+1+(1/2)\log_2 r.$$
This establishes the required approximation error. 

Let us bound the sparsity of the above network. Write $q:=\left\lceil\log_2 r\right\rceil$. As discussed above, 
the proof of \cite[Lemma 8]{kohlerlanger} constructs the $r$-fold
product through a binary tree of two-input multiplication networks. At the first
multiplication level there are $2^{q-1}$ copies of the binary
multiplier, at the second level there are $2^{q-2}$ copies, and this continues until only one copy remains. Each two-input multiplier has $R$ hidden layers and fixed width $18$,
and hence has sparsity $\lesssim R$. So the first multiplication layer has a total sparsity $\lesssim R 2^q$. Then by concatenating the networks across each multiplication level, we get an overall sparsity of 
\[
    \bigl|\theta_\times^{(r)}\bigr|_0
    \lesssim R
    \sum_{\ell=1}^{q} 2^{q-\ell}
    =
     R(2^q-1) \lesssim r(K+2r \log_2 (8A)).
\]
This proves \eqref{eq:wp}.

Finally we bound the largest coefficient in the network. The explicit two-input multiplier construction of \cite[Lemma 6]{kohlerlanger},
when operated on $[-a',a']^2$, has parameter magnitude bounded by
$16(a')^2$, with largest parameter coming from the output layer. In the proof of \cite[Lemma 8]{kohlerlanger}, all intermediate inputs to
the binary multipliers are contained in $[-4^r A^r,\,4^r A^r]$. Therefore the parameter magnitudes are at most $16(4^r A^r)^2=4^{2r+2}A^{2r}$. As all network concatenation constructions only duplicate existing parameters and change
their signs, they do not increase their largest weight.
\end{proof}

\begin{lemma}[One-hot encoding]\label{lem:onehot}
  Let $S\ge 2$ be an integer and write $[S]= \{1, \ldots, S\}$. Then there exists
  $\btheta_{\mathrm{oh}}\in\bTheta_{1,S}\bigl(1,\;S{+}1,\;5S,\;S\bigr)$
  with $\sN_\rho(a\mid\btheta_{\mathrm{oh}})=\e_a$ for all
  $a\in[S]$.
\end{lemma}

\begin{proof}
We construct a neural network with one hidden layer. It  computes the $S{+}1$ ramp functions 
\[
  r_j(\cdot)\;:=\;\sigma(\cdot-j),\qquad j=0,1,\ldots,S, 
\]
The output layer is formed by the second order difference factor 
\[
  [\sN_\rho(\cdot\mid\btheta_{\mathrm{oh}})]_b
  \;=\;r_{b-1}(\cdot)-2r_b(\cdot)+r_{b+1}(\cdot),
\]
for $b\in [S]$. For $b=S$, the term $r_{S+1}$ is omitted, since $\sigma(x-S-1)=0$ for all $x\le S$.

\noindent Let us verify that the above construction exactly represents one-hot encoding. Fix $a\in[S]$ and $b\in[S]$.  Since $\sigma(a-j)=\max(a-j,0)$,
\[
  (r_{b-1}-2r_b+r_{b+1})(a)
  \;=\;
  \begin{cases}
    (a-b+1)-2(a-b)+(a-b-1)=0, & a\ge b+1,\\[2pt]
    1-0+0=1, & a=b,\\[2pt]
    0-0+0=0, & a\le b-1,
  \end{cases}
\]
so the output is exactly $\e_a$.

\noindent Finally let us bound the parameters of the above network. Clearly it has $1$ hidden layer and a width of $S+1$. The hidden layer has $S{+}1$ weight entries and $S$ bias
entries (the bias of $r_0$ is zero). The output layer has
three entries per row for $b\le S{-}1$ and two for $b=S$,
totaling $3S{-}1$. Therefore the sparsity is given by 
$(S{+}1)+S+(3S{-}1)=5S$. The maximum weight is $S$ which comes from the bias term in $r_S(\cdot)$.
\end{proof}

\begin{lemma}[Approximation of $e^{-x}$; Lemma~F.12 of
\cite{suzuki023diffusion}]
\label{lemma: exp}
  Let $0<\varepsilon<1$.  There exists
  $\btheta_{\exp}\in\bTheta_{1,1}(\sL,\sW,\sS,\sB)$ with
  \[
    \sL\lesssim \log^2(\varepsilon^{-1}),
    \qquad
    \sW\lesssim \log(\varepsilon^{-1}),
    \qquad
    \sS\lesssim \log^2(\varepsilon^{-1}),
    \qquad
    \sB\lesssim
    \exp\left(C\log^2(\varepsilon)\right),
  \]
  such that
  \[
    \left|
      \sN_\rho(x'\mid\btheta_{\exp})-e^{-x}
    \right|
    \le
    \varepsilon+|x'-x|
  \]
  for all $x,x'>0$.
\end{lemma}

\begin{lemma}[Approximation of the reciprocal; Lemma~F.7 of
\cite{suzuki023diffusion}]
\label{lemma: reciprocal}
  Let $0<\varepsilon<1$.  There exists
  $\btheta_{\mathrm{rec}}\in\bTheta_{1,1}(\sL,\sW,\sS,\sB)$ with
  \[
    \sL\lesssim \log^2(\epsilon^{-1}),
    \qquad
    \sW\lesssim \log^3(\varepsilon^{-1}),
    \qquad
    \sS\lesssim \log^4(\varepsilon^{-1}),
    \qquad
    \sB\lesssim \varepsilon^{-2},
  \]
  such that
  \[
    \left|
      \sN_\rho(x'\mid\btheta_{\mathrm{rec}})
      -
      \frac1x
    \right|
    \le
    \varepsilon+\frac{|x'-x|}{\varepsilon^2}
  \]
  for all $x\in[\varepsilon,\varepsilon^{-1}]$ and
  $x'\in\bR$.
\end{lemma}

\subsection{Stability and covering numbers}
\label{app:nn_covering}

In this Section, we derive metric entropy bounds for the class of weight sharing networks as in \cref{def:wsnn}. Similar results exist for fully connected networks (see e.g.,~\cite[Lemma 19]{kohlerlanger}) and sparse networks (see e.g.,~\cite[Lemma 5]{schmidt2020nonparametric}). An entropy bound for a different variant of weight sharing networks can also be found in (see e.g.~\cite[Lemma 27]{kwon2025nonparametric}).

\begin{lemma}[Lipschitz realization]\label{lem:lipschitz}
Fix an architecture with $\sL$ hidden layers and
$\max_{0\le \ell\le L+1} n_\ell\le\sW$, and let $\sA\ge1$. For any two parameter vectors $\btheta,\tilde\btheta$ with $|\btheta|_\infty,|\tilde\btheta|_\infty\le\sB$ and any
$\x\in[-\sA,\sA]^{n_0}$,
\[
  \left\|
    \sN_\rho(\x\mid\btheta)
    -
    \sN_\rho(\x\mid\tilde\btheta)
  \right\|_\infty
  \le
  \|\btheta-\tilde\btheta\|_\infty\,
  \sA\Lambda,
  \quad \mbox{where} \quad 
  \Lambda
  :=
  (\sL+1)(\sB\vee1)^\sL(\sW+1)^{\sL+1}.
\]
\end{lemma}

\begin{proof}
Write
$\sN_\rho(\x\mid\btheta)
=\W_{\sL+1}\,\phi_\sL\circ\cdots\circ\phi_1(\x)+\bb_{\sL+1}$
with $\phi_\ell(\cdot)=\sigma(\W_\ell\,\cdot+\bb_\ell)$ for
$\ell=1,\ldots,\sL$, where $\btheta$ collects all the weight matrices and biases in vectorized form.  Define the forward map associated with the parameter vector $\btheta$ as 
$\mathcal{A}_k:=\phi_k\circ\cdots\circ\phi_1$ with
$\mathcal{A}_0(\x)=\x$. Consider a different parameter vector $\tilde\btheta$ and let the above maps corresponding to $\tilde\btheta$ be defined as $\tilde\phi_\ell$ and 
$\tilde{\mathcal{A}}_k$. Define 
$\epsilon:=\|\btheta-\tilde\btheta\|_\infty$. We will complete the proof in three steps.

\smallskip\noindent
\emph{Step~1 (Forward activation bound).}
We prove by induction that
$\|\mathcal{A}_k(\x)\|_\infty \le \sA\,(\sB\vee 1)^{k}(\sW+1)^{k}$
for all $k=0,\ldots,\sL$ and $\x\in[-\sA,\sA]^{n_0}$.

\smallskip\noindent
\emph{Base case ($k=0$).}
$\|\mathcal{A}_0(\x)\|_\infty=\|\x\|_\infty\le \sA
=\sA\,(\sB\vee 1)^0(\sW+1)^0$.

\smallskip\noindent
\emph{Inductive step.}
Suppose $\|\mathcal{A}_{k-1}(\x)\|_\infty\le
\sA\,(\sB\vee 1)^{k-1}(\sW+1)^{k-1}$.  Since $\sigma$ satisfies
$|\sigma(t)|\le|t|$ for all~$t$,
\begin{align*}
  \|\mathcal{A}_k(\x)\|_\infty
  &= \|\sigma(\W_k\mathcal{A}_{k-1}(\x)+\bb_k)\|_\infty
  \le \|\W_k\mathcal{A}_{k-1}(\x)\|_\infty + \|\bb_k\|_\infty.
\end{align*}
For any matrix $M\in\R^{m\times n}$ and vector $z\in\R^n$,
$\|Mz\|_\infty\le n\|M\|_\infty\|z\|_\infty$ (each row dot
product is a sum of $n$ terms).
Since $\W_k\in\R^{n_k\times n_{k-1}}$ with $n_{k-1}\le\sW$,
\[
  \|\W_k\mathcal{A}_{k-1}(\x)\|_\infty
  \;\le\; n_{k-1}\|\W_k\|_\infty\|\mathcal{A}_{k-1}(\x)\|_\infty
  \;\le\; \sW\sB\cdot\sA\,(\sB\vee 1)^{k-1}(\sW+1)^{k-1}.
\]
Since $\sA\ge1$ and $(\sB\vee 1)^{k-1}(\sW+1)^{k-1}\ge1$, the
bias satisfies
$\|\bb_k\|_\infty\le\sB\le
\sA\,\sB\,(\sB\vee 1)^{k-1}(\sW+1)^{k-1}$.  Adding the two and
using $\sW\sB+\sB=\sB(\sW+1)\le(\sB\vee 1)(\sW+1)$,
\[
  \|\mathcal{A}_k(\x)\|_\infty
  \;\le\; \sA\,(\sB\vee 1)(\sW+1)\cdot(\sB\vee 1)^{k-1}(\sW+1)^{k-1}
  \;=\; \sA\,(\sB\vee 1)^k(\sW+1)^k.
\]

\smallskip\noindent
\emph{Step~2 (Layer Lipschitz constant).}
Since $\sigma$ is $1$-Lipschitz componentwise, each hidden layer
satisfies, for all $z,z'$,
\[
  \|\phi_\ell(z)-\phi_\ell(z')\|_\infty
  \;\le\;\|\W_\ell(z-z')\|_\infty
  \;\le\; n_{\ell-1}\|\W_\ell\|_\infty\|z-z'\|_\infty
  \;\le\;\sW\sB\,\|z-z'\|_\infty,
\]
so any composition of $j$ consecutive hidden layers
$\phi_{\ell+j}\circ\cdots\circ\phi_{\ell+1}$ is Lipschitz with
constant $(\sW\sB)^{j}$.

\smallskip\noindent
\emph{Step~3 (Telescoping).}
Define hybrid maps $g^{(0)}:=\mathcal{A}_{\sL}$ (all hidden
layers use~$\btheta$), and for $k=1,\dots,\sL$,
$g^{(k)}:=\phi_{\sL}\circ\cdots\circ\phi_{k+1}\circ
\tilde\phi_k\circ\cdots\circ\tilde\phi_1$
(layers $1,\dots,k$ use~$\tilde\btheta$ while layers
$k{+}1,\dots,\sL$ use~$\btheta$).
Then $g^{(\sL)}=\tilde{\mathcal{A}}_{\sL}$ and
\[
  \sN_\rho(\x\mid\btheta)-\sN_\rho(\x\mid\tilde\btheta)
  \;=\;
  \W_{\sL+1}\, g^{(0)}(\x)+\bb_{\sL+1}
  -\tilde\W_{\sL+1}\, g^{(\sL)}(\x)-\tilde\bb_{\sL+1}.
\]
Adding and subtracting $\W_{\sL+1}\,g^{(\sL)}(\x)$:
\begin{align}
  &\;\;\;\;\;\|\sN_\rho(\x\mid\btheta)-\sN_\rho(\x\mid\tilde\btheta)\|_{\infty}\notag \\ 
  &\le\;
  \|\W_{\sL+1}\bigl(g^{(0)}(\x)-g^{(\sL)}(\x)\bigr)\|_{\infty}
  \;+\;
  \|(\W_{\sL+1}-\tilde\W_{\sL+1})g^{(\sL)}(\x)\|_{\infty}
  \;+\;
  \|\bb_{\sL+1}-\tilde\bb_{\sL+1}\|_{\infty} \notag\\
  &\le\;
  \sW\sB\,\|g^{(0)}(\x)-g^{(\sL)}(\x)\|_\infty
  \;+\;
  \epsilon\bigl(\sW\|g^{(\sL)}(\x)\|_\infty+1\bigr).
  \label{eq:outer}
\end{align}
For the first term, telescope:
$g^{(0)}(\x)-g^{(\sL)}(\x)
=\sum_{k=0}^{\sL-1}\bigl[g^{(k)}(\x)-g^{(k+1)}(\x)\bigr]$.
The maps $g^{(k)}$ and $g^{(k+1)}$ agree on layers
$1,\dots,k$ (both use~$\tilde\btheta$) and on layers
$k{+}2,\dots,\sL$ (both use~$\btheta$). They differ only at
layer $k{+}1$. Using Step~2 for the
$\sL{-}k{-}1$ common outer layers, we get:
\begin{equation}\label{eq:inner}
  \|g^{(k)}(\x)-g^{(k+1)}(\x)\|_\infty
  \;\le\; (\sW\sB)^{\sL-1-k}\,
    \|\phi_{k+1}\circ \tilde{\mathcal{A}}_k(\x)-\tilde\phi_{k+1}\circ \tilde{\mathcal{A}}_k(\x)\|_\infty.
\end{equation}
Now, for any $\x\in [-A,A]^{n_0}$ and using step 1, we get:
\begin{align*}
  &\;\;\;\;\;\|\phi_{k+1}\circ \tilde{\mathcal{A}}_k(\x)-\tilde\phi_{k+1}\circ \tilde{\mathcal{A}}_k(\x)\|_\infty \\ 
  &\le \|(\W_{k+1}-\tilde\W_{k+1})\tilde{\mathcal{A}}_k(\x)
    +(\bb_{k+1}-\tilde\bb_{k+1})\|_\infty \le\; \epsilon\,(\sW A (B\vee 1)^k (W+1)^k + 1).
\end{align*}

Substituting the above observation into~\eqref{eq:inner}, summing over
$k=0,\dots,\sL-1$, multiplying by~$\sW\sB$
from~\eqref{eq:outer}, and adding the output-layer contribution
(where $\|g^{(\sL)}(\x)\|_\infty
\le\sA\,(\sB\vee 1)^{\sL}(\sW+1)^{\sL}$ by Step~1), we get:
\begin{align*}
  &\;\;\;\;\;\|\sN_\rho(\x\mid\btheta)
    -\sN_\rho(\x\mid\tilde\btheta)\|_{\infty} \\
  &\quad\le\;
  \epsilon
  \left[
    \sum_{k=0}^{\sL-1}
    (\sW\sB)^{\sL-k}
    \bigl(\sW\sA\,(\sB\vee 1)^k(\sW+1)^k+1\bigr)
    \;+\;
    \sW\sA\,(\sB\vee 1)^{\sL}(\sW+1)^{\sL}+1
  \right] \\
  &\quad\le\;
  \epsilon\,\sA
  \sum_{k=0}^{\sL}
  (\sW\sB)^{\sL-k}(\sB\vee 1)^k(\sW+1)^{k+1},
\end{align*}
where the second inequality uses $A(B\vee 1)^k (W+1)^k\ge 1$ for $0\le k\le L$. Since $\sW\sB\le(\sB\vee 1)(\sW+1)$, each summand is bounded by
$(\sB\vee 1)^{\sL}(\sW+1)^{\sL+1}$. Summing $\sL+1$ such terms gives
\[
  \bigl|\sN_\rho(\x\mid\btheta)
    -\sN_\rho(\x\mid\tilde\btheta)\bigr|
  \;\le\;
  \epsilon\cdot\sA\,(\sL{+}1)\,(\sB\vee 1)^{\sL}(\sW+1)^{\sL+1}
  \;=\;
  \epsilon\,\sA\,\Lambda.
  \qedhere
\]
\end{proof}

\begin{lemma}[Covering number of weight-sharing networks] \label{lem:covering}
  Let $\sA, \sB, \sW, \sL\ge1$ and $0<\delta<1$. For
  the class $\bTheta^{\mathrm{ws}} \equiv \bTheta^{\mathrm{ws}}(\sL,\sW,\sB,\Sws,\omega)$ of
  \Cref{def:wsnn} with Lipschitz constant $\Lw\ge 1$ and inputs in $[-\sA,\sA]^{n_0}$, the following entropy bound holds: 
  \begin{equation}\label{eq:covering-ws}
    \log\mathcal{N}\!\bigl(\delta,\;\bTheta^{\mathrm{ws}},\;
      \|\cdot\|_{L^\infty([-\sA,\sA]^{n_0})}\bigr)
    \;\le\;
    \Sws\Bigl[(\sL{+}1)\log\bigl((\sW{+}1)(\sB{\vee}1)\bigr)
    +\log\bigl(4\sA\,\sB\,\Lw\,(\sL{+}1)/\delta\bigr)\Bigr].
  \end{equation}
\end{lemma}

\begin{proof}
By \Cref{def:wsnn}, every function in the class $\bTheta^{\mathrm{ws}}$ is
$\sN_\rho(\cdot\mid\omega(\bv))$ for some
$\bv\in[-\sB,\sB]^{\Sws}$, and every $\omega(\bv)$ satisfies
$\|\omega(\bv)\|_\infty\le\sB$. The proof now proceeds by packing the
$\Sws$-dimensional domain of $\omega(\cdot)$ and propagating that bound to the function class $\bTheta^{\mathrm{ws}}$ using \cref{lem:lipschitz}. We break it down into three steps.

\smallskip\noindent
\emph{Step~1 (Composed Lipschitz bound).}\;
Let $\bv,\tilde\bv\in[-\sB,\sB]^{\Sws}$.  Since $\omega$ is
$\Lw$-Lipschitz,
$\|\omega(\bv)-\omega(\tilde\bv)\|_\infty
\le\Lw\|\bv-\tilde\bv\|_\infty$.  Applying \Cref{lem:lipschitz}
with $\btheta=\omega(\bv)$ and
$\tilde\btheta=\omega(\tilde\bv)$, where
$\Lambda=(\sL{+}1)(\sB\vee1)^{\sL}(\sW{+}1)^{\sL+1}$ is the
constant of that lemma:
\begin{equation}\label{eq:composed-ws}
  \bigl\|\sN_\rho(\cdot\mid\omega(\bv))
    -\sN_\rho(\cdot\mid\omega(\tilde\bv))
    \bigr\|_{L^\infty([-\sA,\sA]^{n_0})}
  \;\le\;
  \sA\,\Lambda\,\Lw\,\|\bv-\tilde\bv\|_\infty.
\end{equation}

\smallskip\noindent
\emph{Step~2 (Grid construction).}\;
Define the grid spacing
\[
  \eta\;:=\;\frac{\delta}{\sA\,\Lambda\,\Lw},
\]
and the uniform grid
$G:=\bigl([-\sB,\sB]\cap\eta\,\mathbb{Z}\bigr)^{\Sws}$.  For any
$\bv\in[-\sB,\sB]^{\Sws}$, by rounding each coordinate to the nearest grid point, we get a $\tilde\bv\in G$ with
$\|\bv-\tilde\bv\|_\infty\le\eta$.  By~\eqref{eq:composed-ws},
\[
  \bigl\|\sN_\rho(\cdot\mid\omega(\bv))
    -\sN_\rho(\cdot\mid\omega(\tilde\bv))\bigr\|_\infty
  \;\le\;
  \sA\,\Lambda\,\Lw\cdot\eta
  \;=\;\delta.
\]
Therefore $\mathcal{C}:=\{\sN_\rho(\cdot\mid\omega(\tilde\bv)):
\tilde\bv\in G\}$ is a $\delta$-covering of
$\bTheta^{\mathrm{ws}}$ in the
$L^\infty([-\sA,\sA]^{n_0})$-norm.

\smallskip\noindent
\emph{Step~3 (Cardinality bound).}\;
In each coordinate, the grid has
$\bigl|[-\sB,\sB]\cap\eta\,\mathbb{Z}\bigr|
=2\lfloor\sB/\eta\rfloor+1\le 2\sB/\eta+1$ points.
Since $\sA\ge1$, $\sB\ge1$, $\Lw\ge1$, $\Lambda\ge1$, and
$\delta<1$,
\[
  \frac{2\sB}{\eta}
  \;=\;\frac{2\sA\,\sB\,\Lw\,\Lambda}{\delta}
  \;\ge\;2\;\ge\;1,
\]
so $2\sB/\eta+1\le2\cdot2\sB/\eta=4\sB/\eta$.  Hence
\[
  \log\mathcal{N}\bigl(\delta,\bTheta^{\mathrm{ws}},
    \|\cdot\|_\infty\bigr)
  \;\le\;\log|G|
  \;\le\;
  \Sws\,\log\!\left(
    \frac{4\sA\,\sB\,\Lw\,\Lambda}{\delta}\right).
\]
Substituting
$\Lambda=(\sL{+}1)(\sB\vee1)^{\sL}(\sW{+}1)^{\sL+1}$, we get:
\begin{align*}
  \log\!\left(\frac{4\sA\sB\Lw\Lambda}{\delta}\right)
  &=\log\!\left(\frac{4\sA\sB\Lw(\sL{+}1)}{\delta}\right)
    +\sL\log(\sB\vee1)+(\sL{+}1)\log(\sW{+}1)\\[2pt]
  &\le(\sL{+}1)\log\bigl((\sW{+}1)(\sB\vee1)\bigr)
    +\log\!\left(\frac{4\sA\sB\Lw(\sL{+}1)}{\delta}\right),
\end{align*}
which gives~\eqref{eq:covering-ws}.
\end{proof}

\section{Proof of Theorem \ref{thm:score} --- Score approximation}
\label{app:score_approximation}

\noindent We begin with some notation. For $0\le a<b$, define 
\begin{align}\label{eq:proj}
  \mathsf P_{[a,b]}(u)
  :=
  a+\rho(u-a)-\rho(u-b),
  \qquad
  \mathsf P_B(u)
  :=
  \mathsf P_{[0,B]}(u)
  =
  \rho(u)-\rho(u-B).
\end{align}
These are the exact ReLU projections onto $[a,b]$ and $[0,B]$.
They use one hidden layer of width two and at most seven nonzero
parameters whereas $\mathsf P_B$ uses at most five. If $v\in[a,b]$, then the following bound holds which will be used in the sequel. 
\begin{align}\label{eq:projerr}
  |\mathsf P_{[a,b]}(u)-v|
  \le
  |u-v|.
\end{align}

\subsection{Exact computation of the marginals}
\label{app:exact_elimination}

For the formal construction, fix
$\sm\subseteq[\sD]$ with $|\sm|=k>d$ and write
\[
  \sm^\com=[\sD]\setminus\sm,
  \qquad
  r=|\sm^\com|=\sD-k.
\]
The decomposition~\eqref{eq:Am_exp} can be written as
\begin{equation}\label{eq:marg_split}
  A_\sm(\x_\sm)
  =
  \left(
    \prod_{\substack{\sm'\in\Fd\\ \sm'\subseteq\sm}}
    \psi_{\sm'}(\x_{\sm'})
  \right)
  M_\sm(\x_\sm),
  \qquad
  M_\sm(\x_\sm)
  :=
  \sum_{\z\in[S]^{\sm^\com}}
  \prod_{\substack{\sm'\in\Fd\\ \sm'\not\subseteq\sm}}
  \psi_{\sm'}
  \bigl(
    \x_{\sm'\cap\sm},
    \z_{\sm'\cap\sm^\com}
  \bigr).
\end{equation}
We now write the sum defining $M_\sm$ as a sequence of
one-coordinate sums.

\paragraph{Grouping tables by their free coordinates.}
For a table $\psi_{\sm'}$ appearing in~\eqref{eq:marg_split}, let
\[
  U=\sm'\cap\sm^\com
\]
be the coordinates of that table that are still being summed.
For every
$U\subseteq\sm^\com$ with $1\le |U|\le d$, define
\begin{equation}\label{eq:aggregate}
  \sH_U(\z_U;\x_\sm)
  :=
  \prod_{\substack{\sm'\in\Fd,\,
                   \sm'\cap\sm^\com=U}}
  \psi_{\sm'}
  \bigl(
    \x_{\sm'\setminus U},
    \z_U
  \bigr).
\end{equation}
This product contains
\begin{align}\label{eq:totgp}
  \nu_{|U|}
  :=
  \binom{k}{d-|U|}
\end{align}
tables.  When $|U|\le d-1$, it depends on the fixed coordinates
$\x_\sm$.  When $|U|=d$, it is simply the table
$\psi_U(\z_U)$ and is independent of $\x_\sm$.

Let
\[
  \Lambda
  =
  \left\{
    U\subseteq\sm^\com:
    1\le |U|\le d
  \right\}.
\]
Then
\begin{equation}\label{eq:M_sumprod}
  M_\sm(\x_\sm)
  =
  \sum_{\z\in[S]^{\sm^\com}}
  \prod_{U\in\Lambda}
  \sH_U(\z_U;\x_\sm).
\end{equation}
Each table appearing in~\eqref{eq:marg_split} occurs in exactly one
of the products $\sH_U$.

\paragraph{Successive one-coordinate sums.}
Enumerate $\sm^\com=\{j_1<\cdots<j_r\}$ and set
$V_p=\{j_1,\dots,j_p\}$, with $V_0=\emptyset$.  Assign each set
$U\in\Lambda$ to the step at which its largest coordinate is summed,
\begin{equation*}
  \Lambda_p=\{U\in\Lambda:\max U=j_p\}
  =\{U\in\Lambda: j_p\in U\subseteq V_p\},
  \qquad p=1,\dots,r.
\end{equation*}
Every set $U$ has a unique maximum, so
$\Lambda=\bigsqcup_{p=1}^{r}\Lambda_p$.  At step $p$, define the
input-independent product
\begin{equation}\label{eq:alpha_p}
  \alpha_p(b,\z_{V_{p-1}})
  :=
  \prod_{\substack{U\in\Lambda_p\\ |U|=d}}
    \psi_U\bigl(b,\z_{U\setminus\{j_p\}}\bigr),
  \qquad b\in[S],
\end{equation}
and the input-dependent product
\begin{equation}\label{eq:beta_p}
  \beta_p(b,\z_{V_{p-1}};\x_\sm)
  :=
  \prod_{\substack{U\in\Lambda_p\\ |U|\le d-1}}
    \sH_U\bigl(b,\z_{U\setminus\{j_p\}};\x_\sm\bigr),
  \qquad b\in[S].
\end{equation}
The set $\Lambda_p$ contains $a_p:=\binom{p-1}{d-1}$ 
sets of size $d$ and $b_p:=\sum_{\ell=1}^{d-1}\binom{p-1}{\ell-1}$ sets of size at most $d-1$.
Thus $\alpha_p$ is a product of $a_p$ entries from the
$\psi_U$-tables, whereas $\beta_p$ is a product of $b_p$ quantities
$\sH_U$.  The identities $\sum_{p=1}^{r}a_p=\binom{r}{d}$ and $\sum_{p=1}^{r}b_p=\sum_{\ell=1}^{d-1}\binom{r}{\ell}$
confirm that every set $U$ enters exactly once.  Moreover,
\begin{equation}\label{eq:ab_cover}
  \alpha_p\bigl(z^{j_p},\z_{V_{p-1}}\bigr)\,
  \beta_p\bigl(z^{j_p},\z_{V_{p-1}};\x_\sm\bigr)
  =\prod_{U\in\Lambda_p}\sH_U(\z_U;\x_\sm).
\end{equation}
Define partial sums $h_p:[S]^{V_p}\to\bR$ by $h_r := 1$ and, for
$p=r,\dots,1$,
\begin{equation}\label{eq:elim_rec}
  h_{p-1}(\z_{V_{p-1}};\x_\sm)
  :=
  \sum_{b\in[S]}
  \underbrace{\alpha_p(b,\z_{V_{p-1}})}_{\x\text{-free}}
  \;\cdot\;
  \underbrace{\beta_p(b,\z_{V_{p-1}};\x_\sm)\,
    h_p(\z_{V_{p-1}},b;\x_\sm)}_{\x\text{-dependent}}.
\end{equation}
The following proposition relates the above definition with $M_m(\cdot)$ in \eqref{eq:marg_split}.
\begin{proposition}\label{prop:elim_correct}
  Consider the sequence $h_p(\cdot)$ defined recursively in \eqref{eq:elim_rec}. Then for all $\sm\subseteq [D]$ and $x_\sm\in [S]^{|\sm|}$, we have $h_0(\,;\x_\sm)=M_\sm(\x_\sm)$.
\end{proposition}
\begin{proof}
  We prove by downward induction that, for
  $p=r,r-1,\ldots,0$,
  \[
  h_p(\z_{V_p};\x_\sm)
  =
  \sum_{z^{j_{p+1}},\ldots,z^{j_r}\in[S]}
  \prod_{q=p+1}^{r}
  \prod_{U\in\Lambda_q}
  \sH_U(\z_U;\x_\sm),
  \]
  where the sum and product are interpreted as $1$ when $p=r$.
  For $p=r$, the right-hand side is the empty product, which equals
  $1=h_r$.  Suppose the identity holds at level $p$. Using \eqref{eq:ab_cover} and \eqref{eq:elim_rec}, we get:
  \begin{align*}
    h_{p-1}(\z_{V_{p-1}};\x_\sm)
    &=
    \sum_{z^{j_p}\in[S]}
    \prod_{U\in\Lambda_p}
    \sH_U(\z_U;\x_\sm)\,
    h_p(\z_{V_{p-1}},z^{j_p};\x_\sm)
    \\
    &=
    \sum_{z^{j_p},\ldots,z^{j_r}\in[S]}
    \prod_{q=p}^{r}
    \prod_{U\in\Lambda_q}
    \sH_U(\z_U;\x_\sm).
  \end{align*}
  This is the required identity at level $p-1$.  Taking $p=0$ and using
  $\Lambda=\bigsqcup_{q=1}^r\Lambda_q$,~\eqref{eq:M_sumprod} gives
  \[
    h_0(\,;\x_\sm)
    =
    \sum_{\z\in[S]^{\sm^\com}}
    \prod_{U\in\Lambda}
    \sH_U(\z_U;\x_\sm)
    =
    M_\sm(\x_\sm)
  \]
\end{proof}

Next define 
\[
  u_p(b,\z_{V_{p-1}};\x_\sm)
  :=
  \beta_p(b,\z_{V_{p-1}};\x_\sm)\,
  h_p(\z_{V_{p-1}},b;\x_\sm).
\]
Then~\eqref{eq:elim_rec} becomes
\begin{equation}\label{eq:matrix_form}
  h_{p-1}
  =
  \bC_pu_p,
\end{equation}
where
$\bC_p\in\bR^{S^{p-1}\times S^p}$ is block diagonal.  For each
$\z_{V_{p-1}}\in[S]^{V_{p-1}}$, its corresponding block is
\[
  \bC_p :=
  \begin{bmatrix}
    \mathbf{a}_{\z_{V_{p-1}}} & & \\[-2pt]
    & \ddots &
  \end{bmatrix},
  \qquad
  \mathbf{a}_{\z_{V_{p-1}}}
  =\bigl[\alpha_p(1,\z_{V_{p-1}}),\,\ldots,\,
         \alpha_p(S,\z_{V_{p-1}})\bigr]
  \in\bR^{1\times S}.
\]
Thus $\bC_p$ performs the sum over $b\in[S]$.  It has $S^p$
nonzero positions, and every nonzero entry is a product of entries
from the tables $\psi_U$ with $|U|=d$.

For fixed $\sm$, the matrices $\bC_p$ do not depend on the input
$\x_\sm$.  Their entries are obtained from the
$\binom{r}{d}S^d\le \sD^dS^d$ stored table entries, while the
quantities inside $u_p$
depend on $\x_\sm$.  The following subsections approximate these two
parts separately.

\subsection{Networks used in one summation step}
\label{app:marginal_factors}

\begin{lemma}[Evaluation of stored tables]\label{lem:slices}
  Let $\psi_{\sm'}$ be as in~\eqref{eq:qMRF}, and let
  $A_\sm:[S]^{|\sm|}\to[0,1]$ be the partial marginals appearing
  in the pinning decomposition~\eqref{eq:pinning}.  Write $N_{D,d}:=\binom{\sD}{d}$.  There exists
  a ReLU network
  $\btheta_{\mathrm{table}}\in
  \bTheta_{\sD,\,W_{\mathrm{out}}}(\sL^\se,\sW^\se,\sS^\se,\sB^\se)$
  with
  \[
    \sL^\se=3,
    \qquad
    \sW^\se\le 2d\,\sD^d\,S^d,
    \qquad
    \sS^\se\le(2^d+3d+4)\,\sD^d\,S^d,
    \qquad
    \sB^\se\le\max(S,\,d),
  \]
  and $W_{\mathrm{out}}=\sum_{j=1}^{d}\binom{\sD}{j}
  +N_{D,d}\sum_{\ell=0}^{d-1}\binom{d}{\ell}S^\ell$, that produces
  two families of exact outputs for every $\x\in[S]^{\sD}$:
  \begin{enumerate}[label=(\roman*),itemsep=3pt]
    \item \emph{Stored marginals}: $A_\sm(\x_\sm)$ for each
      $\sm\subseteq[\sD]$ with $1\le|\sm|\le d$; these are the
      scalars entering the pinning decomposition~\eqref{eq:pinning}.
    \item \emph{Evaluations of the $\psi_{\sm'}$-tables}:
      $\psi_{\sm'}(\x_{\sm'\setminus U},\,x_U)$ for each
      $\sm'\in\Fd$, each $U\subsetneq \sm'$ with
      $|U|=\ell\in\{0,\ldots,d{-}1\}$, and each
      $x_U\in[S]^\ell$.  The values for $\ell=0$ enter the first
      product in~\eqref{eq:marg_split}; the values for $\ell\ge1$
      enter the grouped products $\sH_U$ in~\eqref{eq:aggregate}.
  \end{enumerate}
\end{lemma}

\begin{proof}
By~\eqref{eq:qMRF}, $\psi_{\sm'}\in[0,1]$.  The network has three stages.

\smallskip\noindent
\emph{Stage~1: one-hot encoding (2 layers).}\quad
Apply \Cref{lem:onehot} independently to each coordinate
$x^i\in[S]$, producing the binary vector
\begin{equation}\label{eq:onehot}
  \onehot\in\{0,1\}^{\sD S},
  \qquad o_{(i,a)}:=\mathbf{1}\{a=x^i\},
  \qquad i\in[\sD],\ a\in[S].
\end{equation}
The first layer holds the $\sD$ independent ramp blocks of
\Cref{lem:onehot} (width $S{+}1$ each). The second layer holds
their outputs.  Since these outputs lie in $\{0,1\}\subseteq
\R_{\ge0}$, the ReLU acts as the identity on them, so the
one-hot output layer serves as a hidden layer of the composite
network.  Each block has sparsity $5S$ and maximum absolute weight $\le S$.

\begin{center}
\renewcommand{\arraystretch}{1.2}\small
\begin{tabular}{@{}lr@{}}
  \toprule
  Depth $\sL$ & $2$ \\
  Width $\sW$ & $\sD(S+1)$ \\
  Sparsity $\sS$ & $5\sD S$ \\
  Magnitude $\sB$ & $\le S$ \\
  \bottomrule
\end{tabular}
\end{center}

\smallskip\noindent
\emph{Stage~2: configuration indicators (1 layer).}\quad
For $d\ge 2$, $j\in\{2,\ldots,d\}$, 
$J\in\binom{[\sD]}{j}$ with $J=\{J_1<\cdots<J_j\}$, and each
$\mathbf{a}=(a_1,\ldots ,a_j)\in[S]^j$, we construct neurons that compute
\begin{equation}\label{eq:and}
  h^{(j)}_{J,\mathbf{a}}
  :=\sigma\!\Bigl(\sum_{r=1}^{j}o_{(J_r,a_r)}
    -(j-\tfrac{1}{2})\Bigr)
  \;=\;
  \begin{cases}
    \tfrac{1}{2}, & \mathbf{a}=\x_{J},\\[2pt]
    0, & \text{otherwise.}
  \end{cases}
\end{equation}
The quantity before applying ReLU equals $\tfrac{1}{2}$ when all
$j$ one-hot
inputs from~\eqref{eq:onehot} are~$1$ (every coordinate
matches), and is $\le-\tfrac{1}{2}$ when at least $1$ one-hot input is $0$. Each corresponding row of the weight matrix has exactly $j$ ones in the
columns indexed by $\{(J_r,a_r):r\in[j]\}$ and zeros
elsewhere, while the bias is $-(j-\tfrac{1}{2})$. An additional $\sD S$ neurons pass $\onehot(\cdot,\cdot)$ through unchanged via
identity weights, see e.g.,~\cite[Lemma F.2]{suzuki023diffusion}.

 Counting nonzero entries by the number $j$ of coordinates, there are
$\binom{\sD}{j}S^j$ neurons, each contributing $j$ weight entries
and $1$ bias entry, for $(j{+}1)\binom{\sD}{j}S^j$ nonzeros.
The $\sD S$ neurons copying $\onehot(\cdot,\cdot)$ contribute $\sD S$ weights.

\begin{center}
\renewcommand{\arraystretch}{1.2}\small
\begin{tabular}{@{}lr@{}}
  \toprule
  Depth $\sL$ & $1$ \\
  Width $\sW$ & $\sum_{j=2}^{d}\binom{\sD}{j}S^j+\sD S$ \\
  Sparsity $\sS$ & $\sum_{j=2}^{d}(j{+}1)\binom{\sD}{j}S^j+\sD S$ \\
  Magnitude $\sB$ & $\le \max\{1,d-\tfrac{1}{2}\}$ \\
  \bottomrule
\end{tabular}
\end{center}

For $d=1$, the first sum in each count is empty, and this stage
only copies the $\sD S$ one-hot values, i.e., a network $\Theta_{\sD,\sD S}(1,\sD,\sD,1)$.

\smallskip\noindent
\emph{Stage~3: table evaluation (1 layer,
output).} Note that for a table $f:[S]^j\to\R$ and a $j$-element
subset $J\subseteq[\sD]$, the exact evaluation
\begin{equation}\label{eq:extract}
  f(\x_{J})
  \;=\;
  \sum_{\mathbf{a}\in[S]^j}
    f(\mathbf{a})\;\cdot\;\mathbf{1}\{\mathbf{a}=\x_{J}\}
\end{equation}
is computed by one neuron, where
$\mathbf{1}\{\mathbf{a}=\x_{J}\}$ is $o_{(J_1,a)}$
from~\eqref{eq:onehot} when $j=1$ (equals $1$), and
$2\,h^{(j)}_{J,\mathbf{a}}$ from~\eqref{eq:and} when
$j\ge2$.

\smallskip\noindent
\emph{Output~(i): stored marginals.}\;
For each $\sm\subseteq[\sD]$ with $|\sm|=j\in\{1,\ldots,d\}$,
apply~\eqref{eq:extract} with $f=A_\sm$ and $J=\sm$ to
produce $A_\sm(\x_\sm)$.  Since each $A_\sm(\mathbf{a})\le1$, the
coefficients in the weighted sum are at most~$2$ (arising from the $2\,h^{(j)}_{J,\mathbf{a}}$ term). There are
$\sum_{j=1}^{d}\binom{\sD}{j}$ such scalars, and each reads
$S^j$ indicator values, giving sparsity
$\sum_{j=1}^{d}\binom{\sD}{j}S^j\le d\,\sD^d S^d$. Note that for (i), we need output dimension $\sum_{j=1}^d \binom{D}{j}$.

\smallskip\noindent
\emph{Output~(ii): evaluations of the $\psi_{\sm'}$-tables.}\;
For each $\sm'\subseteq\sD$, $|\sm'|=d$, and each $U\subsetneq \sm'$ with
  $|U|=\ell\in\{0,\ldots,d{-}1\}$, set $J:=\sm'\setminus U$.  For each $x_U\in[S]^\ell$,
  apply~\eqref{eq:extract} with $f(\mathbf{a})=
\psi_{\sm'}(\mathbf{a},x_U)$ and $J=\sm'\setminus U$ to produce
$\psi_{\sm'}(\x_{J},x_U)$.  The coefficients are at most~$2$.
There are $N_{\sD,d} \sum_{\ell=0}^{d-1}\binom{d}{\ell}S^\ell$ such
neurons, and the total sparsity is
$N_{\sD,d}S^d(2^d{-}1)\le 2^d \sD^d S^d$
(each neuron reads $S^{d-\ell}$ indicators,
$\sum_\ell\binom{d}{\ell}S^\ell\cdot S^{d-\ell}
=S^d\sum_\ell\binom{d}{\ell}=S^d(2^d{-}1)$). Note that for (ii), we need output dimension of $\binom{D}{d}\sum_{\ell=0}^{d-1} \binom{d}{\ell}S^{\ell}$.

\smallskip\noindent
The weight matrix of Stage~3 has one row per output neuron.
The coefficients are entries of the $A_\sm$ tables and the
$\psi_{\sm'}$ tables.  The bias is zero throughout.  Since all
coefficients and indicators are non-negative, the output is
non-negative and ReLU at the preceding hidden layers acts as
the identity.

\begin{center}
\renewcommand{\arraystretch}{1.2}\small
\begin{tabular}{@{}lr@{}}
  \toprule
  Depth $\sL$ & $1$ (output layer) \\
  Output~(i) neurons
    & $\sum_{j=1}^{d}\binom{\sD}{j}$;
      \quad sparsity $\le d\,\sD^d S^d$ \\
  Output~(ii) neurons
    & $N_{\sD,d}\sum_{\ell=0}^{d-1}\binom{d}{\ell}S^\ell$;
      \quad sparsity $\le 2^d \sD^d S^d$ \\
  Magnitude $\sB$ & $\le 2$ \\
  \bottomrule
\end{tabular}
\end{center}

\noindent
Let us now combine the networks across the three steps above into one network. In the following, we will track the parameters of the said network.

\smallskip\noindent
\emph{Depth}: Stage~1 contributes two hidden layers, stage~2 contributes one while stage~3 is the output layer. Therefore $\sL^\se=3$.

\smallskip\noindent
\emph{Width}: The three hidden layers have widths
$\sD(S{+}1)$, $\sD S$ (one-hot ramps and indicators), and
$\sum_{j=2}^{d}\binom{\sD}{j}S^j+\sD S$ (configuration
indicators).  If $d=1$, their maximum is at most $2\sD S$.
If $d\ge2$, the last width is the largest, and since each summand
is at most $\sD^d S^d$ and there are at most $d$ terms, it is at most
$d \sD^d S^d$. Thus $\sW^\se\le 2d \sD^d S^d$ for every $d\ge1$.

\smallskip\noindent
\emph{Sparsity}: The sparsity from stage 1 (one-hot encoding) is $5DS$. For $d\ge2$, the $j=d$ configuration-indicator term from stage 2 contributes sparsity of 
$(d{+}1)N_{\sD,d}S^d$. For $j\le d{-}1$, the sparsity contribution is 
$\sum_{j=2}^{d-1}(j{+}1)\binom{\sD}{j}S^j
\le d\,\sD^d S^d\sum_{j=2}^{d-1}S^{j-d}\le d\,\sD^d S^d$. The copying indicator term contributes $DS$ sparsity.  As $2\sD S\le \sD^d S^d$, combining with the
stage~3 sparsity bound, we get a total of  $(2^d{+}d+3d+4)\sD^d S^d$. If $d=1$, stages~1--3
use at most $5\sD S+\sD S+2\sD S=8\sD S$, which is bounded
by $(2^1+3+4)N_1S=9\sD S$. 

\smallskip\noindent
\emph{Magnitude}: We get largest weight of $S$ from the one-hot biases,
$\max\{d{-}(1/2),1\}$ from the configuration-indicator biases, and $2$ from
the table-evaluation coefficients. Therefore the stated bound $\max(S,d)$ bounds the maximum.
\end{proof}

\begin{lemma}[Grouped products]\label{lem:aggregation}
  Let $K\in\mathbb{Z}_+$ and $\psi_{\sm'}$ as in \eqref{eq:qMRF}. Fix $\sm\subseteq[\sD]$ with $|\sm|=k\ge d{+}1$ and write
  $\sm^\com=[\sD]\setminus\sm$, so
  $|\sm^\com|=\sD-k$.  Denote by
  $\nu_\ell=\binom{k}{d-\ell}$ (see \eqref{eq:totgp}) for the number of factors in
  $\sH_U$ of~\eqref{eq:aggregate} when $|U|=\ell$, and write
  $\nu_{\max}=\max_{1\le\ell\le d-1}\nu_\ell$ and
  $n_\sH:=\sum_{\ell=1}^{d-1}\binom{\sD-k}{\ell}S^\ell$. Suppose the partially evaluated table values
  $\psi_{\sm'}(\x_{\sm'\setminus U},\z_U)$ for all
  $\sm'\in\Fd$ with $U=\sm'\cap\sm^\com$,
  $1\le|U|\le d{-}1$, and $\z_U\in[S]^{|U|}$ are provided as
  network inputs.  Then there exists a
  ReLU network producing $\widehat\sH_U(\z_U;\x_\sm)$ for every
  set $U\subseteq\sm^\com$ with $1\le|U|\le d{-}1$ and
  every $\z_U\in[S]^{|U|}$ (hence an output dimension of $n_\sH$).  Its outputs satisfy
  $\widehat\sH_U(\z_U;\x_\sm)\in[0,1]$, and its parameters obey
  \[
    \sL_\sH\lesssim \lceil\log_2\nu_{\max}\rceil\,
      \Delta_K^{(\nu_{\max})}+1,
    \quad
    \sW_\sH\lesssim \,\nu_{\max}\,n_\sH,
    \quad
    \sS_\sH\lesssim n_\sH\bigl(\wp_{\nu_{\max},1}+5\bigr),
    \quad
    \sB_\sH\lesssim 4^{2\nu_{\max}},
  \]
where $\wp_{\nu_{\max},1}$ is defined in \eqref{eq:wp} and
  \begin{equation}\label{eq:agg-err}
    \sup_{\substack{U\subseteq\sm^\com,\;1\le|U|\le d-1\\
      \z_U\in[S]^{|U|},\;\x_\sm\in[S]^{\sD}}}
      \bigl|\sH_U(\z_U;\x_\sm)-\widehat\sH_U(\z_U;\x_\sm)\bigr|
    \;\le\;2^{-K}.
  \end{equation}
Here $\Delta_K^{(r)}=K+5r$ comes from \Cref{lem:mult} and
  $\wp_{\nu_{\max},1}$ from~\eqref{eq:wp}, both at $\sA=1$.
\end{lemma}

\begin{proof}
If $d=1$ or $k=\sD$, the collection in \eqref{eq:agg-err} and so the assertion
is immediate.  Hence assume $d\ge2$ and $k\le\sD-1$.  Each grouped product
$\sH_U(\z_U;\x_\sm)$ is a $\nu_{|U|}$-fold product of partially
evaluated table values in $[0,1]$ by~\eqref{eq:aggregate}.  We compute
all $n_\sH$ entries in parallel.

\smallskip\noindent
\emph{Construction.}\;
Fix $\ell\in\{1,\ldots,d{-}1\}$, a set
$U\subseteq\sm^\com$ with $|U|=\ell$, and a point
$\z_U\in[S]^\ell$.  By~\eqref{eq:aggregate}, the entry to be
approximated is
\[
  \sH_U(\z_U;\x_\sm)
  \;=\!\!\prod_{\substack{\sm'\in\Fd\\\sm'\cap\sm^\com=U}}\!\!
    \psi_{\sm'}(\x_{\sm'\setminus U},\z_U)
  \;=\;\prod_{i=1}^{\nu_\ell}z_i,
\]
where $z_1,\ldots,z_{\nu_\ell}\in[0,1]$ are the $\nu_\ell$
partially evaluated table values at the given $(U,\z_U,\x)$.

\smallskip\noindent
Apply \Cref{lem:mult} with $r=\nu_\ell$ inputs in
$[0,1]$. This yields a single
product network with some parameters $\btheta_\times^{(\nu_\ell)}$, which produces an
approximation $\widetilde\sH_U(\z_U;\x_\sm)$ satisfying
\begin{equation}\label{eq:agg-pos-err}
  \bigl|\sH_U(\z_U;\x_\sm)-\widetilde\sH_U(\z_U;\x_\sm)\bigr|
  \;\le\;2^{-K}.
\end{equation}
One copy of this product network is placed at each of the
$\binom{\sD-k}{\ell}S^\ell$ positions (ranging over all sets $U$ of free coordinates
with $|U|=\ell$ and all $\z_U\in[S]^\ell$).

\smallskip\noindent
Each preliminary output is followed by the exact projection (see \eqref{eq:proj})
\[
  \widehat\sH_U=\mathsf P_1(\widetilde\sH_U).
\]
Thus $\widehat\sH_U\in[0,1]$.  Since $\sH_U\in[0,1]$, the
projection does not increase the error in~\eqref{eq:agg-pos-err}.

\smallskip\noindent
\emph{Depth.}\;
At $\sA=1$, the depth constant of \Cref{lem:mult} is
$\Delta_K^{(n)}=K+\lceil2n\log_2 1\rceil+5n=K+5n$, so the
product network with $\nu_\ell$ inputs has depth
$\lceil\log_2\nu_\ell\rceil\,\Delta_K^{(\nu_\ell)}$.
Different values of $\ell$ yield different numbers of inputs, and
the depth is increasing in that number.  All are padded to the common
depth
\[
  \sL_\sH\;\lesssim\;\lceil\log_2\nu_{\max}\rceil\,
    \Delta_K^{(\nu_{\max})}+1
\]
so that every grouped product is available at the same network
layer.  The added hidden layer accounts for the projection
$\mathsf P_1$ (see \eqref{eq:proj}).

\smallskip\noindent
\emph{Width.}\;
At the widest internal layer, each position contributes at
most $\lesssim \nu_\ell$ neurons. Bounding $\nu_\ell\le\nu_{\max}$
and summing over all $n_\sH$ positions, we get $
  \sW_\sH\;\lesssim \,\nu_{\max}\,n_\sH$.

\smallskip\noindent
\emph{Sparsity.}\;
Each copy of the $\nu_\ell$-input product network has sparsity at
most $\wp_{\nu_\ell,1}\le\wp_{\nu_{\max},1}$ by~\eqref{eq:wp}
at $\sA=1$. The projection uses five additional nonzero parameters at each
position.  Summing over all $n_\sH$ positions, we get $\sS_\sH\;\lesssim \;n_\sH \bigl(\wp_{\nu_{\max},1}+5\bigr)$.

\smallskip\noindent
\emph{Magnitude.}\;
By \Cref{lem:mult} at $\sA=1$, each product network satisfies
$|\btheta_\times^{(\nu_\ell)}|_\infty
\lesssim 4^{2\nu_\ell}\le 4^{2\nu_{\max}}$, giving
$\sB_\sH\lesssim 4^{2\nu_{\max}}$.

\smallskip\noindent
\emph{Error.}\;
At each position $(U,\z_U)$, the per-entry
bound~\eqref{eq:agg-pos-err} remains $2^{-K}$ after projection,
uniformly over the number of inputs.  Taking the supremum over all
$U\subseteq\sm^\com$ with
$1\le|U|\le d{-}1$, all $\z_U\in[S]^{|U|}$, and all
$\x\in[S]^{\sD}$ gives~\eqref{eq:agg-err}.
\end{proof}

\begin{lemma}[Products with at most $d$ free coordinates]
\label{lem:betaB}
    Let $K\in\mathbb{Z}_+$.  Fix $\sm\subseteq[\sD]$ with
    $|\sm|=k$ and $1\le\sD-k\le d$.  If $d=1$, then
    $\beta_p\equiv1$ and set $\widehat\beta_p\equiv1$ exactly. Suppose $d\ge2$, and recall  $b_p=\sum_{\ell=1}^{d-1} \binom{p-1}{\ell-1}$, with
    $b_{\max}=b_{\sD-k}\le 2^{d-1}$. Suppose the approximations $\widehat\sH_U(\z_U;\x_\sm)$ for
    every $U\subseteq\sm^\com$ with
    $1\le|U|\le d{-}1$ are provided as network inputs and satisfy
    \[
    |\sH_U(\z_U;\x_\sm)-\widehat\sH_U(\z_U;\x_\sm)|\le\varepsilon_\sH
    \qquad\text{and}\qquad
    \widehat\sH_U(\z_U;\x_\sm)\in[0,1]
    \]
    for all $U$, $\z_U$, $\x$, where $\sH_U$ is as in
    \eqref{eq:aggregate}.  Then there exists a ReLU network
    producing $\widehat\beta_p(b,\z_{V_{p-1}};\x_\sm)$ for every
    $p=1,\ldots,\sD-k$ and every $(b,\z_{V_{p-1}})\in[S]^p$, with
    \[
    \sL_{\beta B}=\lceil\log_2 b_{\max}\rceil\,
      \Delta_K^{(b_{\max})}+1,
    \,
    \sW_{\beta B}\le 18\,b_{\max}\cdot 2S^d,
    \,
    \sS_{\beta B}\le\bigl(\wp_{b_{\max},1}+5\bigr)\cdot 2S^d,
    \,
    \sB_{\beta B}\le 4^{2b_{\max}+2},
    \]
    every output belongs to $[0,1]$, and
    \begin{equation}\label{eq:betaB-err}
    \sup_{\substack{1\le p\le \sD-k,\\(b,\z_{V_{p-1}})\in[S]^p,\, 
      \x\in[S]^{\sD}}}
    |\beta_p(b,\z_{V_{p-1}};\x_\sm)
      -\widehat\beta_p(b,\z_{V_{p-1}};\x_\sm)|
    \;\le\;
    b_{\max}\,\varepsilon_\sH
    \;+\;2^{-K}.
    \end{equation}    
    Here $\Delta_K^{(r)}=K+5r$ comes from \Cref{lem:mult} and
    $\wp_{b_{\max},1}$ from~\eqref{eq:wp}, both at $\sA=1$. Here we use the convention
    $\wp_{1,1}=1$.
\end{lemma}

\begin{proof}
If $d=1$, the product in~\eqref{eq:beta_p} is empty and equals
one, so the stated exact construction applies.  Hence assume $d\ge2$.
All $\sD-k$ steps run in parallel.  The total number of output
positions is $\sum_{p=1}^{\sD-k}S^p\le 2S^{\sD-k}\le 2S^d$ since
$\sD-k\le d$.

\smallskip\noindent
\emph{Construction.}\;
Fix a step $p\in\{1,\ldots,\sD-k\}$ and a position
$(b,\z_{V_{p-1}})\in[S]^p$.  By definition,
\[
  \beta_p(b,\z_{V_{p-1}};\x_\sm)
  \;=\!\!\prod_{\substack{U\in\Lambda_p\\|U|\le d-1}}\!\!
    \sH_U(b,\z_{U\setminus\{j_p\}};\x_\sm)
  \;=\;\prod_{i=1}^{b_p}w_i,
\]
a $b_p$-fold product, where $w_1,\ldots,w_{b_p}$ are the exact
grouped-product values at this position.  The network replaces each
$w_i$ with the corresponding approximate value
$\hat w_i=\widehat\sH_{U_i}(b,\z_{U_i\setminus\{j_p\}};\x_\sm)$,
satisfying $w_i,\hat w_i\in[0,1]$ and
$|w_i-\hat w_i|\le\varepsilon_\sH$.

\smallskip\noindent
If $b_p=1$, the product is a single factor: the network copies the value exactly,
$\widehat\beta_p=\hat w_1$ (one identity weight per
position).  If $b_p\ge2$, apply \Cref{lem:mult} with $r=b_p$
inputs in $[-1,1]$. This yields one product network
$\btheta_\times^{(b_p)}$ which takes the $b_p$ approximate
values $\{\hat{w}_i\}_{i=1}^{b_p}$ as direct inputs and produces $\widetilde\beta_p$
satisfying
\begin{equation}\label{eq:prod-err}
  \bigl|\textstyle\prod_{i=1}^{b_p}\hat w_i
    -\widetilde\beta_p\bigr|
  \;\le\;2^{-K}.
\end{equation}
One copy of this product network is placed at each of the $S^p$
positions.
Set $\widehat\beta_p=\mathsf P_1(\widetilde\beta_p)$ when
$b_p\ge2$.  Then $\widehat\beta_p\in[0,1]$, and the projection
does not increase the error because the exact product of the
$\hat w_i$ belongs to $[0,1]$.

\smallskip\noindent
\emph{Depth.} The depth constant of \Cref{lem:mult} is
$\Delta_K^{(r)}=K+5r$, so the
product network with $b_p\ge2$ inputs has depth
$\lesssim \lceil\log_2 b_p\rceil\,\Delta_K^{(b_p)}$. All of them are
padded to the common depth:
\[
  \sL_{\beta B}\lesssim \lceil\log_2 b_{\max}\rceil\,
    \Delta_K^{(b_{\max})}+1.
\]
The added hidden layer accounts for the projection $\mathsf P_1$ (see \eqref{eq:proj}). The padding also covers the
steps with $b_p=1$.

\smallskip\noindent
\emph{Width.}\;
At the widest internal layer, each position contributes at
most $\lesssim b_p\le b_{\max}$ neurons. Summing over all
$\sum_{p=1}^{\sD-k} S^p\le 2S^d$ positions:
\[
  \sW_{\beta B}\;\lesssim \; b_{\max}S^d.
\]

\smallskip\noindent
\emph{Sparsity.}\;
Each copy of the $b_p$-input product network has sparsity at most
$\lesssim \wp_{b_p,1}\le\wp_{b_{\max},1}$ by~\eqref{eq:wp} at $\sA=1$. The
$\wp_{1,1}=1$ case covers the identity weight when $b_p=1$.
The projection contributes at most five more nonzero parameters per position. Summing over all positions, we get: 
\[
  \sS_{\beta B}\;\lesssim 
  \bigl(\wp_{b_{\max},1}+5\bigr) S^d.
\]

\smallskip\noindent
\emph{Magnitude.}\;
By \Cref{lem:mult} at $\sA=1$, each product network satisfies
$|\btheta_\times^{(b_p)}|_\infty
\lesssim 4^{2b_p}\le 4^{2b_{\max}}$ while the weights for the projection are bounded by~$1$.  Hence
$\sB_{\beta B}\lesssim 4^{2b_{\max}}$.

\smallskip\noindent
\emph{Error.}\;
Fix a step $p$ and a position.  Since $\beta_p\in[0,1]$,
replacing $\widetilde\beta_p$ by its projection $\hat{\beta}_p$ onto $[0,1]$
cannot increase the distance to $\beta_p$.
The triangle inequality therefore gives
\[
  |\beta_p-\widehat\beta_p|
  \;\le\;
  \underbrace{\bigl|\textstyle\prod_{i=1}^{b_p}w_i
    -\prod_{i=1}^{b_p}\hat w_i\bigr|}_{\text{(I): input error}}
  \;+\;
  \underbrace{\bigl|\textstyle\prod_{i=1}^{b_p}\hat w_i
    -\widetilde\beta_p\bigr|}_{\text{(II): product error}}.
\]
Term~(II) is at most $2^{-K}$ by~\eqref{eq:prod-err} when
$b_p\ge2$, and equals $0$ when $b_p=1$.  For term~(I), we use the telescoping representation 
\[
  \prod_{i=1}^{b_p}w_i-\prod_{i=1}^{b_p}\hat w_i
  \;=\;\sum_{i=1}^{b_p}
    \Bigl(\prod_{j<i}\hat w_j\Bigr)(w_i-\hat w_i)
    \Bigl(\prod_{j>i}w_j\Bigr).
\]
Since $w_j,\hat w_j\in[0,1]$, each term is at most
$\varepsilon_\sH$.  Thus
\[
  \text{(I)}
  \;\le\;b_p\varepsilon_\sH
  \;\le\;b_{\max}\varepsilon_\sH.
\]
Combining (I) and (II) and taking the
supremum over all $p$,
$(b,\z_{V_{p-1}})$, and $\x$ gives~\eqref{eq:betaB-err}.
\end{proof}

\begin{lemma}[Products with more than $d$ free coordinates]
\label{lem:betaC}
Let $K\in\mathbb{Z}_+$.  Fix $\sm\subseteq[\sD]$ with
  $|\sm|=k$ and $\sD-k\ge d{+}1$.
  If $d=1$, then $\beta_p\equiv1$ and set
  $\widehat\beta_p\equiv1$ exactly, with no additional network.
  Suppose $d\ge2$.  Denote $b_p$ as in~\eqref{eq:beta_p}, with
  $b_{\max}=b_{\sD-k}=\sum_{j=0}^{d-2}\binom{\sD-k-1}{j}$. Suppose the approximations $\widehat\sH_U(\z_U;\x_\sm)$
  for every set $U\subseteq\sm^\com$ with $1\le|U|\le d-1$ are provided as network inputs, satisfying
  \[
    |\sH_U(\z_U;\x_\sm)-\widehat\sH_U(\z_U;\x_\sm)|
    \le\varepsilon_\sH
    \qquad\text{and}\qquad
    \widehat\sH_U(\z_U;\x_\sm)\in[0,1]
  \]
  for all $U$, $\z_U$, $\x$, where $\sH_U$ is as in
  \eqref{eq:aggregate}.  Then there exists a weight-sharing ReLU
  network as in \Cref{def:wsnn}, with sharing map
  $\omega_{\beta C}$, producing
  $\widehat\beta_p(b,\z_{V_{p-1}};\x_\sm)$ for every
  $p=1,\ldots,\sD-k$ and every $(b,\z_{V_{p-1}})\in[S]^p$, with
  \[
    \sL_{\beta C} \lesssim \lceil\log_2 b_{\max}\rceil\,
      \Delta_K^{(b_{\max})}+1,
    \;
    \sW_{\beta C}\lesssim \,b_{\max}\cdot S^{\sD-k},
    \;
    {\Sws}_{\beta C}\lesssim (\sD-k)\bigl(\wp_{b_{\max},1}+5\bigr),
    \;
    \sB_{\beta C}\lesssim 4^{2b_{\max}},
  \]
  every output belongs to $[0,1]$,
  the sharing map has Lipschitz constant $\Lw^{\beta C}=1$,
  and
  \begin{equation}\label{eq:betaC-err}
    \sup_{\substack{1\le p\le \sD-k,\\(b,\z_{V_{p-1}})\in[S]^p,\,
      \x\in[S]^{\sD}}}
    |\beta_p(b,\z_{V_{p-1}};\x_\sm)
      -\widehat\beta_p(b,\z_{V_{p-1}};\x_\sm)|
    \;\le\;
    b_{\max}\,\varepsilon_\sH
    \;+\;2^{-K}.
  \end{equation}
Here $\Delta_K^{(r)}=K+5r$ comes from \Cref{lem:mult} and
  $\wp_{b_{\max},1}$ from~\eqref{eq:wp}, both at $\sA=1$. By convention, we again use $\wp_{1,1}=1$.
\end{lemma}

\begin{proof}
If $d=1$, the product in~\eqref{eq:beta_p} is empty and equals
one, so the stated exact construction applies.  Hence assume $d\ge2$.
All $\sD-k$ steps run in parallel.  The total number of
assignments of the free coordinates across all steps is again $\sum_{p=1}^{\sD-k}S^p\le 2S^{\sD-k}$.

\smallskip\noindent
\emph{Construction.}\;
Fix a step $p\in\{1,\ldots,\sD-k\}$ and an assignment
$(b,\z_{V_{p-1}})\in[S]^p$.  By definition,
\[
  \beta_p(b,\z_{V_{p-1}};\x_\sm)
  \;=\!\!\prod_{\substack{U\in\Lambda_p\\|U|\le d-1}}\!\!
    \sH_U(b,\z_{U\setminus\{j_p\}};\x_\sm)
  \;=\;\prod_{i=1}^{b_p}w_i,
\]
a $b_p$-fold product, where $w_1,\ldots,w_{b_p}$ are the exact
grouped-product values at this assignment.  The network replaces
each $w_i$ with the corresponding approximate value
$\hat w_i=\widehat\sH_{U_i}(b,\z_{U_i\setminus\{j_p\}};\x_\sm)$,
satisfying $w_i,\hat w_i\in[0,1]$ and
$|w_i-\hat w_i|\le\varepsilon_\sH$.

\smallskip\noindent
If $b_p=1$, the network copies the single value exactly,
$\widehat\beta_p=\hat w_1$ (one identity weight per assignment).  If
$b_p\ge2$, apply \Cref{lem:mult} with $r=b_p$ inputs in
$[0,1]$. This yields one product network $\btheta_\times^{(b_p)}$ which 
takes the $b_p$ approximate values as direct inputs and
produces $\widetilde\beta_p$ satisfying
\begin{equation}\label{eq:prodC-err}
  \bigl|\textstyle\prod_{i=1}^{b_p}\hat w_i
    -\widetilde\beta_p\bigr|
  \;\le\;2^{-K}.
\end{equation}
One copy of this product network is placed at each of the $S^p$
assignments at step $p$.  Set
$\widehat\beta_p=\mathsf P_1(\widetilde\beta_p)$ when $b_p\ge2$.
Then $\widehat\beta_p\in[0,1]$, and the projection does not increase
the error because $\prod_i\hat w_i\in[0,1]$.

\smallskip\noindent
\emph{Sharing map.}\;
The parameter vector $\bv$ stores, for each step $p$ with
$b_p\ge2$, the internal weights of the product network
$\btheta_\times^{(b_p)}$, including the fixed scaling factors,
breakpoints, and biases that make the ReLU network
approximate a $b_p$-fold product. Note that each of the $S^p$ assignments share the same weights and there are at most $\lesssim\wp_{b_p,1}$
coordinates by~\eqref{eq:wp}. For each such step, the parameter
vector $\bv$ also stores the at most five nonzero parameters of the projection 
$\mathsf P_1$ (see \eqref{eq:proj}). A single additional coordinate
$v_{\mathbf 1}=1$ is used by all steps with $b_p=1$. The sharing map
$\omega_{\beta C}$ replicates the product and projection network coordinates across
all $S^p$ assignments at step $p$, given by 
\begin{equation}\label{eq:omega-betaC}
  [\omega_{\beta C}(\bv)]_{\text{step }p,\;
    \text{assignment }i,\;\text{entry }j}
  :=\;v_{(p,\,j)},
  \qquad i=1,\ldots,S^p,
\end{equation}
and, for steps with $b_p=1$, fills all $S^p$ identity weights
from the single coordinate $v_{\mathbf 1}=1$.  The indices $p$
(for the step) and $j$ (for the entry) determine the
coordinate of $\bv$, while the assignment index $i\in S^p$ does not. The
parameters of $\mathsf P_1$ are replicated in the same way.

\smallskip\noindent
\emph{Depth.}\; The depth constant of \Cref{lem:mult} is
$\Delta_K^{(r)}=K+5r$, so the
product network with $b_p\ge2$ inputs has depth
$\lesssim \lceil\log_2 b_p\rceil\,\Delta_K^{(b_p)}$. All are
padded to the common depth:
\[
  \sL_{\beta C}=\lceil\log_2 b_{\max}\rceil\,
    \Delta_K^{(b_{\max})}+1.
\]
The added hidden layer accounts for the projection $\mathsf P_1$ (see \eqref{eq:proj}). The padding also covers the
steps with $b_p=1$.

\smallskip\noindent
\emph{Width.}\;
At the widest internal layer, each assignment contributes
at most $\lesssim b_p\lesssim b_{\max}$ neurons.  Summing over all
$\sum_{p=1}^{\sD-k} S^p\le 2S^{\sD-k}$ assignments:
\[
  \sW_{\beta C}\;\lesssim \;\,b_{\max}\cdot S^{\sD-k}.
\]

\smallskip\noindent
\emph{Effective sparsity.}\;
At each step $p$ with $b_p\ge2$, all $S^p$ copies read the
same coordinates of $\bv$ via~\eqref{eq:omega-betaC}, at most
$\lesssim \wp_{b_p,1}+5\le\wp_{b_{\max},1}+5$ parameters by~\eqref{eq:wp}. The steps with
$b_p=1$ together read the single coordinate $v_{\mathbf 1}=1$.
Since $b_1=1$, at most $\sD-k{-}1$ steps have $b_p\ge2$, so
\[
  {\Sws}_{\beta C}
  \;\lesssim \;1+\sum_{p\,:\,b_p\ge2}
    \bigl(\wp_{b_p,1}+5\bigr)
  \;\le\;(\sD-k)\bigl(\wp_{b_{\max},1}+5\bigr).
\]

\smallskip\noindent
\emph{Magnitude.}\;
By \Cref{lem:mult} at $\sA=1$, each product network satisfies
$|\btheta_\times^{(b_p)}|_\infty
\lesssim 4^{2b_p}\le 4^{2b_{\max}}$. 
The coordinate $v_{\mathbf 1}$ equals~$1$.  Hence
$\sB_{\beta C}\lesssim 4^{2b_{\max}}$.

\smallskip\noindent
\emph{Lipschitz constant.}\;
The sharing map $\omega_{\beta C}$ reads a single coordinate
of $\bv$ at each weight position ($v_{(p,j)}$
or $v_{\mathbf 1}$) and no products of coordinates are involved.
For any $\bv,\bv'$:
\[
  \bigl|[\omega_{\beta C}(\bv)]_{p,i,j}
    -[\omega_{\beta C}(\bv')]_{p,i,j}\bigr|
  \;=\;|v_{(p,j)}-v'_{(p,j)}|
  \;\le\;\|\bv-\bv'\|_\infty,
\]
so $\Lw^{\beta C}=1$.

\smallskip\noindent
\emph{Error.}\;
Fix a step $p$ and an assignment.  Since $\beta_p\in[0,1]$,
replacing $\widetilde\beta_p$ by its projection onto $[0,1]$
cannot increase the distance to $\beta_p$.
The triangle inequality therefore
gives
\[
  |\beta_p-\widehat\beta_p|
  \;\le\;
  \underbrace{\bigl|\textstyle\prod_{i=1}^{b_p}w_i
    -\prod_{i=1}^{b_p}\hat w_i\bigr|}_{\text{(I): input error}}
  \;+\;
  \underbrace{\bigl|\textstyle\prod_{i=1}^{b_p}\hat w_i
    -\widetilde\beta_p\bigr|}_{\text{(II): product error}}.
\]
Term~(II) is at most $2^{-K}$ by~\eqref{eq:prodC-err} when
$b_p\ge2$, and equals $0$ when $b_p=1$.  For term~(I), we use the telescoping representation 
\begin{align}\label{eq:tscope}
  \prod_{i=1}^{b_p}w_i-\prod_{i=1}^{b_p}\hat w_i
  \;=\;\sum_{i=1}^{b_p}
    \Bigl(\prod_{j<i}\hat w_j\Bigr)(w_i-\hat w_i)
    \Bigl(\prod_{j>i}w_j\Bigr).
\end{align}
Since $w_j,\hat w_j\in[0,1]$, each term is at most
$\varepsilon_\sH$.  Thus
\[
  \text{(I)}
  \;\le\;b_p\varepsilon_\sH
  \;\le\;b_{\max}\varepsilon_\sH.
\]
Combining (I) and (II) and taking the supremum over all $p$,
$(b,\z_{V_{p-1}})$, and $\x$ gives~\eqref{eq:betaC-err}.
\end{proof}

\begin{lemma}[Summation matrices]\label{lem:contract}
  Let $\psi_{\sm'}$ be as in~\eqref{eq:qMRF}, so that
  $\psi_{\sm'}\in[0,1]^{S^d}$.  Fix $\sm\subseteq[\sD]$ with
  $|\sm|=k\ge d{+}1$ and write $\sm^\com=[\sD]\setminus\sm$, so
  $|\sm^\com|=\sD-k$.  
  For each summation step $p=1,\ldots,\sD-k$, let
  $\mathbf{C}_p\in\R^{S^{p-1}\times S^p}$ be the summation
  matrix of~\eqref{eq:matrix_form}, block-diagonal with $S^{p-1}$
  blocks of size $1\times S$.  By~\eqref{eq:alpha_p}, its
  nonzero entries are
  \begin{equation}\label{eq:Cp-entry}
    [\mathbf{C}_p]_{\z_{V_{p-1}},\,(b,\z_{V_{p-1}})}
    \;=\;\prod_{\substack{U\in\Lambda_p\\|U|=d}}
      \psi_U(b,\,\z_{U\setminus\{j_p\}}),
    \qquad b\in[S],\;\z_{V_{p-1}}\in[S]^{V_{p-1}},
  \end{equation}
  a product of $a_p=\binom{p-1}{d-1}$ entries of the
  $\psi_U$-tables (equal to~$1$ when $p<d$, where the product is
  empty), and all other entries zero. Then the following conclusions hold: 
  \begin{enumerate}[label=(\roman*),itemsep=4pt]
    \item $\mathbf{C}_p$ has at most $S^p$ nonzero entries,
      each lying in $[0,1]$.
    \item For $p\ge d$, the number of distinct entries of the
      $\psi_U$-tables entering $\mathbf{C}_p$ is $a_p\cdot S^d$,
      each filling $S^{p-d}$ entry positions.
    \item The sets of distinct entries used at different steps (across $p=1,2,\ldots , \sD-k$) are
      disjoint, and the total across all $\sD-k$ steps is
      \[
        \sum_{p=1}^{\sD-k}a_p\cdot S^d
        \;=\;\binom{\sD-k}{d}\,S^d
        \;\le\;\sD^dS^d.
      \]
      \item Consider a parameter vector $v$ indexed as follows: $v=\{v_{U,a_U}: U\subseteq m^c,\ |U|=d,\ a_U\in[S]^U\}$. Define a parameter sharing map $\omega(\cdot)$ whose output is the collection of
matrices \((\tilde{\mathbf C}_1(v),\ldots,\tilde{\mathbf C}_r(v))\), defined as follows: 
\begin{align}\label{eq:omega-contract}
[\tilde{\mathbf C_p}(v)]_{z,(b,z')}
:=
\mathbf 1\{z=z'\}
\prod_{U\in\Lambda_p,\, |U|=d}
P_1\!\left(v_{U,(b,z_{U\setminus\{j_p\}})}\right),
\end{align}
for \(p=1,\ldots,r\), \(z,z'\in[S]^{V_{p-1}}\), and \(b\in[S]\). Here the empty product is interpreted as one. If $v^\psi_{U,a_U}:=\psi_U(a_U)$, then \(\tilde{\mathbf C}_p(v^\psi)=\mathbf C_p\) for every \(p=1,\ldots,D-k\).
Moreover, $\omega(\cdot)$ has Lipschitz norm bounded by $L_{\omega}:= \sD^d$.
  \end{enumerate}
\end{lemma}

\begin{proof}
(i) Each row $\z_{V_{p-1}}$ has $S$ structurally
nonzero positions (one per $b\in[S]$), and there are
$S^{p-1}$ rows, so $\mathbf{C}_p$ has at most $S^p$ nonzero
entries.  Each entry is a product of values in $[0,1]$ (the
empty product equals~$1$), hence all of the nonzero entries lie in $[0,1]$.

\smallskip\noindent

(ii) Fix a step $p\ge d$ and a set $U\in\Lambda_p$ with $|U|=d$. There are $a_p=\binom{p-1}{d-1}$ sets $U\in\Lambda_p$ with
$|U|=d$, corresponding to $a_p$ distinct tables. 
The entry $\psi_U(b,\z_{U\setminus\{j_p\}})$ depends on
$|\{j_p\}\cup(U\setminus\{j_p\})|=|U|=d$ of the $p$ coordinates
in $(b,\z_{V_{p-1}})$, and is independent of the remaining
$p-d$ coordinates $\z_{V_{p-1}\setminus(U\setminus\{j_p\})}$.
As $(b,\z_{V_{p-1}})$ ranges over $[S]^p$, the argument
$(b,\z_{U\setminus\{j_p\}})$ takes $S^d$ distinct values, each appearing at $S^{p-d}$ entry positions of $\mathbf{C}_p$. 

\smallskip\noindent Each contributes $S^d$
coordinates, so step $p$ draws from $a_p\cdot S^d$ distinct
coordinates of~$\bv$ in total.  For $p<d$, $a_p=0$: the
product in~\eqref{eq:Cp-entry} is empty, every entry
equals~$1$, and no coordinates are read.

\smallskip\noindent 
(iii) A table $\psi_U$ with $|U|=d$ belongs to
$\Lambda_p$ if and only if $\max U=j_p$.  Therefore $\{\Lambda_p\}+{p=1}^{D-k}$ forms a pairwise disjoint collection. Summing over $p$ yields 
\[
  \sum_{p=1}^{\sD-k}a_p\cdot S^d
  \;=\;S^d\sum_{p=1}^{\sD-k}\binom{p-1}{d-1}
  \;=\;S^d\cdot\binom{\sD-k}{d}
  \;\le\;\sD^dS^d,
\]
where the second equality is the hockey-stick identity.

\smallskip\noindent 
(iv) As every entry of each 
\(\psi_U\)-table lies in \([0,1]\), the projection \(P_1\) is inactive at
\(v^\psi\). Hence, for a nonzero entry, namely when \(z=z'\), by definition, we have 
\[
[\tilde{\mathbf C}_p(v^\psi)]_{z,(b,z)}
=
\prod_{U\in\Lambda_p,\; |U|=d}
\psi_U(b,z_{U\setminus\{j_p\}}) = [\mathbf{C}_p]_{z,(b,z)}.
\]
When
\(z\ne z'\), the factor \(\mathbf 1\{z=z'\}\) gives zero, matching the
block-diagonal structure of \(\mathbf C_p\). Therefore
\(\tilde{\mathbf C}_p(v^\psi)=\mathbf C_p\) for every \(p\).

It remains to prove the Lipschitz bound. Fix \(p,z,\), and \(b\). For $p=1,2,\ldots ,D-k$, recall that $a_p=\binom{p-1}{d-1}\le \sD^d$ is the cardinality of the set $\{U\in\Lambda_p:\, |U|=d\}$. By the same telescoping argument as in \eqref{eq:tscope}, we get:
\begin{align*}
    &\;\;\;\;\;\bigg|\prod_{U\in\Lambda_p:\, |U|=d}P_1\!\left(v_{U,(b,z_{U\setminus\{j_p\}})}\right)-\prod_{U\in\Lambda_p:\, |U|=d} P_1\!\left(v'_{U,(b,z_{U\setminus\{j_p\}})}\right)\bigg| \\
    &\le \sum_{U\in\Lambda_p:\, |U|=d} \big|P_1\!\left(v_{U,(b,z_{U\setminus\{j_p\}})}\right)- P_1\!\left(v'_{U,(b,z_{U\setminus\{j_p\}})}\right)\big| \\
    &\le \sD^d \|v-v'\|_{\infty}. 
\end{align*}
The last inequality follows from the fact that the projection map $P_1$ is $1$-Lipschitz.
\end{proof}

\subsection{Approximation of the successive sums}
\label{app:approx_elimination}

\begin{lemma}[Product before each summation]\label{lem:uprod}
Let $K\in\mathbb{Z}_+$.  Fix $\sm\subseteq[\sD]$ with
  $|\sm|=k>d$ and $\sD-k>d$.  For
  each step $p=1,\ldots,\sD-k$, suppose the approximate partial sums
  and the approximations of $\beta_p$ satisfy
  \[
    \widehat\beta_p(b,\z_{V_{p-1}};\x_\sm)\in[0,1],
    \qquad
    \widehat h_p(\z_{V_p};\x_\sm)\in[0,\sA_p],
  \]
  for all $(b,\z_{V_{p-1}})\in[S]^p$ and $\x\in[S]^{\sD}$, where $\sA_p:=S^{\sD-k-p}$. Suppose also that $\hat{h}_{D-k}=1$.

  \smallskip\noindent
  At $p=\sD-k$, set $\widehat u_{\sD-k}=\widehat\beta_{\sD-k}=\hat{\beta}_{D-k}\hat{h}_{D-k}$. For each $p=1,\ldots,\sD-k{-}1$, there exists a
  weight-sharing ReLU network as in \Cref{def:wsnn}, with
  sharing map $\omega_p^{u}$, producing
  $\widehat u_p(b,\z_{V_{p-1}};\x_\sm)\in[0,S^{\sD-k-p}]$
  at every assignment
  $(b,\z_{V_{p-1}})\in[S]^p$ with
  \begin{equation}\label{eq:uprod-err}
    \sup_{\substack{(b,\z_{V_{p-1}})\in[S]^p\\
      \x\in[S]^{\sD}}}
      |\widehat\beta_p\,\widehat h_p
        -\widehat u_p|
    \;\le\;2^{-K}
    \;=:\;\varepsilon_p^{\mathrm{mult}},
  \end{equation}
  and network parameters
  \[
    \sL_{u,p}\lesssim K+\lceil 4\log_2\sA_p\rceil+11,
    \quad
    \sW_{u,p}\lesssim \,S^p,
    \quad
    {\Sws}_{u,p}\lesssim \wp_{2,\sA_p}+5,
    \quad
    \sB_{u,p}\lesssim \,S^{4(\sD-k-p)},
  \]
where $\wp_{2,\sA_p}$ comes from~\eqref{eq:wp} at
  $\sA=\sA_p$.  The sharing map has Lipschitz constant
  $\Lw^{u}=1$.
\end{lemma}

\begin{proof}
At $p=\sD-k$, $h_{\sD-k}\equiv\widehat h_{\sD-k}\equiv1$, so
$\widehat u_{\sD-k}=\widehat\beta_{\sD-k}$ with zero error and no
additional network layers. 

\smallskip\noindent
\emph{Construction.}\;
Fix $p\in\{1,\ldots,\sD-k{-}1\}$ and an assignment
$(b,\z_{V_{p-1}})\in[S]^p$.  The quantity to be
approximated is the binary product
\[
  \widehat\beta_p(b,\z_{V_{p-1}};\x_\sm)\;\cdot\;
  \widehat h_p(\z_{V_{p-1}},b;\x_\sm).
\]
Both inputs lie in $[0,\sA_p]$, where
$\sA_p=S^{\sD-k-p}$.  Apply \Cref{lem:mult} with $n=2$ and
$\sA=\sA_p$. This yields a single product network $\btheta_\times^{(2)}$ with two
inputs, that produces $\widetilde u_p$ with error at most $2^{-K}$.
Set the projection map (see \eqref{eq:proj})
\[
  \widehat u_p=\mathsf P_{\sA_p}(\widetilde u_p).
\]
The exact product belongs to $[0,\sA_p]$, so the projection does not
increase the error.  Thus $\widehat u_p\in[0,\sA_p]$ and
\eqref{eq:uprod-err} holds. One copy of this product network and projection is placed at each of
the $S^p$ assignments at step $p$.

\smallskip\noindent
\emph{Sharing map.}\;
For each step $p<\sD-k$, the parameter vector $\bv$ stores the
internal weights of the product network $\btheta_\times^{(2)}$ at
range $\sA_p$, including the fixed scaling factors, breakpoints, and
biases that make the ReLU network approximate a binary
product at that range, at most $\lesssim \wp_{2,\sA_p}$ coordinates
by~\eqref{eq:wp}.  It also stores the at most five nonzero
parameters of $\mathsf P_{\sA_p}$.  The sharing map
$\omega_p^{u}$ replicates these across all $S^p$ assignments:
\begin{equation}\label{eq:omega-uprod}
  [\omega_p^{u}(\bv)]_{\text{assignment }i,\;
    \text{entry }j}
  :=\;v_{(p,\,j)},
  \qquad i=1,\ldots,S^p.
\end{equation}
The assignment index $i$ does not affect which coordinate
of $\bv$ is read.

\smallskip\noindent
\emph{Depth.}\;
By \Cref{lem:mult} at $n=2$ and range $\sA_p$, the product network
has depth
\[
  \sL_{u,p}
  \;\lesssim \;\lceil\log_2 2\rceil\,
    \bigl(K+\lceil 2\cdot2\log_2\sA_p\rceil+5\cdot2\bigr)
  \;=\;K+\lceil 4\log_2\sA_p\rceil+11,
\]
largest at $p=1$ and smallest at $p=\sD-k{-}1$.  The added hidden
layer accounts for $\mathsf P_{\sA_p}$.

\smallskip\noindent
\emph{Width.}\;
Each assignment contributes $18\cdot2=36$ neurons (the
width of a single binary multiplier from \Cref{lem:mult}).
Step $p$ has $S^p$ assignments, giving
\[
  \sW_{u,p}\lesssim \,S^p.
\]

\smallskip\noindent
\emph{Effective sparsity.}\;
All $S^p$ copies at step $p$ read the same coordinates of
$\bv$ via~\eqref{eq:omega-uprod}, at most
$\wp_{2,\sA_p}+5$ of them by~\eqref{eq:wp} at $\sA=\sA_p$,
including the projection parameters:
\[
  {\Sws}_{u,p}\;\lesssim \;\wp_{2,\sA_p}+5.
\]
This grows only logarithmically in $S$, through
$\log_2(8\sA_p)=3+(\sD-k{-}p)\log_2 S$.

\smallskip\noindent
\emph{Magnitude.}\;
By \Cref{lem:mult} at $n=2$ and range $\sA_p$, the product network
satisfies
\[
  \bigl|\btheta_\times^{(2)}\bigr|_\infty
  \;\lesssim\;\,\sA_p^{2\cdot2}
  \;=\;S^{4(\sD-k-p)},
\]
giving $\sB_{u,p}\le 2^{12}S^{4(\sD-k-p)}$.

\smallskip\noindent
\emph{Lipschitz constant.}\;
The sharing map~\eqref{eq:omega-uprod} reads a single
coordinate $v_{(p,j)}$ at each weight position, with no
products of coordinates.  For any $\bv,\bv'$:
\[
  \bigl|[\omega_p^{u}(\bv)]_{i,j}
    -[\omega_p^{u}(\bv')]_{i,j}\bigr|
  \;=\;|v_{(p,j)}-v'_{(p,j)}|
  \;\le\;\|\bv-\bv'\|_\infty,
\]
so $\Lw^{u}=1$.

\smallskip\noindent
\emph{Error.}\;
At $p=\sD-k$, set $\widehat u_{\sD-k}=\widehat\beta_{\sD-k}$.  Since
  $\widehat h_{\sD-k}\equiv1$, this is exact, with
  $\varepsilon_{\sD-k}^{\mathrm{mult}}=0$ and no additional
  network.  At $p<\sD-k$, \Cref{lem:mult} and the nonexpansiveness
  of the projection guarantee
\[
  |\widehat\beta_p\,\widehat h_p-\widehat u_p|
  \;\le\;2^{-K}
\]
at each assignment.  Taking the supremum over all
$(b,\z_{V_{p-1}})\in[S]^p$ and $\x\in[S]^{\sD}$
gives~\eqref{eq:uprod-err}.
\end{proof}

\begin{lemma}[Error propagation through successive sums]
\label{lem:error}
  Fix $\sm\subseteq[\sD]$ with $|\sm|=k\ge d{+}1$, and set
  $r=\sD-k$.  Let $\psi_{\sm'}$ be as in~\eqref{eq:qMRF}.
  For each step $p=1,\ldots,r$, let $\widehat\beta_p\in[0,1]$
  approximate $\beta_p$ in~\eqref{eq:elim_rec}, and let
  $\widehat u_p$ approximate $\widehat\beta_p\widehat h_p$, with
  \[
    \sup_{\z_{V_p},\,\x}
      |\beta_p-\widehat\beta_p|
    \;\le\;\varepsilon_p^{\beta},
    \qquad
    \sup_{\z_{V_p},\,\x}
      |\widehat\beta_p\,\widehat h_p-\widehat u_p|
    \;\le\;\varepsilon_p^{\mathrm{mult}},
    \qquad p=1,\ldots,r.
  \]
  Define the approximate partial sums by
  \[
    \widehat h_r\equiv1,
    \qquad
    \widetilde h_{p-1}
    =\sum_{b\in[S]}\alpha_p\widehat u_p,
    \qquad
    \widehat h_{p-1}
    =\mathsf P_{S^{r-p+1}}(\widetilde h_{p-1}),
  \]
  for $p=r,\ldots,1$, and set $\widehat M_\sm=\widehat h_0$.
  Since $\widehat h_r\equiv1$, set
  $\widehat u_r=\widehat\beta_r$ and
  $\varepsilon_r^{\mathrm{mult}}=0$.
  If the summation matrix $\mathbf{C}_p$ of~\eqref{eq:matrix_form} is
  applied exactly at every step, then
  \begin{equation}\label{eq:err-closed}
    \sup_{\x}\,|M_\sm(\x)-\widehat M_\sm(\x)|
    \;\le\;
    S^r\sum_{j=1}^{r}\varepsilon_j^{\beta}
    \;+\;\sum_{j=1}^{r-1}S^j\,\varepsilon_j^{\mathrm{mult}}.
  \end{equation}
  In particular, if
  $\varepsilon_j^{\beta}\le\varepsilon^{\beta}$ for all $j$
  and $\varepsilon_j^{\mathrm{mult}}\le
  \varepsilon^{\mathrm{mult}}$ for all $j<r$, then
  \[
    \sup_{\x}\,|M_\sm-\widehat M_\sm|
    \;\le\;
    rS^r\varepsilon^{\beta}
    \;+\;2S^{r-1}\varepsilon^{\mathrm{mult}}.
  \]
\end{lemma}

\begin{proof}
Write
$e_p:=\sup_{\z_{V_p},\x}|h_p-\widehat h_p|$, where $h_p$
is the exact partial sum in~\eqref{eq:elim_rec}, and define
\[
  E_p
  :=
  S^{r-p}\sum_{j=p+1}^{r}\varepsilon_j^{\beta}
  +\sum_{j=p+1}^{r}S^{j-p}\varepsilon_j^{\mathrm{mult}},
  \qquad p=0,1,\ldots,r,
\]
so that $E_r=0$.  Since $h_r\equiv\widehat h_r\equiv1$, we
have $e_r=0$.  Moreover, $h_p$ is a sum of at most $S^{r-p}$
products of numbers in $[0,1]$, and therefore
\[
  0\le h_p\le S^{r-p}.
\]
The projection in the definition of $\widehat h_p$ consequently
does not increase the error and ensures
$0\le\widehat h_p\le S^{r-p}$.

\smallskip\noindent
\emph{One-step bound.}\;
The exact recursion from~\eqref{eq:elim_rec} and the approximate
recursion using the summation matrix $\mathbf{C}_p$
of~\eqref{eq:matrix_form} are
\[
  h_{p-1}
  =\sum_{b\in[S]}\alpha_p\,\beta_p\,h_p,
  \qquad
  \widetilde h_{p-1}
  =\sum_{b\in[S]}\alpha_p\,\widehat u_p.
\]
Subtracting and decomposing $\beta_ph_p-\widehat u_p$ into
three pieces:
\[
  \beta_ph_p-\widehat u_p
  \;=\;
  \underbrace{\beta_p(h_p-\widehat h_p)}_{\text{partial-sum error}}
  \;+\;
  \underbrace{(\beta_p-\widehat\beta_p)\,\widehat h_p}_{
    \beta\text{ error}}
  \;+\;
  \underbrace{\widehat\beta_p\,\widehat h_p-\widehat u_p}_{
    \text{product error}}.
\]
By \Cref{lem:contract}(i), the entries $\alpha_p$ of
$\mathbf{C}_p$ lie in $[0,1]$.  Moreover,
$|\beta_p|\le1$ and $|\widehat h_p|\le S^{r-p}$.  Since
$h_{p-1}\in[0,S^{r-p+1}]$, the final projection cannot increase
$|h_{p-1}-\widetilde h_{p-1}|$.  Taking suprema and summing the
$S$ terms therefore gives
\begin{equation}\label{eq:err-rec}
  e_{p-1}
  \;\le\;
  S\,e_p
  \;+\;S^{r-p+1}\,\varepsilon_p^{\beta}
  \;+\;S\,\varepsilon_p^{\mathrm{mult}}.
\end{equation}

\smallskip\noindent
\emph{Induction.}\;
We prove by downward induction on $p=r,r{-}1,\ldots,0$ that
$e_p\le E_p$.  At $p=r$, both sides are zero.

\smallskip\noindent
\emph{Inductive step.}\;
Assume $e_p\le E_p$.  The one-step bound~\eqref{eq:err-rec}
and a direct calculation give
\[
  S\,E_p+S^{r-p+1}\,\varepsilon_p^{\beta}
    +S\,\varepsilon_p^{\mathrm{mult}}
  \;=\;
  S^{r-p+1}\sum_{j=p}^{r}\varepsilon_j^{\beta}
  +\sum_{j=p}^{r}S^{j-p+1}\varepsilon_j^{\mathrm{mult}}
  \;=\;E_{p-1},
\]
so $e_{p-1}\le E_{p-1}$. At $p=0$, the induction gives
\[
  e_0
  \;\le\;E_0
  \;=\;S^r\sum_{j=1}^{r}\varepsilon_j^{\beta}
  +\sum_{j=1}^{r}S^j\varepsilon_j^{\mathrm{mult}}
  \;=\;S^r\sum_{j=1}^{r}\varepsilon_j^{\beta}
  +\sum_{j=1}^{r-1}S^j\varepsilon_j^{\mathrm{mult}},
\]
using $\varepsilon_r^{\mathrm{mult}}=0$.  By
\Cref{prop:elim_correct}, $h_0=M_\sm$ and
$\widehat h_0=\widehat M_\sm$, which
gives~\eqref{eq:err-closed}.

\noindent
When uniform bounds are available on $\varepsilon_j^{\beta}$ and $\varepsilon^{\mathrm{mult}}_j$, we have 
$$S^r\sum_{j=1}^r \varepsilon_j^{\beta}+\sum_{j=1}^{r-1}S^j\varepsilon_j^{\mathrm{mult}}\le rS^r\varepsilon^{\beta}+\varepsilon^{\mathrm{mult}}\frac{S^r-S}{S-1}.$$
Substituting the above into
\eqref{eq:err-closed} and using $S^r-S\le 2S^{r-1}(S-1)$ for $S\ge 2$ gives the stated uniform bound.
\end{proof}

\subsection{Approximation of $g_t$}
\label{app:approx_g}
\begin{lemma}[Approximation of $\alpha_t$]\label{lem:NN-alpha}
  Let $S\ge2$, $0<\varepsilon<1$, and
  $\alpha_t=\frac{1-e^{-t}}{1+(S-1)e^{-t}}$.  There exists a
  network parameter
  $\btheta_{\alpha}\in\bTheta_{1,1}(\sL,\sW,\sS,\sB)$ with
  \[
    \sL=O\bigl(\log^2(S/\varepsilon)\bigr),
    \qquad
    \sW=O\bigl(\log^3(S/\varepsilon)\bigr),
    \qquad
    \sS=O\bigl(\log^4(S/\varepsilon)\bigr),
    \qquad
    \sB=e^{O(\log^2(S/\varepsilon))},
  \]
  such that
  \[
    \bigl|\sN_\rho(t'\mid\btheta_{\alpha})-\alpha_t\bigr|
    \;\le\;
    \varepsilon+\frac{17\,S^3}{\varepsilon^{2}}\,|t'-t|,
    \qquad\text{for all }t,t'\in\bR_{>0}.
  \]
  Moreover, $\sN_\rho(t'\mid\btheta_{\alpha})\in[0,1]$ for every $t'\in\bR_{>0}$.
\end{lemma}

\begin{proof}
Write $u=e^{-t}$, $y=1-u$, and $v=1+(S-1)u$, so that
$\alpha_t=y\cdot(1/v)$.  For every $t>0$ we have $u\in(0,1)$,
$y\in(0,1)$, and $v\in(1,S)$, hence $1/v\in(1/S,1)$.  Fix
\[
  \varepsilon_1:=\frac{\varepsilon^3}{256\,S^3},
  \qquad
  \varepsilon_2:=\frac{\varepsilon}{4S},
  \qquad
  K:=\Bigl\lceil\log_2\frac{8}{\varepsilon}\Bigr\rceil.
\]

\smallskip\noindent
\emph{Construction.}\;
The network is a composition of five steps. Apply \Cref{lemma: exp} at accuracy $\varepsilon_1$ to the
input $t'>0$, producing $\widehat u$ with
\[
  |\widehat u-u|\;\le\;\varepsilon_1+|t'-t| =: \delta.
\]
One linear layer forms the two branches
$\widehat y=1-\widehat u$ and $\widehat v=1+(S-1)\widehat u$,
with $|\widehat y-y|\le\delta$ and
$|\widehat v-v|\le(S-1)\delta$.

\smallskip\noindent
Apply \Cref{lemma: reciprocal} at accuracy $\varepsilon_2$ to the
input $\widehat v$, producing $\widehat w$ (intended to approximate $1/v$). The true value
satisfies $v\in(1,S)\subseteq[\varepsilon_2,\varepsilon_2^{-1}]$,
since $\varepsilon_2\le1/S$ and $\varepsilon_2^{-1}=4S/\varepsilon\ge S$.
The input $\widehat v$ may fall anywhere in $\bR$, which the
lemma permits.  Hence
\[
  \bigl|\widehat w-\tfrac1v\bigr|
  \;\le\;\varepsilon_2+\frac{|\widehat v-v|}{\varepsilon_2^{2}}
  \;\le\;\frac{\varepsilon}{4}
  +\frac{16\,S^2(S-1)}{\varepsilon^{2}}\,\delta.
\]
The branch $\widehat y$ passes through this step
unchanged (by the standard positive/negative part representation; e.g.~\cite[Section B.1.1]{nakada2020adaptive}).

\smallskip\noindent
Next we project both branches onto $[0,1]$ by $\mathsf P_1$ (see \eqref{eq:proj}). The true values $y$ and $1/v$ lie in $[0,1]$, and projection onto an interval containing the true value does not increase the
distance to it, so the projected branches carry the same error
bounds.  Both projected values lie in $[0,1]$.  Apply
\Cref{lem:mult} with $r=2$ and $\sA=1$. This yields the product network
$\btheta_\times^{(2)}$ which produces a scalar $N$ with
\[
  \bigl|N-\mathsf P_1(\widehat y)\,
    \mathsf P_1(\widehat w)\bigr|\;\le\;2^{-K}.
\]

\smallskip\noindent
The final layer applies $\mathsf P_1$ to $N$.  Thus the network output belongs to $[0,1]$.

\smallskip\noindent
\emph{Error.}\;
Since $|y|\le1$, $|\hat{y}-y|\le\delta$, and $|\mathsf P_1(\widehat w)|\le1$,
\[
  \bigl|y\cdot\tfrac1v-\mathsf P_1(\widehat y)\,
    \mathsf P_1(\widehat w)\bigr|
  \;\le\;
  \bigl|\tfrac1v-\mathsf P_1(\widehat w)\bigr|
  +\bigl|y-\mathsf P_1(\widehat y)\bigr|
  \;\le\;
  \frac{\varepsilon}{4}
  +\Bigl(\frac{16\,S^3}{\varepsilon^{2}}+1\Bigr)\delta
  \;\le\;
  \frac{\varepsilon}{4}
  +\frac{17\,S^3}{\varepsilon^{2}}\,\delta.
\]

\noindent Adding the product error and substituting
$\delta=\varepsilon_1+|t'-t|$, with
\[
  \frac{17S^3}{\varepsilon^2}\,\varepsilon_1
  =
  \frac{17\varepsilon}{256}
  \le
  \frac{\varepsilon}{8},
  \qquad
  2^{-K}\le\frac{\varepsilon}{8},
\]
gives
\[
  |N-\alpha_t|
  \le
  \frac{\varepsilon}{2}
  +
  \frac{17S^3}{\varepsilon^2}|t'-t|.
\]
Since $\alpha_t\in[0,1]$ and $\mathsf P_1$ is the projection onto
$[0,1]$, we have $|\mathsf P_1(N)-\alpha_t|
  \le
  |N-\alpha_t|$. 
Therefore the claimed error bound holds, and the network output
belongs to $[0,1]$.

\smallskip\noindent
\emph{Parameter accounting.}\;
The steps run in sequence, so depths add and the width is the
maximum over steps. Therefore the depth is given by
$$O(\log^2(1/\varepsilon_1))+O(\log^2(1/\varepsilon_2))
+(K+10)+O(1)=O(\log^2(S/\varepsilon)).$$ For width, the reciprocal
step dominates with
$O(\log^3(1/\varepsilon_2))=O(\log^3(S/\varepsilon))$ neurons, while the
exponential step has width $O(\log(S/\varepsilon))$, the
product network, the projection, and the identity steps have width $O(1)$. The sparsity is given by adding up the sparsities from the above networks, i.e., 
$O(\log^2(1/\varepsilon_1))+O(\log^4(1/\varepsilon_2))
+O(\log^2(S/\varepsilon))+O(\wp_{2,1})=O(\log^4(S/\varepsilon))$. Using the fact that
$\wp_{2,1}=O(K)$ from~\eqref{eq:wp}, the conclusion follows. Finally the largest absolute weight magnitude comes from the exponential map approximation step and is given by $e^{O(\log^2(1/\varepsilon_1))}=e^{O(\log^2(S/\varepsilon))}$.
\end{proof}

\begin{lemma}[Coefficients for the level averages]
\label{lem:level_coefficients}
  For $a\in[0,1]$ and $k=0,\ldots,\sD$, let
  \[
    p_k(a):=\binom{\sD}{k}(1-a)^k a^{\sD-k}.
  \]
  Let $0<\delta\le1$ and set
    $K_\delta
    :=
    \left\lceil
      \log_2\frac{\sD(\sD+3)}{\delta}
    \right\rceil$.
  There exists a ReLU network
  $\btheta_{\mathrm{coef}}\in
  \bTheta_{1,\sD+1}(\sL^{\mathrm{coef}},
    \sW^{\mathrm{coef}},\sS^{\mathrm{coef}},
    \sB^{\mathrm{coef}})$ whose outputs
  $\widehat p_0(a),\ldots,\widehat p_\sD(a)$ belong to $[0,1]$
  and satisfy
  \begin{equation}\label{eq:level-coeff-err}
    \sum_{k=0}^{\sD}
    |p_k(a)-\widehat p_k(a)|
    \le
    \delta,
    \qquad a\in[0,1].
  \end{equation}
  The network parameters satisfy
  \[
    \sL^{\mathrm{coef}}
    =O(\sD(K_\delta+12)),
    \qquad
    \sW^{\mathrm{coef}}
    =O(\sD),
 \qquad
    \sS^{\mathrm{coef}}
    =
    O(\sD^2(K_\delta+1)),
    \qquad
    \sB^{\mathrm{coef}}
    =O(1).
  \]
  Moreover, for
  $a,a'\in[0,1]$,
  \begin{equation}\label{eq:level-coeff-stability}
    \sum_{k=0}^{\sD}
    |p_k(a)-\widehat p_k(a')|
    \le
    2\sD|a-a'|+\delta.
  \end{equation}
\end{lemma}

\begin{proof}
\emph{Construction.}\;
Set $p_0^{(0)}=1$ and $p_j^{(0)}=0$ for $j\ne0$.  For
$m=1,\ldots,\sD$, define
\begin{equation}\label{eq:level-coeff-rec}
  p_j^{(m)}(a)
  :=
  a\,p_j^{(m-1)}(a)
  +(1-a)\,p_{j-1}^{(m-1)}(a),
  \qquad j=0,\ldots,m,
\end{equation}
where terms corresponding to indices outside $\{0,\ldots,m-1\}$ are set to $0$. Induction
on $m$ gives
\[
  p_j^{(m)}(a)=\binom{m}{j}(1-a)^j a^{m-j},
\]
so the last level gives the required coefficients.

\noindent Initialize $\hat{p}_0^{(0)}=1$ and $\hat{p}_j^{(0)}=0$ for $j\neq 0$. The network forms $1-a$ exactly.  At level $m$, it applies
\Cref{lem:mult} with two inputs, $\sA=1$, and precision
$K_\delta$ to each of the two products on the right-hand side
of~\eqref{eq:level-coeff-rec}, with $\hat{p}_j^{(m-1)}(\cdot)$, $\hat{p}_{j-1}^{(m-1)}(\cdot)$ replacing $p_j^{(m-1)}$ and $p_{j-1}^{(m-1)}(\cdot)$ respectively. It adds the two approximate
products and applies $\mathsf P_1$.  All $m+1$ coordinates at a
fixed level are computed in parallel; the levels themselves are
computed successively.

\smallskip\noindent
\emph{Error.}\;
Write $\widehat p^{(m)}$ for the vector produced at level $m$ and put
\[
  e_m
  =
  \sum_{j=0}^{m}
  |p_j^{(m)}(a)-\widehat p_j^{(m)}(a)|.
\]
The linear map in~\eqref{eq:level-coeff-rec} does not increase the
$\ell^1$-norm: each input coordinate enters with total coefficient
$a+(1-a)=1$.  Each output coordinate uses two product networks,
each with error at most $2^{-K_\delta}$.  The exact update of a
vector in $[0,1]^{m}$ lies in $[0,1]^{m+1}$, so the final projection
does not increase its error.  Hence
\[
  e_m
  \le
  e_{m-1}+2(m+1)2^{-K_\delta}.
\]
Since $e_0=0$,
\[
  e_\sD
  \le
  \sD(\sD+3)2^{-K_\delta}
  \le
  \delta,
\]
which proves~\eqref{eq:level-coeff-err}.

For stability, applying~\eqref{eq:level-coeff-rec} at $a$ and
$a'$ and differencing the two, we get: 
\begin{align*}
    p_j^{(m)}(a)-p_j^{(m)}(a')&=a(p_j^{(m-1)}(a)-p_j^{(m-1)}(a'))+(1-a)(p_{j-1}^{(m-1)}(a)-p_{j-1}^{(m-1)}(a')) \\ &+(a-a')(p_{j}^{(m-1)}(a')-p_{j-1}^{(m-1)}(a')).
\end{align*}
Therefore by summing over $j$, the following holds:
\[
  \sum_{j=0}^{m}
  |p_j^{(m)}(a)-p_j^{(m)}(a')|
  \le
  \sum_{j=0}^{m-1}
  |p_j^{(m-1)}(a)-p_j^{(m-1)}(a')|
  +2|a-a'|.
\]
Starting from level zero yields
$\sum_j|p_j(a)-p_j(a')|\le2\sD|a-a'|$.  Combining this with
\eqref{eq:level-coeff-err} proves
\eqref{eq:level-coeff-stability}.

\smallskip\noindent
\emph{Parameter accounting.}\;
At level $m$ there are $2(m+1)$ binary product networks, so the
total number of copies is
\[
  2\sum_{m=1}^{\sD}(m+1)=\sD(\sD+3).
\]
By~\eqref{eq:wp}, each copy has depth $O(K_\delta+10)$, width $O(1)$,
sparsity $O(K_\delta+1)$, and magnitude at most $O(1)$.  The exact
sum and projection add at most two layers, two neurons per output,
and $O(1)$ nonzero parameters per output.  Propagating $a$ and
$1-a$ through the successive levels costs
$O(\sD(K_\delta+1))$ additional nonzero parameters.  Adding these
contributions gives the stated depth, width, sparsity, and magnitude
bounds.
\end{proof}

\begin{lemma}[Approximation of $g_t$]\label{lem:main2}
  Let $\psi_{\sm'}$ be as in~\eqref{eq:qMRF} and fix arbitrary $T>0$.
  For $t\in(0,T]$ denote
  $\alpha_t$ and
  $\mu_t$ as in \eqref{eq:score_uniform}.  Let $K$ be a positive
  integer and set
  \[
    \kappa_K
    =
    K+\sD\log_2(2S)+\log_2(e\sD).
  \]
  There exists a weight-sharing ReLU network as in
  \Cref{def:wsnn}, with sharing map $\omega$ and a single
  parameter vector $\bv\in\R^{\Sws}$, taking the input
  $(\x,t)\in[S]^{\sD}\times(0,T]$ and computing an
  approximation $g_t^{\mathrm{NN}}(\x)\in[0,1]$ of $g_t(\x)$,
  such that
  \begin{equation}\label{eq:main-err}
    \sup_{t\in(0,T]}\;\sup_{\x\in[S]^{\sD}}
    |g_t(\x)-g_t^{\mathrm{NN}}(\x)|
    \;\le\;
    C_d\sD^{d}\bigl(S^{\sD-d-1}+1\bigr)2^{-K}.
  \end{equation}
  The network parameters satisfy
\begin{equation}\label{eq:main-sws}
\begin{aligned}
  \Sws
  &\le
  C_d\left[
    \sD^dS^d
    +\sD^{3d+2}\kappa_K
    +\kappa_K^4
  \right],
  \qquad
  \sL\le C_d\sD^{2d+1}\kappa_K^2,
  \\
  \sW
  &\le
  C_d\sD^{2d+1}
  \left[
    S^d+(S+1)^{\sD-d-1}+\kappa_K^3
  \right],
  \qquad
  \sB\le \exp\!\left(C_d\sD^d\kappa_K^2\right),
  \qquad
  \Lw\le\sD^{d}.
\end{aligned}
\end{equation}
  Here and below, $C_d$ depends only on $d$ (not on $T$).  In particular, for
  every $0<\delta\le1$, by choosing 
  \begin{align}\label{eq:parach}
    K\equiv K_{\delta}
    :=
    \left\lceil
      \log_2\!\left(
        \frac{C_d\sD^{d-1}(S^{\sD-d-1}+1)}{\delta}
      \right)
    \right\rceil \;\; \mbox{and} \;\; \kappa\equiv \kappa_K:=C_d\log(e\sD S^\sD/\delta)
  \end{align}
  we get an uniform bound of atmost $\delta$ in \eqref{eq:main-err}.
\end{lemma}

\begin{proof}[Proof of \Cref{lem:main2}]
Throughout this proof, we will use $C_d$ to denote constants depending on $d$ that change from one line to another. Grouping the terms in~\eqref{eq:pinning} by $|\sm|$, write
\[
  g_t(\x)=\sum_{k=0}^{\sD}\mu_t^k\,\alpha_t^{\sD-k}
  \sum_{\substack{\sm\subseteq[\sD]\\|\sm|=k}}A_\sm(\x_\sm).
\]

The coefficients $\mu_t^k\alpha_t^{\sD-k}$ are the only
quantities in this expansion that depend on $t$.  The partial
marginals $A_\sm$ (see \eqref{eq:marg-def}) are built from the $\psi_{\sm'}$-tables alone,
so the construction splits into two subnetworks --- one
depending only on $\x$, producing $\widehat A_\sm(\x_\sm)$
for every set $\sm$ and their level (i.e.,~$\{\sm:|\sm|=k\}$) averages, and one
depending only on $t$, producing the coefficients.  When
$|\sm|\le d$, we construct neural networks to represent
$\widehat A_\sm=A_\sm$ exactly, by using \cref{lem:slices}, part (i).  For $|\sm|\ge d{+}1$, recall
from~\eqref{eq:marg_split} that
\[
  A_\sm(\x_\sm)
  =\underbrace{\prod_{\substack{\sm'\subseteq\sm\\|\sm'|=d}}
    \psi_{\sm'}(\x_{\sm'})}_{\Pi_\sm}
  \;\cdot\;
  \underbrace{\sum_{\z\in[S]^{\sD-k}}\prod_{\substack{\sm'\in\Fd\\
    \sm'\not\subseteq\sm}}
    \psi_{\sm'}(\x_{\sm'\cap\sm},\z_{\sm'\cap\sm^\com})}_{M_\sm(\x_\sm)},
\]
where $k=|\sm|$.  The sum $M_\sm$ is evaluated by the successive
one-coordinate sums in~\eqref{eq:elim_rec},
\[
  h_{p-1}(\z_{V_{p-1}};\x_\sm)
  =\sum_{b\in[S]}
    \alpha_p(b,\z_{V_{p-1}})\cdot
    \beta_p(b,\z_{V_{p-1}};\x_\sm)\cdot
    h_p(\z_{V_{p-1}},b;\x_\sm),
  \qquad h_{\sD-k}\equiv1,
\]
with $h_0=M_\sm$ by \Cref{prop:elim_correct}. With the above notation in mind, we will complete the rest of the proof in six steps. Throughout the proof, $\Sws$ counts the distinct coordinates of $\bv$.
The ordinary ReLU network of \Cref{lem:slices} is embedded using
the identity map on its parameters.  Whenever identical product
networks occur at several sets or assignments, their parameters
are stored once and placed at all of those positions by the sharing
map.  The proof records both the number of computational positions,
which affects width, and the number of distinct stored parameters.

\bigskip\noindent
\emph{(Stage 0) Table evaluation.}\;
Apply \Cref{lem:slices}.  Given $\x\in[S]^{\sD}$, the network
produces the stored marginals $A_\sm(\x_\sm)$ for
$1\le|\sm|\le d$, the values $\psi_{\sm'}(\x_{\sm'})$ for each
$\sm'\in\Fd$, and the partially evaluated values
$\psi_{\sm'}(\x_{\sm'\setminus U},\z_U)$ for
$1\le|U|\le d-1$. Set $\widehat A_\emptyset=A_\emptyset=1$.  The stored marginals
give $\widehat A_\sm=A_\sm$ for every $\sm$ with
$1\le|\sm|\le d$, with zero error.  The values
$\psi_{\sm'}(\x_{\sm'})$ enter the first product in
\eqref{eq:marg_split} at Stage~4, and the partially evaluated
values are the inputs to Stage~1.

\smallskip\noindent
For $\sm=[\sD]$, the target is
$A_{[\sD]}(\x)=\prod_{\sm'\in\Fd}\psi_{\sm'}(\x_{\sm'})$, a product of
$\binom{\sD}{d}$ values in $[0,1]$.  Apply \Cref{lem:mult} with
$r=\binom{\sD}{d}$ and $\sA=1$, and then apply $\mathsf P_1$,
to produce $\widehat A_{[\sD]}\in[0,1]$ with error at most
$2^{-K}$.  Its parameters are included with the Stage~4 product
networks below.

\begin{center}
\renewcommand{\arraystretch}{1.25}\small
\begin{tabular}{@{}l@{\hskip 50pt}r@{}}
  \toprule
  Depth $\sL$ & $\lesssim 1$ \\
  Width $\sW$ & $\lesssim d\sD^dS^d$ \\
  Effective sparsity $\Sws$
    & $\lesssim (2^d+3d+4)\sD^d S^d$ \\
  &
  {\small
  \shortstack[r]{
    ($\psi_{\sm'}$-tables, marginal tables, and evaluation\\
    constants, all with $\omega=\mathrm{id}$)
  }}
  \\
  Magnitude $\sB$ & $\lesssim \max(S,d)$ \\
  Lipschitz $\Lw$ & $\lesssim 1$ \\
  \bottomrule
\end{tabular}
\end{center}
\bigskip\noindent
\emph{(Stage 1) Grouped products.}\;
If $d=1$, this stage is empty.  Suppose $d\ge2$.  For each set $\sm$ with $\sm\le d$, $\hat{A}_m(\cdot)$s have been constructed already in stage 0 and they are simply carried forward. For 
$\sm$ with $d+1\le|\sm|\le\sD-1$, the partially evaluated table
values from Stage~0 are multiplied into approximations
$\widehat\sH_U(\z_U;\x_\sm)\in[0,1]$ by
\Cref{lem:aggregation}.  By~\eqref{eq:aggregate}, each grouped
product contains $\nu_{|U|}=\binom{k}{d-|U|}$ inputs in $[0,1]$.
\Cref{lem:mult} at $\sA=1$ gives the per-entry error
\[
  \varepsilon_\sH
  =\sup_{U,\z_U,\x}|\sH_U-\widehat\sH_U|
  \;\le\;2^{-K}.
\]
For one set $\sm$, the total number of entries is
$n_\sH=\sum_{\ell=1}^{d-1}\binom{\sD-k}{\ell}S^\ell
\le d\sD^{d-1}S^{d-1}$, and
$\nu_{\max}\le\sD^{d-1}$.

\begin{center}
\renewcommand{\arraystretch}{1.25}\small
\begin{tabular}{@{}l@{\hskip 50pt}r@{}}
  \toprule
  Depth $\sL$
    & $\lesssim \lceil\log_2\nu_{\max}\rceil\,
      \Delta_K^{(\nu_{\max})}+1$ \\
  Width $\sW$ & $\lesssim \,\nu_{\max}\,n_\sH$ \\
  Sparsity $\sS$ for one $\sm$
    & $\lesssim n_\sH(\wp_{\nu_{\max},1}+5)$ \\
  & {\small (ordinary ReLU network of \Cref{lem:aggregation})} \\
  Magnitude $\sB$ & $\lesssim 4^{2\nu_{\max}}$ \\
  Lipschitz after shared placement $\Lw$ & $\lesssim 1$ \\
  \bottomrule
\end{tabular}
\end{center}

\bigskip\noindent
\emph{(Stage 2) Approximation of $\beta_p$.}\;
If $d=1$, the product in~\eqref{eq:beta_p} is empty, so
$\widehat\beta_p=\beta_p=1$ exactly and this stage uses no
network.  Suppose $d\ge2$.  At each summation step $p$, the
input-dependent quantity
$\beta_p$ from~\eqref{eq:beta_p} is a product of $b_p$
approximations $\widehat\sH_U\in[0,1]$.  In both ranges below,
\[
  b_{\max}
  \le
  \sum_{j=0}^{d-2}\sD^j
  \le
  d\sD^{d-2}.
\]

\smallskip\noindent
When $1\le\sD-k\le d$, apply \Cref{lem:betaB}.  The total
number of assignments across all $\sD-k$ steps is
$\sum_pS^p\le 2S^d$.  The error at each assignment
satisfies
\[
  |\beta_p-\widehat\beta_p|
  \;\le\;b_{\max}\varepsilon_\sH+2^{-K}
  \;\le\;(d\sD^{d-2}+1)2^{-K}
  \;=:\;\varepsilon^\beta.
\]
The result is an ordinary ReLU network for one set $\sm$.  In the
global weight-sharing network, equal product-network and projection
parameters are placed at every occurrence having the same value of
$p$.

\smallskip\noindent
When $\sD-k\ge d{+}1$, apply the error bound from \Cref{lem:betaC}. The assignments now total
$\sum_pS^p\lesssim S^{\sD-k}$.  The sharing map
$\omega_{\beta C}$ of~\eqref{eq:omega-betaC} replicates each
product network and projection across all $S^p$ assignments at
step $p$.  Steps with one input read the single stored coordinate
$v_{\mathbf 1}$.

\begin{center}
\renewcommand{\arraystretch}{1.25}\small
\begin{tabular}{@{}l@{\hskip 25pt}cc@{}}
  \toprule
  & $1\le\sD-k\le d$
  & $\sD-k\ge d+1$ \\
  \midrule
  Depth $\sL$
    & $\lesssim \lceil\log_2 b_{\max}\rceil\,\Delta_K^{(b_{\max})}+1$
    & same \\
  Width $\sW$
    & $\lesssim b_{\max}\cdot S^d$
    & $\lesssim b_{\max}\cdot S^{\sD-k}$ \\
  Sparsity for one $\sm$
    & $\lesssim (\wp_{b_{\max},1}+5)\cdot S^d$
    & {\small not used as the global count} \\
  Effective sparsity $\Sws$
    & {\small shared globally by $p$}
    & $\lesssim (\sD-k)(\wp_{b_{\max},1}+5)$ \\
  Magnitude $\sB$
    & $\lesssim 4^{2b_{\max}}$
    & same \\
  Lipschitz $\Lw$ & $\lesssim 1$ & $\lesssim 1$ \\
  \bottomrule
\end{tabular}
\end{center}

\bigskip\noindent
\emph{(Stage 3) Recursive summation.}\; As in stage 1, we will only consider the case where 
For each set $\sm$ with $d+1\le|\sm|\le\sD-1$, put
$r=\sD-k$ and run the $r$ successive sums in
\eqref{eq:matrix_form}.  At each step $p$, the summation matrix
$\mathbf{C}_p$ from \Cref{lem:contract} serves as the
weight matrix of a linear layer.  Row $\z_{V_{p-1}}$ holds the
$S$ weights
$\alpha_p(1,\z_{V_{p-1}}),\ldots,\alpha_p(S,\z_{V_{p-1}})$
from~\eqref{eq:alpha_p}.  The layer computes
\[
  \widetilde h_{p-1}(\z_{V_{p-1}};\x_\sm)
  =\sum_{b=1}^{S}
    [\mathbf{C}_p]_{\z_{V_{p-1}},(b,\z_{V_{p-1}})}
    \cdot\widehat u_p(b,\z_{V_{p-1}};\x_\sm).
\]
This summation layer is exact and introduces no approximation
error.  It is followed by the componentwise projection
\[
  \widehat h_{p-1}
  =
  \mathsf P_{S^{r-p+1}}(\widetilde h_{p-1}).
\]
The projection keeps the approximate partial sum in the range of
the exact one and does not increase its error.  The entries of
$\mathbf C_p$ lie in $[0,1]$ by \Cref{lem:contract}(i).  By
\Cref{lem:contract}(ii)--(iv), the sharing
map~\eqref{eq:omega-contract} fills its weight entries from
$a_pS^d$ entries of the $\psi_U$-tables already stored in
$\bv$.  Each coordinate fills $S^{p-d}$ entries when $p>d$.
The Lipschitz constant is $\Lw^{\mathbf{C}}\le a_r\le\sD^{d}$.

\smallskip\noindent
The product $\widehat u_p$ fed into each summation layer is
formed as follows.  At the top step $p=\sD-k$, the partial sum is
$h_{\sD-k}\equiv1$, so we set $\widehat u_{\sD-k}=\widehat\beta_{\sD-k}$.
This is exact, so $\varepsilon_{\sD-k}^{\mathrm{mult}}=0$.  At each
remaining step $p=\sD{-}k{-}1,\ldots,1$, the network computes the
binary product $\widehat u_p\approx\widehat\beta_p\cdot
\widehat h_p$ by \Cref{lem:mult} at $r=2$ and range
$\sA_p=S^{\sD-k-p}$, followed by
$\mathsf P_{S^{\sD-k-p}}$.  When $\sD-k>d$, apply
\Cref{lem:uprod}.  The
sharing map $\omega_p^u$ of~\eqref{eq:omega-uprod} replicates
the binary product network and projection across the $S^p$
assignments.  When $1\le\sD-k\le d$, use the same ordinary ReLU
construction at the at most $\lesssim S^{d-1}$ assignments.  In the global
network, copies with the same range share their parameters.  In
both cases the per-step product error is
$\varepsilon_p^{\mathrm{mult}}=2^{-K}$. Stage~2 gives $\widehat\beta_p\in[0,1]$.  Starting from $\widehat h_{\sD-k}=1$, the projection after each summation gives $\widehat h_p\in[0,S^{\sD-k-p}]$.  Thus every binary product is computed on the stated range of \cref{lem:uprod}. No small-error assumption is needed.

\smallskip\noindent
After $\sD-k$ steps, \Cref{lem:error} with the three-term
decomposition
\[
  \beta_ph_p-\widehat u_p
  =\beta_p(h_p-\widehat h_p)
  +(\beta_p-\widehat\beta_p)\widehat h_p
  +(\widehat\beta_p\widehat h_p-\widehat u_p)
\]
and the closed form~\eqref{eq:err-closed} gives
\[
  \sup_\x|M_\sm-\widehat M_\sm|
  \;\le\;(\sD-k)S^{\sD-k}\varepsilon^\beta
    +2S^{\sD-k-1}2^{-K}.
\]
For $d\ge2$, Stage~2 gives
$\varepsilon^\beta\le(d\sD^{d-2}+1)2^{-K}$; for $d=1$,
$\varepsilon^\beta=0$.  Consequently, for every $d\ge1$,
\begin{equation}\label{eq:M-approx-err}
  \sup_\x|M_\sm-\widehat M_\sm|
  \le
  C_d\sD^{d-1}S^{\sD-k}2^{-K},
  \qquad
  \widehat M_\sm\in[0,S^{\sD-k}].
\end{equation}

\begin{center}
\renewcommand{\arraystretch}{1.25}\small
\begin{tabular}{@{}l@{\hskip 20pt}cc@{}}
  \toprule
  & $1\le r\le d$
  & $r\ge d+1$ \\
  \midrule
  Depth $\sL$ (per step $p$)
    & $\lesssim K+\lceil 4\log_2\sA_p\rceil+14$
    & same \\
  Width $\sW$ for one $\sm$
    & $\lesssim S^d$
    & $\lesssim S^r$ \\
  \multicolumn{3}{@{}l}{Effective sparsity $\Sws$:} \\
  \quad summation matrices
    & $\lesssim\binom{r}{d}S^d\le\sD^dS^d$
    & same \\
  & \multicolumn{2}{c}{\small (entries of the $\psi_U$-tables,
    already counted in Stage~0)} \\
  \quad products before sums
    & $\lesssim (r-1)(\wp_{2,\sA_1}+5)$
    & same \\
  \quad partial-sum projections
    & $\lesssim r$
    & same \\
  Magnitude $\sB$
    & $\lesssim S^{4r}$ & same \\
  Summation-matrix Lipschitz $\Lw^{\mathbf{C}}$
    & $\lesssim \sD^{d-1}$
    & same \\
  Product network Lipschitz $\Lw^{u}$
    & $\lesssim 1$
    & $\lesssim 1$ \\
  \bottomrule
\end{tabular}
\end{center}

\bigskip\noindent
\emph{(Stage 4) Partial marginals and level averages.}\; For each set $\sm$ with $d+1\le|\sm|\le\sD-1$, the first product $\Pi_\sm=\prod_{\substack{\sm'\in\Fd,\, \sm'\subseteq\sm}} \psi_{\sm'}(\x_{\sm'})$ is a product of $\binom{k}{d}$ values from Stage~0, each in $[0,1]$. Apply \Cref{lem:mult} at $\sA=1$ with error $2^{-K}$ and then the projection $\mathsf P_1$ (see \eqref{eq:proj}). Let us call the result $\widehat\Pi_\sm\in[0,1]$. Next apply \Cref{lem:mult} with two inputs, range $S^{\sD-k}$, and error $2^{-K}$ to $\widehat\Pi_\sm\widehat M_\sm$, and project the result onto $[0,1]$.  Call this final value $\widehat A_\sm$.  Using \eqref{eq:M-approx-err},
\[
  |A_\sm-\widehat A_\sm|
  \;\le\;|M_\sm-\widehat M_\sm|
    +S^{\sD-k}2^{-K}
    +2^{-K}.
\]
Together with the exact values for $|\sm|\le d$, the value for $\sm=[\sD]$ constructed at Stage~0, and the worst case $k=d+1$, this gives
\begin{equation}\label{eq:res-err}
  \max_{\sm\subseteq[\sD]}\ \sup_{\x\in[S]^\sD}
  |A_\sm(\x_\sm)-\widehat A_\sm(\x_\sm)|
  \le
  C_d\sD^{d-1}\bigl(S^{\sD-d-1}+1\bigr)2^{-K}
  =:\eta,
\end{equation}
and every $\widehat A_\sm$ belongs to $[0,1]$.

For $k=0,\ldots,\sD$, form the level averages
\[
  \overline A_k(\x)
  :=
  \binom{\sD}{k}^{-1}
  \sum_{|\sm|=k}A_\sm(\x_\sm),
  \qquad
  \widehat{\overline A}_k(\x)
  :=
  \binom{\sD}{k}^{-1}
  \sum_{|\sm|=k}\widehat A_\sm(\x_\sm).
\]
Both averages lie in $[0,1]$, and~\eqref{eq:res-err} gives
\begin{equation}\label{eq:level-average-err}
  |\overline A_k(\x)-\widehat{\overline A}_k(\x)|
  \le
  \eta.
\end{equation}
The averaging layer uses the single coefficient $\binom{\sD}{k}^{-1}$ at every position in level $k$.  Its sharing map therefore needs only one stored coordinate per level, or $\sD+1$ coordinates in total, and is $1$-Lipschitz.

\begin{center}
\renewcommand{\arraystretch}{1.25}\small
\begin{tabular}{@{}l@{\hskip 50pt}r@{}}
  \toprule
  Depth $\sL$ for one $\sm$
    & $\lesssim \lceil\log_2\binom{k}{d}\rceil
      (K+5\binom{k}{d})
      +K+\lceil4(\sD-k)\log_2S\rceil+14$ \\
  Effective sparsity $\Sws$
    & $\lesssim \wp_{\binom{k}{d},1}+\wp_{2,S^{\sD-k}}+10$ \\
  & {\small (one shared parameter set for each value of $k$)} \\
  Magnitude $\sB$
    & $\lesssim\max\{4^{2\binom{k}{d}+2},\,4^6S^{4(\sD-k)}\}$ \\
  Lipschitz $\Lw$ & $\lesssim 1$ \\
  \bottomrule
\end{tabular}
\end{center}

\bigskip\noindent
\emph{(Stage 5) Time-dependent coefficients and final sum.}\; The scalar input $t\in(0,T]$ is the input to the network of \Cref{lem:NN-alpha} at accuracy
\[
  \varepsilon_\alpha=\frac{2^{-K}}{8\sD}.
\]
Evaluating its guarantee at $t'=t$ gives an output $\widehat\alpha$ with $|\widehat\alpha-\alpha_t|\le\varepsilon_\alpha$ for every $t>0$, and $\widehat\alpha\in[0,1]$. For $a\in[0,1]$, let
\[
  p_k(a)=\binom{\sD}{k}(1-a)^k a^{\sD-k}.
\]
Using $\mu_t=1-\alpha_t$, the pinning decomposition becomes
\begin{equation}\label{eq:g-level-average}
  g_t(\x)
  =
  \sum_{k=0}^{\sD}p_k(\alpha_t)\overline A_k(\x).
\end{equation}
The coefficients are nonnegative and sum to one.  Apply \Cref{lem:level_coefficients} to $\widehat\alpha$ with $\delta=2^{-K}/4$.  It produces $\widehat p_0,\ldots,\widehat p_\sD\in[0,1]$ satisfying \begin{equation}\label{eq:coefficient-total-error}
  \sum_{k=0}^{\sD}
  |p_k(\alpha_t)-\widehat p_k|
  \le
  2\sD\varepsilon_\alpha+\frac{2^{-K}}4
  \le
  \frac{2^{-K}}2.
\end{equation}

Set
\[
  K_f
  =
  K+\left\lceil\log_2(4(\sD+1))\right\rceil.
\]
For each $k$, apply \Cref{lem:mult} with two inputs, $\sA=1$, and precision $K_f$ to $\widehat p_k\widehat{\overline A}_k$, sum the $\sD+1$ outputs exactly and apply $\mathsf P_1$. This defines $g_t^{\mathrm{NN}}(\x)\in[0,1]$.

\smallskip\noindent
\emph{Error.}\;
Since projection cannot increase the distance to
$g_t(\x)\in[0,1]$, \eqref{eq:g-level-average} gives
\[
  |g_t(\x)-g_t^{\mathrm{NN}}(\x)|
  \;\le\;
  \sum_{k=0}^{\sD}p_k(\alpha_t)
    |\overline A_k-\widehat{\overline A}_k|
  +\sum_{k=0}^{\sD}
    |p_k(\alpha_t)-\widehat p_k|
    |\widehat{\overline A}_k|
  +(\sD+1)2^{-K_f}.
\]
The first sum is at most $\eta$ by \eqref{eq:level-average-err} and $\sum_kp_k(\alpha_t)=1$. The second is at most $\lesssim 2^{-K}$ by \eqref{eq:coefficient-total-error}, because every $\widehat{\overline A}_k$ belongs to $[0,1]$.  The last term is at most $\lesssim 2^{-K}$.  Combining these bounds with \eqref{eq:res-err} proves~\eqref{eq:main-err}, uniformly over $t\in(0,T]$.

\begin{center}
\begin{tabular}{@{}l@{\hskip 50pt}r@{}}
  \toprule
  Depth $\sL$
    & $\lesssim \{\sD(K+\log\sD)+(K+\log(\sD S))^2\}$ \\
  Width $\sW$
    & $\lesssim \{\sD+(K+\log(\sD S))^3\}$ \\
  Effective sparsity $\Sws$
    & $\lesssim \{\sD^2(K+\log\sD)
      +(K+\log(\sD S))^4\}$ \\
  &
  \begin{tabular}{@{}r@{}}
    {\small (The networks of \Cref{lem:NN-alpha,lem:level_coefficients}}\\
    {\small and one shared final product network)}
  \end{tabular}
  \\
  Magnitude $\sB$
    & $\lesssim \exp\{O(K+\log(\sD S))^2\}$ \\
  Lipschitz $\Lw$ & $\lesssim 1$ \\
  \bottomrule
\end{tabular}
\end{center}
\bigskip\noindent
\emph{Summary of parameter accounting.}\;
Every product network in the construction has at most
\[
  r
  \le
  \max\left\{
    \binom{\sD}{d},\ \sD^{d-1},\ d\sD^{d-2},\ 2
  \right\}
  \lesssim \sD^d
\]
inputs and range $\sA\le S^\sD$.  Its precision parameter is at most $K+C\log\sD$.  Therefore~\eqref{eq:wp} and the definition of $\kappa_K$ give
\begin{equation}\label{eq:wp-unif}
  \wp_{r,\sA}
  \le
  C_d\sD^{3d+1}\kappa_K
\end{equation}
for every product network used above.  The bound also covers the constant number of parameters added by each projection.

\smallskip\noindent
\emph{Depth and width.}\;
The stages run in sequence.  The input counts satisfy $\nu_{\max}\le\sD^{d-1}$, $b_{\max}\le d\sD^{d-2}$ when $d\ge2$, and $\binom{k}{d}\le\sD^d$.  There are at most $\sD$ successive summation steps.  Combining the depth bounds in the six stages and the two time-dependent networks gives
\[
  \sL
  \le
  C_d\sD^{2d+1}\kappa_K^2.
\]
For $d=1$, Stages~1--2 are empty, and the same bound follows from the remaining stages.

\noindent To count simultaneous computational positions, note that
\[
  \sum_{r=0}^{\sD-d-1}\binom{\sD}{r}S^r
  \le
  \binom{\sD}{d+1}(S+1)^{\sD-d-1}
  \le
  \sD^{d+1}(S+1)^{\sD-d-1}.
\]
Indeed, for each set $\sm$ with $|\sm|\ge d+1$, choose a $(d+1)$-element subset of $\sm$. Note that summing over those choices may count a set more than once and yields the middle expression.  Each position uses at most $C_d\sD^d$ neurons at a product layer.  The table-evaluation and time-dependent subnetworks obey the same final bound, so
\[
  \sW
  \le 
  C_d\sD^{2d+1}
  \left[(S+1)^{\sD-d-1}+\kappa_K^3\right].
\]

\smallskip\noindent
The six stages run in sequence.  The parameter vector $\bv$ holds two kinds of coordinates.  Data-dependent entries (the $\psi_{\sm'}$-tables, marginal tables, and evaluation constants) are placed with $\omega=\mathrm{id}$.  Product-network internals are shared across every occurrence having the same number of inputs, range, and precision.  The indices needed to distinguish these networks are $k$, $|U|$, $p$, and $r-p$. Each of the networks have at most $\sD$ possible values. The combinations actually used above give at most $\le C_d\sD$ distinct product-network parameter sets.

\begin{center}
\renewcommand{\arraystretch}{1.3}\small
\begin{tabular}{@{}llr@{}}
  \toprule
  \emph{Stage} & \emph{Source}
    & \emph{Effective sparsity $\Sws$} \\
  \midrule
  0. Table evaluation
    & $\psi$ + marginal tables,
    & $\lesssim (2^d+3d+4)\binom{\sD}{d}S^d$ \\
  & evaluation constants ($\omega=\mathrm{id}$)
    & \\[3pt]
  1. Grouped products
    & shared by $(k,|U|)$
    & $\lesssim d\sD(\wp_{\max\{2,\sD^{d-1}\},1}+5)$ \\[3pt]
  2. $\beta_p$ ($r\le d$)
    & shared by $p$
    & $\lesssim \sD(\wp_{\max\{2,d\sD^{d-2}\},1}+5)$ \\[3pt]
  2. $\beta_p$ ($r>d$)
    & shared by $p$
    & $\lesssim \sD(\wp_{\max\{2,d\sD^{d-2}\},1}+5)$ \\[3pt]
  3. Summation matrices
    & entries of the $\psi_U$-tables
    & $\lesssim \sD^dS^d$ \\[3pt]
  3. Products before sums ($r\le d$)
    & shared by the range $S^{r-p}$
    & $\lesssim \sD(\wp_{2,S^\sD}+5)$ \\[3pt]
  3. Products before sums ($r>d$)
    & shared by the range $S^{r-p}$
    & $\lesssim \sD(\wp_{2,S^\sD}+5)$ \\[3pt]
  3. Partial-sum projections
    & shared by their upper endpoint
    & $\lesssim 5\sD$ \\[3pt]
  4. Partial marginals
    & products shared by $k$ or $r$
    & $\lesssim \sD(\wp_{\sD^d,1}+\wp_{2,S^\sD}+10)$ \\
  \quad level averages
    & one coefficient per level
    & $\lesssim \sD$ \\[3pt]
  5. Coefficients + final sum
    & \Cref{lem:NN-alpha,lem:level_coefficients}
    & $\lesssim \sD^2(K+\log\sD)+\kappa_K^4$ \\[3pt]
  \midrule
  \multicolumn{2}{@{}l}{\emph{Total}}
    & \multicolumn{1}{r@{}}{$\lesssim C_d(\sD^dS^d
      +\sD^{3d+2}\kappa_K+\kappa_K^4)$} \\
  \bottomrule
\end{tabular}
\end{center}

\smallskip\noindent
Every row in the above summary counts distinct parameters across all sets $\sm$, not the number of physical weight positions.  In particular, the ordinary networks in \Cref{lem:aggregation,lem:betaB} retain their stated local sparsity, while the global weight-sharing network stores one copy of equal product-network parameters and places it at all matching positions.  Using~\eqref{eq:wp-unif} in the table gives \eqref{eq:main-sws}.  When $d=1$, the rows for grouped products and $\beta_p$ are absent. The ordinary table-evaluation network from Stage~0 is embedded with the identity map, so all of its nonzero parameters are counted exactly as in \Cref{lem:slices}.  Separately, the summation matrices of~\eqref{eq:omega-contract}, across all sets $\sm$, use at most $\lesssim \binom{\sD}{d}S^d\le\sD^dS^d$ table coordinates.  The first products in Stage~4 take the evaluations produced in Stage~0 as network inputs and therefore require no additional table-coordinate block.  Altogether, these data-dependent contributions remain bounded by $C_d\sD^dS^d$.

\smallskip\noindent
\emph{Magnitude and Lipschitz constant.}\; The largest weight entries come from the exponential subnetwork inside \Cref{lem:NN-alpha}, of size $\exp\{O(K+\log(\sD S))^2\}$ at accuracy $\varepsilon_\alpha$.  The product-network bound in \Cref{lem:mult}, with at most $C_d\sD^d$ inputs and range at most $S^\sD$, is bounded by $\exp(O(\sD^d\kappa_K^2))$.  The summation-matrix entries and all coefficient outputs lie in $[0,1]$, and every projection endpoint is at most $S^\sD$.  Hence $\sB\le\exp(O(\sD^d\kappa_K^2))$.

\noindent The assembled sharing map acts on disjoint blocks of weight positions.  The blocks are identity placements, replication of the product networks, the constant $v_{\mathbf 1}$, and the products of projected table entries of~\eqref{eq:omega-contract}.  Its Lipschitz constant is therefore the maximum of the block constants. By \Cref{lem:contract}(iv), this maximum is at most $\lesssim \sD^{d}$.

\smallskip\noindent
Finally by plugging in the choices for $K$ and $\kappa$ from \eqref{eq:parach} into \eqref{eq:main-err}, the uniform bound on the right hand side of \eqref{eq:parach} simplifies to $\delta$.
\end{proof}

\subsection{Proof of Theorem~\ref{thm:score}}
\label{app:proof_score}

\begin{corollary}[Mass bounds]\label{cor:mass_bounds}
  Let $t>0$ and set $\rho_t := \min\{R,\alpha_t^{-1}\}$.  Then for every   $\x\in[S]^\sD$,
  \begin{equation}\label{eq:mass_bounds}
    \bigl(1+(S-1)\rho_t\bigr)^{-\sD}
    \;\le\; q_t(\x) \;\le\;
    \bigl(1+(S-1)\rho_t^{-1}\bigr)^{-\sD}.
  \end{equation}
  In particular $1/q_t(\x)\le(S\rho_t)^{\sD}\le(SR)^{\sD}$.
\end{corollary}
\begin{proof}
  The transition probabilities in~\eqref{eq:qt0_hamming} are   strictly positive for $t>0$, so $q_t(\x)>0$ for every  $\x\in[S]^\sD$.  For $\mathbf{c}\in[S]^\sD$, write
  \[
    \x\oplus\mathbf{c}
    =
    (x^1\oplus c^1,\dots,x^\sD\oplus c^\sD).
  \]
  For fixed $\x$, the map   $\mathbf{c}\mapsto\x\oplus\mathbf{c}$ is a bijection of   $[S]^\sD$, so
  \begin{equation*}
    q_t(\x)
    =\frac{q_t(\x)}{\sum_{\mathbf{c}\in[S]^\sD}q_t(\x\oplus\mathbf{c})}
    =\frac{1}{\sum_{\mathbf{c}\in[S]^\sD}s_t^{\star}(\x\oplus\mathbf{c},\x)}.
  \end{equation*}
  A shift of $h$ coordinates is a composition of $h$   flips, so \Cref{lem:score_bounds} telescopes to
  \[
    \rho_t^{-h}
    \le
    s_t^{\star}(\x\oplus\mathbf{c},\x)
    \le
    \rho_t^{h}.
  \]
  There are $\binom{\sD}{h}(S-1)^h$ shifts of $h$  coordinates, and the binomial theorem gives
  \begin{equation*}
    \bigl(1+(S-1)\rho_t^{-1}\bigr)^{\sD}
    \;\le\;
    \sum_{\mathbf{c}\in[S]^\sD}s_t^{\star}(\x\oplus\mathbf{c},\x)
    \;\le\;
    \bigl(1+(S-1)\rho_t\bigr)^{\sD}.
  \end{equation*}
  Inverting gives~\eqref{eq:mass_bounds}.  The last claim uses   $1+(S-1)\rho_t\le S\rho_t$, which holds since $\rho_t\ge1$.
\end{proof}

\begin{proof}[Proof of \Cref{thm:score}]
By~\eqref{eq:score_uniform} and~\eqref{eq:g_def},
\[
  s_t^{\star}(\x\oplus_i c,\x)
  \;=\;\frac{g_t(\x\oplus_i c)}{g_t(\x)},
  \qquad i\in[\sD],\ c\in[S-1].
\]
The network computes the numerators and denominator by copies of the network in \Cref{lem:main2}, approximates the reciprocal of the denominator using \Cref{lemma: reciprocal}, and then multiplies.

\smallskip\noindent
\emph{Bounds on $g_t$.}\; A direct computation gives $q_t(\x)=\bigl(\tfrac{1+(S-1)e^{-t}}{S}\bigr)^{\sD}g_t(\x)$, and $\tfrac{1+(S-1)e^{-t}}{S}\le1$.  Hence, by \Cref{cor:mass_bounds},
\[
  g_t(\x)\;\ge\;q_t(\x)   \;\ge\;(SR)^{-\sD}.
\]
Set $m:=(SR)^{-\sD}$.  Also, $g_t(\x)\le1$ since $\alpha_t\le1$ in~\eqref{eq:g_def}.  In particular $1/g_t(\x)\in[1,m^{-1}]$, uniformly over $t>0$ and $\x$. Choose
\[
  \varepsilon_r=\frac{m\varepsilon}{8},
  \qquad
  \varepsilon_g=\frac{m^2\varepsilon^3}{512},
  \qquad
  \varepsilon_\times=\frac{\varepsilon}{4}.
\]

\smallskip\noindent
\emph{Construction.}\;
For each pair $(i,c)$, the shifted input $\x\oplus_i c$ is an exact function of $\x$: coordinate $i$ satisfies $x^i\oplus c=x^i+c-S\,\mathbf 1\{x^i\ge S-c+1\}$, and on integer inputs $\mathbf 1\{x^i\ge S-c+1\} =\sigma\bigl(x^i-(S-c)\bigr)-\sigma\bigl(x^i-(S-c)-1\bigr)$, so one exact ReLU layer produces, in parallel, the point $\x$ and the $(S-1)\sD$ shifted points $\x\oplus_i c$, using two neurons for each shifted coordinate and one for each copied one.  On each of these $(S-1)\sD+1$ points, together with the input $t$, the network runs one copy of the network of \Cref{lem:main2} at accuracy $\varepsilon_g$.  The copies are identical, so the assembled sharing map replicates the coordinates of \Cref{lem:main2} across all of them, and their effective sparsity is counted once.  Each output satisfies the error bound of \Cref{lem:main2}, since that bound is uniform over the input domain: the copy at $\x\oplus_i c$ produces $g_t^{\mathrm{NN}}(\x\oplus_i c)$ with error at most $\varepsilon_g$, and the copy at $\x$ produces $g_t^{\mathrm{NN}}(\x)$.

\smallskip\noindent
The denominator branch applies \Cref{lemma: reciprocal} at accuracy $\varepsilon_r$ to the input $g_t^{\mathrm{NN}}(\x)$.  The true value satisfies $g_t(\x)\in[m,1]
\subseteq[\varepsilon_r,\varepsilon_r^{-1}]$, since $\varepsilon_r\le m$. Note that the input $g_t^{\mathrm{NN}}(\x)$ may fall anywhere in $\R$, which the lemma permits.  The output is projected by $\mathsf P_{[1,m^{-1}]}$; the true value $1/g_t(\x)$ lies in this interval, so the projection does not increase the error.  Call the result $\widehat w$, so that
\[
  \bigl|\widehat w-\tfrac{1}{g_t(\x)}\bigr|
  \;\le\;\varepsilon_r
  +\frac{\varepsilon_g}{\varepsilon_r^{2}}.
\]

\smallskip\noindent
For each pair $(i,c)$, one copy of the product network $\btheta_\times^{(2)}$ of \Cref{lem:mult} at $r=2$ and range $\sA=m^{-1}$ multiplies $g_t^{\mathrm{NN}}(\x\oplus_i c)$ and $\widehat w$. Observe that both inputs lie in $[-\sA,\sA]$ since $g_t^{\mathrm{NN}}(\x\oplus_i c)\in[0,1]$ and $\widehat w\in[1,m^{-1}]$.  Choose the product-network precision so that its error is at most $\varepsilon_\times$.  The output is projected by $\mathsf P_{[R^{-1},R]}$; the true score lies in this interval by \Cref{lem:score_bounds} under Assumption~\ref{ass:flip}, so the projection does not increase the error and the outputs are bounded as claimed.

\smallskip\noindent
\emph{Error.}\;
Fix $t$, $\x$, $(i,c)$ and abbreviate $g_f=g_t(\x\oplus_i c)$ and $g=g_t(\x)$.  Then
\[
  \bigl|s_t^{\star}-s_t^{\mathrm{NN}}\bigr|
  \;\le\;
  |g_f|\,\Bigl|\tfrac1g-\widehat w\Bigr|
  \;+\;\widehat w\,\bigl|g_f-g_t^{\mathrm{NN}}(\x\oplus_i c)\bigr|
  \;+\;\varepsilon_\times
  \;\le\;
  \Bigl(\varepsilon_r
    +\frac{\varepsilon_g}{\varepsilon_r^{2}}\Bigr)
  \;+\;m^{-1}\varepsilon_g
  \;+\;\varepsilon_\times,
\]
using $|g_f|\le1$ and $\widehat w\le m^{-1}$.  By the choices of the three internal accuracies,
\[
  \frac{\varepsilon_g}{\varepsilon_r^{2}}
  =\frac{\varepsilon}{8},
  \qquad
  m^{-1}\varepsilon_g
  =\frac{m\varepsilon^3}{512}
  \le\frac{\varepsilon}{512},
  \qquad
  \varepsilon_r\le\frac{\varepsilon}{8}.
\]
Together with $\varepsilon_\times=\varepsilon/4$, these estimates show that the score error is at most $\varepsilon$, uniformly over $t\in(0,T]$, $\x$, $i$, and $c$.

\smallskip\noindent
\emph{Parameters.}\;
Let
\[
  \ell_\varepsilon
  :=
  \log\!\left(
    \frac{e\sD(SR)^{3\sD}}{\varepsilon}
  \right),
\]
as in the theorem statement.  Since $m=(SR)^{-\sD}$,
\[
  \log\!\left(\frac{e\sD S^\sD}{\varepsilon_g}\right)
  =
  \log\!\left(
    \frac{512e\sD S^\sD(SR)^{2\sD}}{\varepsilon^3}
  \right)
  \le
  C\ell_\varepsilon
\]
for some constant $C>0$. Thus all $(S-1)\sD+1$ copies of the network in \Cref{lem:main2} share one parameter vector and contribute
\[
  C_d\left[
    \sD^dS^d+\sD^{3d+2}\ell_\varepsilon
    +\ell_\varepsilon^4
  \right]
\]
distinct coordinates, for a intrinsic dimension $d$ dependent constant $C_d$. The shifted-input layer and the range projections use $O(\sD S)$ further coordinates.  The reciprocal network contributes $O(\ell_\varepsilon^4)$ coordinates because $\log(1/\varepsilon_r)\le C\ell_\varepsilon$.  All final binary product networks share one parameter vector. By~\eqref{eq:wp}, their contribution is $O(\ell_\varepsilon)$.  These terms give the stated bound for $\Sws$.

\noindent Depths add along the longest path, giving
\[
  \sL\le C_d\sD^{2d+1}\ell_\varepsilon^2.
\]
The $(S-1)\sD+1\le\sD S$ copies of the $g_t$ network operate in parallel.  Adding the reciprocal and final product widths gives
\[
  \sW
  \le
  C_d\sD^{2d+2}S
  \left[(S+1)^{\sD-d-1}+\ell_\varepsilon^3\right].
\]
The largest parameter magnitude among the $g_t$, reciprocal, and final product networks is at most
\[
  \sB
  \le
  \exp\!\left(O(\sD^d\ell_\varepsilon^2)\right).
\]
Finally, the new placements are identity maps or direct replications, each with Lipschitz constant one.  Hence the assembled sharing map retains the bound $\Lw\le\sD^{d}$ from \Cref{lem:main2}. The selected values $\varepsilon_g$, $\varepsilon_r$, and $\varepsilon_\times$ give the required uniform error and the parameter bounds above.  The final range projection gives outputs in $[R^{-1},R]$, so the constructed network belongs to $\bTheta_R^{\mathrm{ws}}$.
\end{proof}

\subsection{Proof of Theorem~\ref{thm:bias}}
\label{app:proof_bias}
\begin{proof}[Proof of \Cref{thm:bias}]
  Given $\eta>0$, set
  \[
    \eta
    :=
    \left(
      \frac{\varepsilon}
      {8(1\vee T)\sD R^2\log(2R)}
    \right)^{1/2}.
  \]
  Since $0<\varepsilon\le1$, we have $0<\eta\le1$.   Apply \Cref{thm:score} with uniform score accuracy $\eta$.   This gives a network   $s^{\mathrm{NN}}\in\bTheta_R^{\mathrm{ws}}$ satisfying
  \[
    \left|
      s_t^{\mathrm{NN}}(\x\oplus_i c,\x)
      -
      s_t^\star(\x\oplus_i c,\x)
    \right|
    \le
    \eta
  \]
  for every $t\in(0,T]$, $\x\in[S]^\sD$,
  $i\in[\sD]$, and $c\in[S-1]$.

  Fix one such quadruple and write
  \[
    a=s_t^{\mathrm{NN}}(\x\oplus_i c,\x),
    \qquad
    b=s_t^\star(\x\oplus_i c,\x).
  \]
  Both $a$ and $b$ lie in $[R^{-1},R]$, and hence   $a/b\in[R^{-2},R^2]$. Recall the map $K(u)=u-\log{u}-1$.  For every   $v\in[R^{-2},R^2]$,
  \[
    K(v)
    \le
    8\log(2R)(v-1)^2.
  \]
  Indeed, for $v\ge1$,
  \[
    K(v)
    =
    \int_1^v\frac{w-1}{w}\,\mathrm dw
    \le
    \frac12(v-1)^2.
  \]
  For $1/2\le v<1$,
  \[
    K(v)
    =
    \int_v^1\frac{1-w}{w}\,\mathrm dw
    \le
    (1-v)^2.
  \]
  Finally, for $R^{-2}\le v<1/2$,
  \[
    K(v)\le-\log v\le2\log R
    \le
    8\log R\,(1-v)^2.
  \]

 \noindent Since $b\ge R^{-1}$, $\left(\frac ab-1\right)^2
    =
    \frac{(a-b)^2}{b^2}
    \le
    R^2\eta^2$. Therefore
  \[
    bK(a/b)
    \le
    8R^2\log(2R)\,\eta^2\,b.
  \]
  Substituting this inequality into the population loss and using   \eqref{eq:score-mass},
  \[
    \cL_t(s^{\mathrm{NN}})
    \le
    8\sD R^2\log(2R)\,\eta^2
    =
    \frac{\varepsilon}{1\vee T}
  \]
  for every $t\in(0,T]$.  Consequently,
  \[
    \cL(s^{\mathrm{NN}})
    =
    \int_0^T
    \cL_t(s^{\mathrm{NN}})\,\mathrm dt
    \le
    \frac{T}{1\vee T}\,\varepsilon
    \le
    \varepsilon.
  \]
  It remains to express the network parameters in terms of   $\varepsilon$.  The logarithmic factor in \Cref{thm:score} at   accuracy $\eta$ is
  \[
    \ell_\eta
    =
    \log\left(
      e\sD(SR)^{3\sD}
      \sqrt{
        \frac{
          8(1\vee T)\sD R^2\log(2R)
        }{
          \varepsilon
        }
      }
    \right).
  \]
  Since $R\ge1$, $S\ge2$, $\sD\ge2d+2$, and   $0<\varepsilon\le1$, this is bounded by
  \[
    C_d
    \log\!\left(
      \frac{e\sD(SR)^{3\sD}(1\vee T)}{\varepsilon}
    \right)
    =
    C_d\ell_{\varepsilon,T}.
  \]
  Substituting this estimate for $\ell_\eta$ into the parameter bounds of \Cref{thm:score} gives exactly the stated bounds in terms of $\ell_{\varepsilon,T}$.
\end{proof}
\end{document}